\documentclass[10pt,a4paper]{article}
\usepackage[T2A]{fontenc}
\usepackage{amsmath,amsthm,amsfonts,amssymb}
\usepackage{mathrsfs,bbm,upgreek}
\usepackage[dvipsnames]{xcolor}
\usepackage{graphicx,float,longtable,tikz-cd}
\usepackage{url}
\usepackage[colorlinks=true,linkcolor=blue,citecolor=red,urlcolor=blue]{hyperref}
\usepackage[shortlabels]{enumitem}
\usepackage{scalerel}
\usepackage{mathtools}
\usepackage{etoolbox}
\usepackage{thmtools}

\usepackage{varioref}

\newtheoremstyle{bfnote}%
{}
{}%
{\itshape}
{}%
{\bfseries}
{\vspace{\baselineskip}}
{\newline}
{\thmname{#1}\thmnumber{ #2}\thmnote{\fbr[\big]{#3}}}

\theoremstyle{plain}
\newtheorem{theorem}{Theorem}[section]
\newtheorem{proposition}[theorem]{Proposition}
\newtheorem{corollary}[theorem]{Corollary}
\newtheorem{lemma}{Lemma}[section]

\theoremstyle{definition}
\newtheorem{definition}{Definition}[section]
\newtheorem{notation}{Notation}[section]
\theoremstyle{remark}
\newtheorem{remark}{Remark}[section]

\newcommand{\bbG}{\mathbb{G}}
\newcommand{\bbR}{\mathbb{R}}
\newcommand{\bbC}{\mathbb{C}}
\newcommand{\bbN}{\mathbb{N}}
\newcommand{\bbE}{\mathbb{E}}
\newcommand{\bbP}{\mathbb{P}}
\newcommand{\bbF}{\mathbb{F}}
\newcommand{\bbK}{\mathbb{K}}
\newcommand{\bbX}{\mathbb{X}}

\newcommand{\bfD}{\mathbf{D}}
\newcommand{\bfs}{\mathbf{s}}
\newcommand{\bfr}{\mathbf{r}}

\newcommand{\boldT}{\boldsymbol{T}}
\newcommand{\boldpi}{\boldsymbol{\pi}}

\newcommand{\calA}{\mathcal{A}}
\newcommand{\calB}{\mathcal{B}}
\newcommand{\calD}{\mathcal{D}}
\newcommand{\calE}{\mathcal{E}}
\newcommand{\calH}{\mathcal{H}}
\newcommand{\calR}{\mathcal{R}}
\newcommand{\calL}{\mathcal{L}}
\newcommand{\calX}{\mathcal{X}}
\newcommand{\calZ}{\mathcal{Z}}

\newcommand{\frakA}{\mathfrak{A}}
\newcommand{\frakB}{\mathfrak{B}}
\newcommand{\frakC}{\mathfrak{C}}
\newcommand{\frakm}{\mathfrak{m}}
\newcommand{\frakM}{\mathfrak{M}}
\newcommand{\frakN}{\mathfrak{N}}

\newcommand{\frakP}{\mathfrak{P}}
\newcommand{\frakS}{\mathfrak{S}}
\newcommand{\frakZ}{\mathfrak{Z}}

\newcommand{\rmA}{\mathrm{A}}
\newcommand{\rmB}{\mathrm{B}}
\newcommand{\rmD}{\mathrm{D}}
\newcommand{\rmd}{\mathrm{d}}
\newcommand{\rmQ}{\mathrm{Q}}

\newcommand{\scrP}{\mathscr{P}}
\newcommand{\scrI}{\mathscr{I}}

\newcommand{\ttB}{\mathtt{B}}
\newcommand{\ttC}{\mathtt{C}}

\providecommand\given{}

\DeclarePairedDelimiter\fbr{(}{)}
\DeclarePairedDelimiter\seq{(}{)}
\DeclarePairedDelimiter\sbr{\{}{\}}
\DeclarePairedDelimiter\tbr{[}{]}
\DeclarePairedDelimiterX\abs[1]\lvert\rvert{\ifblank{#1}{\:\cdot\:}{#1}}
\newcommand{\boldsemicolon}{{\,\boldsymbol{;}\,}}
\DeclarePairedDelimiterX\kernel[2]{(}{)}{\,#1\boldsemicolon #2\,}
\DeclarePairedDelimiterX\Set[1]\{\}{\renewcommand\given{\SetSymbol[\delimsize]}#1}
\DeclarePairedDelimiterXPP\asvector[1]{\boldsymbol{\pi}^{-1}}(){}{#1}
\DeclarePairedDelimiterXPP\aspc[1]{\boldpi}(){}{#1}
\DeclarePairedDelimiterXPP\Borel[1]{\calB}(){}{#1}
\DeclarePairedDelimiterXPP\cS[1]{\scrP}(){}{#1}
\DeclarePairedDelimiterXPP\Prob[1]{\bbP}(){}{\renewcommand\given{\nonscript\:\delimsize\vert\nonscript\:\mathopen{}}#1}

\newcommand{\om}{\Omega_{\infty}^m}
\newcommand{\omn}{\Omega_n^m}
\newcommand{\omnd}{\Omega_{\nd}^m}
\newcommand{\omnk}{\Omega_{n_k}^m}
\newcommand{\oj}{\Omega(j)}
\newcommand{\onkj}{\Omega_{n_k}(j)}
\newcommand{\onlj}{\Omega_{n_l}(j)}
\newcommand{\onj}{\Omega_n(j)}

\newcommand{\ocj}{\Omega^{\mathrm{corr}}(j)}
\newcommand{\goodomega}{\Omega_{\nd}^{\mathrm{good}}}
\newcommand{\OMG}[5]{\Omega_{#2}^{#3}\big[#5\boldsemicolon #1\boldsemicolon #4\big]}
\newcommand{\OGm}[2]{\Uptheta^m_{{#1}}[#2]}
\newcommand{\eneg}[1]{\calE_\delta^{#1}}
\newcommand{\pneg}[1]{\Prob{\calE_\delta^{#1}}}
\newcommand{\cneg}[1]{C_{#1}}
\newcommand{\ej}{\calE(j)}
\newcommand{\bT}[1]{\boldT_{#1}} 
\newcommand{\bTp}[1]{\boldT^\prime_{#1}} 
\newcommand{\Zhe}{\mbox{\usefont{T2A}{\rmdefault}{m}{n}\CYRZH}}
\newcommand{\laplacian}{\Delta}
\newcommand{\symmdiff}{\triangle}

\newcommand{\lev}{\rigidity}
\newcommand{\thetagap}{\theta}
\newcommand{\ginisumone}{S_1}
\newcommand{\ginisumtwo}{S_2}
\newcommand{\ginisumthree}{\widetilde{S}_3}
\newcommand{\prsp}{\Zhe}

\newcommand{\ndmij}{\nd-m+(i\wedge j)}
\newcommand{\mij}{m-(i\wedge j)}
\def\s{\sigma}
\newcommand{\z}{\zeta}
\newcommand{\uz}{\underline{\zeta}}
\newcommand{\us}{\underline{s}}
\newcommand{\uo}{\underline{\omega}}
\newcommand{\uw}{\underline{\mathrm{w}}}
\newcommand{\uzp}{\underline{\zeta}^\prime}

\newcommand{\zz}[1]{\frakZ_{#1}[\delta]}
\newcommand{\zp}[1]{\frakZ^\prime_{#1}[\delta]}
\newcommand{\zpp}[1]{\frakZ^{\prime\prime}_{#1}[\delta]}
\newcommand{\zzl}{\zz{l}}
\newcommand{\zpl}{\zp{l}}
\newcommand{\zppl}{\zpp{l}}
\newcommand{\zzv}{\underline{\frakZ}[\delta]}
\newcommand{\zpv}{\underline{\frakZ}^\prime[\delta]}
\newcommand{\zppv}{\underline{\frakZ}^{\prime\prime}[\delta]}
\newcommand{\nka}{\fbr*{\binom{n}{k}k!}^{\alpha/2}}

\newcommand{\ndckat}{\fbr*{\binom{\nd}{k}k!}^{\alpha/2}}
\newcommand{\ncka}{\fbr*{\binom{n}{k}k!}^{\alpha}}
\newcommand{\nckat}{\fbr*{\binom{n}{k}k!}^{\alpha/2}}
\newcommand{\RR}{\bbR}
\newcommand{\CC}{\bbC}
\newcommand{\NN}{\bbN}
\newcommand{\nat}{\NN\cup\{0\}}
\newcommand{\edel}{\epsilon(\delta)}
\newcommand{\Cdel}{C\fbr*{\delta}} 
\newcommand{\kdel}{k\fbr*{\delta}}
\newcommand{\Rdel}{R\fbr*{\delta}}
\newcommand{\gdel}{g\fbr*{\delta}}
\newcommand{\Ld}{L\fbr*{\delta}}
\newcommand{\hLd}{h\fbr*{\Ld}}
\newcommand{\nd}{n_\delta}
\newcommand{\ndk}{n_{\delta_k}}
\newcommand{\fta}{\left\lfloor\frac{2}{\alpha}\right\rfloor}
\newcommand{\foa}{\left\lfloor\frac{1}{\alpha}\right\rfloor}
\newcommand{\moment}{\bfs_\alpha}
\newcommand{\sqmoment}{\bfs^2_\alpha}
\newcommand{\deri}{{\;}\rmd}
\newcommand{\Exp}{\bbE}
\newcommand{\PROB}{\bbP}
\newcommand{\el}{\calL}
\newcommand{\leb}{\calL_{\Sigma_{\usm}}}
\newcommand{\mm}{\frakm}
\newcommand{\MM}{\frakM}
\newcommand{\mD}{\calD}

\newcommand{\cc}{{\,,\,}}
\newcommand{\g}{w}
\newcommand{\pss}{\mathrm{S}}
\newcommand{\pssg}{\mathrm{S}_{\mathrm{good}}\big(\uond\big)}
\newcommand{\X}{\boldsymbol{X}}
\newcommand{\BasisInsidempts}{\frakA^m_\inn}
\newcommand{\FUBasisInsidempts}{\widehat{\frakA}^m_\inn}
\newcommand{\BorelInsidempts}{\overline{\frakA}^m_\inn}
\newcommand{\BasisOutside}{\frakB_\out}
\newcommand{\FUBasisOutside}{\widehat{\frakB}_\out}
\newcommand{\BorelOutside}{\overline{\frakB}_\out}
\newcommand{\Dset}{\mathcal{D}}
\newcommand{\constraintinternal}{\frakC_{\alpha,\inn}}
\newcommand{\constraintinternalpc}{\frakC^\ast_{\alpha,\inn}}
\newcommand{\constraintexternal}{\frakC_{\alpha,\out}}
\newcommand{\extsum}{\frakC_{\out}}
\newcommand{\numberofpoints}{\frakN}
\newcommand{\usm}{{m,\us}}
\newcommand{\InProdSqNd}{\abs*{\Pi\tbr[\big]{\fand\boldsemicolon\inn}}^2}
\newcommand{\InProdNd}{\Pi\tbr[\big]{\fand\boldsemicolon\inn}}
\newcommand{\InProdSq}{\abs*{\Pi\tbr[\big]{\falphagaf\boldsemicolon\inn}}^2}
\newcommand{\InProd}{\Pi\tbr[\big]{\falphagaf\boldsemicolon\inn}}
\newcommand{\uond}{\uo^{[\nd]}}
\newcommand{\usnd}{\us^{[\nd]}}
\newcommand{\usnk}{\us^{[n_k]}}
\newcommand{\tn}[1]{\langle #1\rangle}
\newcommand{\okbe}{o_k\big(\,1\boldsemicolon\epsilon\,\big)}
\newcommand{\oep}{o_\epsilon(1)}
\newcommand{\ok}{o_k(1)}

\newcommand{\xo}{\calX}
\newcommand{\QFZop}[0]{\rmQ_{\mathrm{tail}}^{\delta,\mathrm{avg}}}
\newcommand{\QFIop}[0]{\rmQ_{i,j}^{\delta}}
\newcommand{\QFZ}[1]{\QFZop\Big(#1\Big)}
\newcommand{\QFI}[1]{\QFIop\Big(#1\Big)}
\newcommand{\linf}[1]{\lvert #1\rvert_{\infty}}
\newcommand{\closure}[1]{\overline{#1}}
\newcommand{\van}[1]{\triangle\big(\,#1\,\big)}
\newcommand{\indicatorone}{\mathbbm{1}}
\newcommand{\indicator}[1]{\indicatorone\tbr[\Big]{\,#1\,}}
\newcommand{\Bigindicator}[1]{\indicatorone\tbr[\Big]{\,#1\,}}
\newcommand{\bigindicator}[1]{\indicatorone\tbr[\big]{\,#1\,}}
\newcommand{\vancross}[2]{\Gamma\big(\,#1\boldsemicolon #2\,\big)}
\newcommand{\ffa}[2]{\pi_{\alpha}[#1\boldsemicolon #2]}
\newcommand{\h}[2]{\calR\big(#1\boldsemicolon #2\big)}
\newcommand{\sym}[2]{\sigma_{#1}\big(\,#2\,\big)}
\newcommand{\fDij}[2]{\rmD_{i,j}\big(\,#1\,,\,#2\,\big)}
\newcommand{\fD}[2]{\rmD\big(\,#1\,,\,#2\,\big)}
\newcommand{\fDhat}[2]{\widehat{\rmD}\big(\,#1\,,\,#2\,\big)}
\newcommand{\tail}[2]{\eta_{#1}^{[#2]}}
\newcommand{\gammacnd}[2]{\gamma_{\nd}\big(#1\boldsemicolon #2\big)}
\newcommand{\frho}[2]{\rho\big(\,#1\boldsemicolon #2\,\big)}

\newcommand{\drho}[3]{\uprho_{#1}^{#2}\big(\,#3\,\big)}
\newcommand{\FB}[3]{\rmB\Big(\,#1\,,\,#2\,,\,#3\,\Big)}
\newcommand{\FA}[3]{\rmA\Big(\,#1\,,\,#2\,,\,#3\,\Big)}
\newcommand{\mucnd}[3]{\mu_{\nd}\big(#1\boldsemicolon #2\boldsemicolon #3\big)}
\newcommand{\intcnd}[3]{\mathrm{I}^{[#3]}_{\nd}\big(#1\boldsemicolon #2\big)}
\newcommand{\vecsym}[3]{\frakS\big[#1\boldsemicolon #2\boldsemicolon #3\big]}
\newcommand{\vecsymb}[3]{\overline{\frakS}\big[#1\boldsemicolon #2\boldsemicolon #3\big]}
\newcommand{\vecsymn}[3]{\Big\|\frakS\big[#1\boldsemicolon #2\boldsemicolon #3\big]\Big\|_2^2}
\newcommand{\absinv}[4]{\bigg|\scrI\Big[#1\boldsemicolon #2\boldsemicolon\big(#3,#4\big)\Big]\bigg|}
\newcommand{\inv}[4]{\scrI\Big[#1\boldsemicolon #2\boldsemicolon\big(#3,#4\big)\Big]}
\newcommand{\invabs}[4]{\scrI_{|\cdot|}\Big[#1\boldsemicolon #2\boldsemicolon\big(#3,#4\big)\Big]}
\newcommand{\inn}{\mathrm{in}}
\newcommand{\out}{\mathrm{out}}
\newcommand{\fgaf}{\bbF_{\infty}}
\newcommand{\gaf}{\calZ_{\infty}}

\newcommand{\fgafn}{\bbF_{n}}
\newcommand{\gafn}{\calZ_n}

\newcommand{\falphagaf}{\bbF_{\alpha,\infty}}
\newcommand{\alphagaf}{\calZ_{\alpha,\infty}}
\newcommand{\alphagafout}{\calZ_{\alpha,\infty,\out}}
\newcommand{\alphagafin}{\calZ_{\alpha,\infty,\inn}}
\newcommand{\falphagafn}{\bbF_{\alpha,n}}
\newcommand{\alphagafn}{\calZ_{\alpha,n}}
\newcommand{\alphagafnout}{\calZ_{\alpha,n,\out}}
\newcommand{\alphagafnin}{\calZ_{\alpha,n,\inn}}
\newcommand{\alphagafnkout}{\calZ_{\alpha,n_k,\out}}
\newcommand{\alphagafndout}{\calZ_{\alpha,n_\delta,\out}}
\newcommand{\alphagafndin}{\calZ_{\alpha,n_\delta,\inn}}

\newcommand{\Gini}{\bbG}
\newcommand{\Ginin}{\bbG_{n}}
\newcommand{\Gininin}{\bbG_{n,\inn}}
\newcommand{\Gininout}{\bbG_{n,\out}}
\newcommand{\Gininkin}{\bbG_{n_k,\inn}}
\newcommand{\Gininkout}{\bbG_{n_k,\out}}
\newcommand{\GIF}{\bbG_{\infty}}
\newcommand{\GIFOUT}{\bbG_{\infty,\out}}
\newcommand{\GIFIN}{\bbG_{\infty,\inn}}
\newcommand{\fand}{\bbF_{\alpha,\nd}} 

\newcommand{\XX}{\bbX_{\infty}}
\newcommand{\xin}{\bbX_{\infty,\inn}}
\newcommand{\xout}{\bbX_{\infty,\out}}
\newcommand{\xn}{\bbX_{n}}
\newcommand{\xnin}{\bbX_{n,\inn}}
\newcommand{\xnout}{\bbX_{n,\out}}

\newcommand{\xinnk}{\bbX_{n_k,\inn}}
\newcommand{\xoutnk}{\bbX_{n_k,\out}}

\newcommand{\zinnk}{\calZ_{\alpha,n_k,\inn}}
\newcommand{\zoutnk}{\calZ_{\alpha,n_k,\out}}
\newcommand{\zin}{\calZ_{\alpha,\infty,\inn}}
\newcommand{\zout}{\calZ_{\alpha,\infty,\out}}
\newcommand{\znd}{\calZ_{\alpha,\nd}}
\newcommand{\zinnd}{\calZ_{\alpha,\nd,\inn}}
\newcommand{\zoutnd}{\calZ_{\alpha,\nd,\out}}

\def\c{\complement}

\newcommand{\qgB}{\ttB}
\newcommand{\qgC}{\ttC}
\newcommand{\Poisson}{\frakP}
\newcommand{\rigidity}{\bfr_\alpha}
\newcommand{\vm}[1]{\boldsymbol{\mathbf{#1}}}
\newcommand{\permutation}{\pi}
\newcommand{\correctionterm}[2]{\vartheta\fbr*{#1,#2}}

\newcommand{\totalsumnum}{\big\langle\vecsymb{\nd}{i}{(\uz,\uo)},\vecsym{\nd}{j}{(\uz,\uo)}\big\rangle}
\newcommand{\totalsumdenom}{\vecsymn{\nd}{0}{(\uz,\uo)}}
\newcommand{\frontsum}{\big\langle\vecsym{\nd}{i}{\uo}\odot\mathbbm{1}_{[0:\nd-\Cdel)},\vecsym{\nd}{j}{\uo}\odot\mathbbm{1}_{[0:\nd-\Cdel)}\big\rangle}
\newcommand{\tailsum}{\big\langle\vecsym{\nd}{i}{\uo}\odot\mathbbm{1}_{[\nd-\Cdel,\nd]},\vecsym{\nd}{j}{\uo}\odot\mathbbm{1}_{[\nd-\Cdel,\nd]}\big\rangle}
\newcommand\numberthis{\addtocounter{equation}{1}\tag{\theequation}}

\newcounter{kdm} 
\newcommand{\newkdm}[1]{\refstepcounter{kdm}\label{#1}} 
\newcommand{\usekdm}[1]{K_{\scaleto{\ref{#1}}{5.5pt}}(\Dset,m)}
\newcounter{C}
\newcommand{\newc}[1]{\refstepcounter{C}\label{#1}}
\newcommand{\usec}[1]{C_{\scaleto{\ref{#1}}{5.5pt}}}
\newcounter{kd} 
\newcommand{\newkd}[1]{\refstepcounter{kd}\label{#1}} 
\newcommand{\usekd}[1]{K_{\scaleto{\ref{#1}}{5.5pt}}(\Dset)}

\allowdisplaybreaks

\makeatletter
\newcommand\appendix@section[1]{%
  \refstepcounter{section}%
  \orig@section*{Appendix \@Alph\c@section: #1}%
  \addcontentsline{toc}{section}{Appendix \@Alph\c@section: #1}%
}
\let\orig@section\section
\g@addto@macro\appendix{\let\section\appendix@section}
\makeatother

\title{Approximate Gibbsian structure in strongly correlated point fields 
and generalized Gaussian zero ensembles}
\author{{Ujan Gangopadhyay}\thanks{Dept. of Mathematics, National University of Singapore, \texttt{ujan@nus.edu.sg}}
\and{Subhroshekhar Ghosh}\thanks{Dept. of Mathematics, National University of Singapore, \texttt{subhrowork@gmail.com}}
\and{Kin Aun Tan}\thanks{Dept. of Mathematics, National University of Singapore, \texttt{e0196827@u.nus.edu}}}
\date{}

\begin{document}

\maketitle

\begin{abstract}
Gibbsian structure in random point fields has been a classical tool for studying their spatial properties. However, exact Gibbs property is available only in a relatively limited class of models, and it does not adequately address many random fields with a strongly dependent spatial structure. In this work, we provide a very general framework for approximate Gibbsian structure for strongly correlated random point fields, including those with a highly singular spatial structure. These include processes that exhibit strong spatial \textit{rigidity}, in particular, a certain one-parameter family of analytic Gaussian zero point fields, namely the $\alpha$-GAFs, that are known to demonstrate a wide range of such spatial behaviour. Our framework entails conditions that may be verified via finite particle approximations to the process, a phenomenon that we call an approximate Gibbs property. We show that these enable one to compare the spatial conditional measures in the infinite volume limit with Gibbs-type densities supported on appropriate singular manifolds, a phenomenon we refer to as a generalized Gibbs property. Our work provides a general mechanism to rigorously understand the limiting behaviour of spatial conditioning in strongly correlated point processes with growing system size. We demonstrate the scope and versatility of our approach by showing that a  generalized Gibbs property holds with a logarithmic pair potential for the $\alpha$-GAFs for any value of $\alpha$. In this vein, we settle in the affirmative an open question regarding the existence of point processes with any specified level of rigidity. In particular, for the $\alpha$-GAF zero process, we establish the \textit{level of rigidity} to be exactly $\lfloor \frac{1}{\alpha} \rfloor$, a fortiori demonstrating the phenomenon of spatial \textit{tolerance} subject to the local conservation of $\lfloor \frac{1}{\alpha} \rfloor$ moments. For such processes involving complex, many-body interactions, our results imply that the local behaviour of the random points still exhibits 2D Coulomb-type repulsion in the short range. Our techniques can be leveraged to estimate the relative energies of configurations under local perturbations, with possible implications for dynamics and stochastic geometry on strongly correlated random point fields.
\end{abstract}

\textbf{Keywords}: Gibbs property; Quasi-Gibbs property; Random matrix; Random polynomials; Thermodynamics; Equilibrium statistical mechanics; Rigidity phenomena; Gaussian analytic functions;  Interacting particle systems; Stochastic geometry.

\tableofcontents


\newkdm{ratio:sym:agaf}
\newkdm{471}
\newkdm{472}
\newkdm{473}
\newkdm{571}
\newkdm{pf:p:sym}
\newkdm{pfrec1}
\newkdm{pfrec2}
\newkdm{pfrec3}
\newkd{pf:p:2:5:1}
\newkdm{pf:p:2:5:2}
\newkdm{pf:p:2:5:3}
\newkdm{pf:p:2:5:4}
\newkdm{pf:p:2:5:5}
\newkd{gst}
\newkd{p45}
\newkd{pf:p:van}
\newkd{pgcr}
\newkd{pginil1} 
\newc{CPhi}
\newc{p42}
\newc{p43}
\newc{p44} 
\newc{pf:p:inv:1:1}
\newc{pf:p:inv:1:2}
\newc{pf:p:inv:1:3}
\newc{pf:p:inv:1:4}
\newc{pf:p:inv:1:5}
\newc{pf:p:inv:1:6}
\newc{pf:p:inv:1:7}
\newc{pf:p:inv:1:8}
\newc{pf:p:inv:2:1}
\newc{pf:p:inv:2:2}
\newc{pf:p:inv:3:1}
\newc{pf:p:inv:3:2}

\section{Introduction}

\subsection{Random point fields}

A random point field $\Pi$, also known as a point process, on a locally compact second countable Hausdorff space $\Xi$ is a random locally finite point configuration on the space $\Xi$. In other words, a point process $\Pi$ is a random variable taking values in the space of locally finite point configurations on the ambient space $\Xi$. Random point fields are objects of fundamental interest in a wide range of areas in pure and applied mathematics, including but not limited to probability theory, statistical physics, spatial statistics, network science and stochastic geometry; for a comprehensive treatment we refer the reader to \cite{DV-1,DV-2}.

The most basic model of randomness in stochastic systems is perhaps that of independent random variables. In the world of random point fields, the canonical model of statistical independence is the Poisson point process, which is characterized by point counts being independent across disjoint domains in the ambient space. The model of statistically independent randomness has led to a vast body of literature spanning several decades. However, some of the most interesting large-scale stochastic phenomena turn out to be a result of the collaborative behavior of interacting particle systems, which makes statistically independent models limited in their scope.

\subsection{Local conditioning, Gibbs and quasi-Gibbs properties}\label{ss:localconditioning}

Incorporating spatial interactions poses significant mathematical challenges in terms of the tractability of the models. A classical concept that endeavors to locate a tractable structure in spatially dependent models is the so-called \textit{Gibbs property}. In the simple setting of a finite point configuration, the Gibbs property entails that the likelihood of a point set $\s$ (for instance, its probability density with respect to an appropriate background measure) takes the form of $\exp(-\beta \calH[\s])$ for a suitable energy functional $\calH[\s]$ and an inverse temperature parameter $\beta$. Concretely, the finite volume Gibbs measure on a bounded domain $\Dset\subset\RR^d$ is given by (\cite{Dereudre2019IntroductionProcesses,Dereudre2020ExistenceInteraction})
\begin{equation}\label{eq:Gibbs_def}
    \deri\PROB_\Dset(\cdot) = 
    \frac{1}{Z_\Dset} 
    \exp\left(-\beta\calH_\Dset[\cdot]\right)
    \cdot\deri\mathfrak{P}_\Dset(\cdot)\;,
\end{equation}
where $\mathfrak{P}_\Dset$ is the homogeneous Poisson point process on $\Dset$, 
$Z_\Dset$ is a normalizing constant, 
and $\beta>0$ is the inverse temperature. 
A significant class of such energy functionals are characterized by so-called \textit{pairwise interactions}; i.e., 
\[
\calH_{\Dset}[\s] = \sum_{\{x,y\} \subset \s} \Psi(x-y)
\]
for a \textit{potential function} $\Psi:\RR^d\to\RR\cup\Set*{\infty}$ with appropriate decay properties. This class includes, in particular, the well-known Lennard-Jones pair potential from statistical physics. In the present article, we would not have the occasion to dwell further on the elaborate theory of Gibbsian point processes, and instead refer the interested reader to the extensive treatments in the classic references \cite{Geo,Ruelle}, and the excellent survey \cite{Dereudre2019IntroductionProcesses}.

An important aspect of the Gibbsian structure is the description of the local conditional behavior of such processes. This is particularly effective for infinite volume systems, where it is not straightforward to assign a probability density to a given (infinite) configuration of particles, but the same is much simpler for its spatial conditionings, where we look at the conditional law of the  process on a bounded domain $\Dset$ given the configuration outside $\Dset^\c$ (the \textit{environment}). To be more precise, we consider the conditional distribution $\Pi_{\Dset | \Dset^\c}$ of a Gibbsian point process $\Pi$ restricted to a bounded domain $\Dset \subset \Xi$, given the configuration $\Pi$ on $\Dset^\c$. The celebrated \textit{Dobrushin-Landau-Ruelle} (abbrv. DLR) equations entail that such conditional distribution is specified by a yet another Gibbs-type density, in this case assuming the form $\exp(-\beta \calH_\Dset[\s])$ for a so-called \textit{local energy functional} $\calH_\Dset[\cdot]$ applied to the full point configuration $\s$. To provide an idea of how such local energy functionals may be structured, we content ourselves here with the setting of a pairwise interaction potential $\Psi$, where the local energy functional $\calH_\Dset$ decomposes neatly as a sum of two terms, one capturing the mutual interaction of the points of $\s$ inside $\Dset$ (denoted by $\s_{|\Dset}$) and the other comprising of the interactions across the boundary of $\Dset$ i.e., between points of $\s$ inside $\Dset$ and points of $\s$ outside $\Dset$ (denoted by $\s_{|\Dset^\c}$). In particular, in this setting we have 
\[
\calH_\Dset[\s] = \sum_{x,y \in \s_{|\Dset}} \Psi(|x-y|) + \sum_{x\in \s_{|\Dset}; \; z \in \s_{|\Dset^\c}} \Psi(|x-z|)\;.
\]
We notice that this immediately leads to a multiplicative decomposition of the conditional density of $\Pi_{\Dset|\Dset^\c}$
of the form $\propto \cdot \exp(-\beta \calH_1[\s_{|\Dset}]) \cdot \exp(-\beta \calH_2[\s_{|\Dset},\s_{|\Dset^\c}])$. 
In summary, 
    the spatially conditioned Gibbs measure has the form
\begin{equation}\label{eq:Gibbs_cond_def}
\deri\PROB_{\Dset|\Dset^\c}
\left[\s_{|\Dset}\big|\s_{|\Dset^\c}\right] 
= 
\frac{1}{Z_\Dset(\s_{|\Dset^\c})} 
\cdot 
\exp\left(-\beta \calH_1[\s_{|\Dset}]\right) 
\cdot 
\exp\left(-\beta \calH_2[\s_{|\Dset^{\vphantom{\c}}},\s_{|\Dset^\c}]\right)
\cdot 
\deri\mathfrak{P}_\Dset(\s_{|\Dset})\;,
\end{equation}
where $\Poisson_\Dset$ is the homogeneous Poisson point process on $\Dset$, $Z_\Dset(\s_{|\Dset^\c})$ is a normalizing constant (that depends on the environment $\s_{|\Dset^\c}$), and $\beta>0$ is the inverse temperature. It remains to note that if reasonable control on the cross-boundary interaction $\calH_2[\cdot,\cdot]$ can be obtained, the conditional density of $\Pi_{\Dset | \Dset^\c}$ may be bounded from above and below between constant factors of a density $\exp(-\beta \calH_1[\cdot])$ that depends solely on the finite point set $\s_{|\Dset}$ inside $\Dset$, and often has a tractable algebraic form. Such tractable bounds are of great interest, in particular in the study of stochastic dynamics of such particle systems ; c.f. \cite{Os12,Os13-1,Os13-2}. 

Several difficulties beset the implementation of the broad program outlined above. Most of these issues straddle both technical and conceptual aspects of these models,
are related to problems in rigorously formulating Gibbsian concepts for infinite volume systems, and involve delicate questions of existence/stability of such systems and making concrete sense of the local Gibbs structures. Some of the difficulties include, but are not limited to, slow decay of the pair potential $\Psi$ (e.g., logarithmic, as in the case of Coulomb type systems in 2D), and complications that arise when the interactions are not merely pairwise but are of higher order. For a more detailed account, we refer the reader to \cite{Dereudre2019IntroductionProcesses,Lew,Ruelle} and the references therein.

In order to address these difficulties, Osada \cite{Os13-1} introduced the concept of \textit{Quasi-Gibbs measures} for point processes. Roughly speaking, it entails that while an exact local Gibbs structure might not be available in many models, certain consequences of such structure (in particular, inequalities on the local conditional distributions alluded to earlier) might nonetheless suffice to understand important properties of such systems. This is formalized in the notion of quasi-Gibbs measures.

The quasi-Gibbs property relaxes the requirement of the classical Gibbs property (c.f. \eqref{eq:Gibbs_def},\eqref{eq:Gibbs_cond_def}) by positing that we need not have an exact equality for the spatially conditioned density, but in fact we have upper and lower bounds on it in terms of classical Gibbs measures (as in the right hand side of  \eqref{eq:Gibbs_def}), and it would further suffice to have such comparison inequalities only with respect to \textit{slices} of the Poisson process on $\Dset$ that restrict the latter to a fixed number of particles. For a detailed description of the quasi-Gibbs property, we refer the reader to Appendix \ref{a:Osada}; for an even more general account we refer to \cite{Os12} and the expository tract \cite{Os21survey}.

The quasi-Gibbs structure is employed in \cite{Os13-1,Os13-2,Os12,Os21}, among other works, to study interacting stochastic dynamics on infinite particle systems with logarithmic interaction potentials. This approach is successful in understanding the dynamics of a collection of infinitely many interacting Brownian particles with a family of equilibrium measures that include Ginibre and Airy random point fields and Dyson's measures, in spite of the lack of exact Gibbs structure in such models.

\subsection{Singularities of conditional measures and rigidity phenomena} 

In recent years, the local conditional structure of random point fields have been investigated extensively, and it has been demonstrated that the conditional measure $\Pi_{\Dset|\Dset^\c}$ of a point process $\Pi$ on $\RR^d$ (with $\Dset\subset\RR^d$ being a bounded domain) can exhibit a rich variety of singular phenomena. Introduced in \cite{GP}, the notion of \textit{rigidity phenomena} formalizes this behavior. Roughly speaking, the rigidity phenomenon for a statistic $\Phi$ on $\Pi_{|\Dset}$ as above entails that the value of the random variable $\Phi(\Pi_{|\Dset})$ is in fact determined almost surely (abbrv. a.s.) by the point configuration $\Pi_{|\Dset^\c}$; in other words, the random variable $\Phi(\Pi_{|\Dset})$ is measurable with respect to $\Pi_{|\Dset^\c}$. For a rigorous definition we refer the reader to Definition~\ref{def:rigidity}.

The possible nature of the rigid statistic can, in principle, be quite arbitrary. A natural class of possible statistics is provided by various moments of the points in $\Pi_{|\Dset}$, and indeed, these turn out to be the rigid statistics in many natural point processes, as was established in \cite{GP} for the Ginibre ensemble and the zeros of the standard planar Gaussian analytic function. The infinite Ginibre ensemble is the weak limit of the eigenvalues of non-Hermitian random matrices with independent and identically distributed (abbrv. i.i.d.) standard complex Gaussian entries. The standard planar Gaussian analytic function (abbrv. GAF) is the random entire function $\sum_{k=0}^\infty \xi_k \frac{z^k}{\sqrt{k!}}$ in the complex variable $z$, with the coefficients $\xi$ being i.i.d.\ standard complex Gaussians ; it is well-known that its zeros form an isometry-invariant point process on $\CC$. For more detailed descriptions of these models, we refer the reader to Appendices~\ref{a:ginibre} and \ref{a:gaf}. 

In \cite{GP}, it was shown that for the Ginibre ensemble, the number of points of the process in a bounded domain $\Dset \subset \CC$ is rigid, i.e., determined almost surely by the configuration of points in $\Dset^\c$. On the other hand, for the GAF zero ensemble, the rigid statistics are the number as well as the center-of-mass (i.e., the mean) of the points of the process in $\Dset$. The investigation of rigidity phenomena for random point fields has spawned an substantial literature. This entails investigation of rigidity structures in a wide array point processes that are of interest in probability theory and statistical physics, including the Dyson sine process \cite{G1}, the Airy, Bessel and Gamma processes \cite{Buf1}, and more generally a wide class of \textit{determinantal point processes} \cite{Buf2,Buf3,BufShi,Qi,KaShi,Ka}. Rigidity phenomena have also been investigated in more general settings, such as stationary stochastic processes and random Schrodinger or stochastic Airy operators \cite{BuDeQi,Ghosal1,Ghosal2,GoSho}. Related phenomena, such as appearance of forbidden regions under spatial conditioning \cite{GNi1,GNi2}, maximal rigidity \cite{GL1,KiNi}, the relationship between rigidity phenomena and Palm measures \cite{G2,OsSh,BuFaQi}, applications to percolation \cite{GKP,KaKa,GSa} as well as completeness problems \cite{G1}, Coulomb and Riesz gases \cite{Cha,GaSa,LeSe,DeVa,Lew}, random measures and stable matchings \cite{KlLa,KlLaYo} and directional effects in rigidity and dependency phenomena \cite{GL2,GRi} have attracted attention. Investigation of DLR equations, especially in the context of Dyson-type processes, has been undertaken in \cite{Leble,DeVa}; see also \cite{BuQiSha} for spatial conditioning in general determinantal processes and its connections to the Lyons-Peres completeness conjecture. For an overview of rigidity phenomena and its interfaces to wider themes in statistical mechanics, we refer the reader to \cite{GLsurvey,GL0,Cos,Lew}.  

It was also established in \cite{GP} that the local particle number for the Ginibre ensemble or the local number and local center of mass for the GAF zero ensemble, form a complete set of rigid statistics for the respective point processes. E.g., if $N$ is the number of Ginibre points in a disk $\Dset$, the conditional distribution $\Pi_{\Dset | \Dset^\c}$ is mutually absolutely continuous with respect to the Lebesgue measure on $\Dset^N$. A similar result holds for the GAF zero ensemble, where given the number $N$ and the sum $s$ of the GAF zeros in $\Dset$, the conditional distribution $\Pi_{\Dset | \Dset^\c}$ is mutually absolutely continuous with respect to the Lebesgue measure on $\Sigma_{N,s}$; with $\Sigma_{N,s}$ being the set $\Set{(z_1,\ldots,z_N)\in\Dset^N \given \sum_{i=1}^N z_i = s} \subset \CC^N$. The latter class of phenomena is referred to as \textit{tolerance}. In \cite{GK}, a one-parameter family of general Gaussian analytic functions (called $\alpha$-GAFs) was introduced, which exhibits an increasing number of rigid moments of the zeros in $\Dset$ as the parameter $\alpha$ varies over $\RR_+$. For concrete definitions and statements of these results, we refer the reader to Theorem~\ref{thm:app:alphagaf}.

With increasing levels of rigidity, the conditional measure $\Pi_{\Dset | \Dset^\c}$ becomes increasingly singular, in the sense that their support becomes even more restricted and lower dimensional subsets of the ambient space. The notion of quasi-Gibbs property, however, entails mutual absolute continuity of the local conditional distribution of the point process with the Poisson process on the same domain (conditioned on the particle number). It is thus of limited effectiveness in studying point processes with higher orders of singularity, where the constraints on local particle configurations are much more than the mere conservation of their numbers. 

\subsection{Approximate and Generalized Gibbsianity} 

In the present work, we put forward a new and more general paradigm of approximate Gibbsian structure for random point fields, with the objective of mitigating the difficulties outlined above in the context of point processes with strong spatial singularities. The precise structure of our approach is laid out in detail in Section~\ref{sec:technique}; herein, we discuss some important features thereof. 

First, we proceed to define a notion of \textit{generalized Gibbs property} that is primarily meant for a system in the infinite volume limit. However, such a definition would necessarily be mostly conceptual, since estimates can usually be obtained for finite particle systems. Thus, for a sequence of finite particle approximations, we will subsequently introduce a notion of \textit{approximate Gibbs property}, which entails certain inequalities that can be verified via the joint probability densities of the finite particle systems. We demonstrate that the approximate Gibbsian structure on the finite particle systems implies a generalized Gibbs property for their infinite volume limit. This is the content of Theorem~\ref{thm:abscont}.

To lay out the programme in more concrete terms, let $\XX$ be a point process on $\RR^d$ that exhibits \textit{rigidity of numbers}, and let $\Dset$ be a bounded domain in $\RR^d$. 
Let $\cS{F}$ denote the space of locally finite point configurations on a Borel subset $F\subset\RR^d$.
Let $\xin=\XX\cap\Dset$, $\xout=\XX\cap\Dset^\c$.
Thus, the particle number $|\xin|=\numberofpoints(\xout)$ a.s.\ 
for some measurable function $\numberofpoints:\cS{\Dset^\c}\to\nat$.
Consider the conditional distribution of $\xin$ given $\xout$, denoted by $\PROB_{\Dset|\Dset^\c}\left[\cdot | \cdot \right]$, which exists by the general theory of regular conditional distributions (c.f.\  \cite{Kallenberg2021FoundationsProbability}). Suppose that, conditioned on $\xout=\Upsilon$, the points of $\xin$, considered as a vector, live on a smooth symmetric submanifold $\Sigma(\Upsilon)\subset\CC^{\numberofpoints(\Upsilon)}$ (here \textit{symmetric} entails that if $\zeta\in\Sigma(\Upsilon)$ then $\permutation\cdot\zeta\in\Sigma(\Upsilon)$ for all permutations $\permutation\in S_{\numberofpoints(\Upsilon)}$, where $\permutation\cdot\zeta$ is the vector in $\CC^{\numberofpoints(\Upsilon)}$ obtained by permuting the coordinates of $\zeta$ by the action of $\permutation$). Let $\Phi,\Psi:\RR^d\to\RR\cup\{\infty\}$ be two \textit{potential functions}. For a finite point configuration $\s\subset\RR^d$, define the Hamiltonian 
\[
\calH^{\Phi,\Psi}[\s]=\sum_{x\in\s}\Phi(x) + \sum_{\{x,y\}\subset\s}\Psi(x-y)\;.
\]
Let $\Poisson_\Dset^{\numberofpoints(\Upsilon),\Sigma(\Upsilon)}$ be the standard Poisson point process on $\Dset$ conditioned to have $\numberofpoints(\Upsilon)$ points, and for those points, considered as a vector in $\CC^{\numberofpoints(\Upsilon)}$, to lie on the submanifold $\Sigma(\Upsilon)$. We say that $\XX$ satisfies the \textit{generalized Gibbs property} with the potentials $(\Phi,\Psi)$ if for $\PROB_{\infty,\out}$-a.s.\ $\Upsilon$ there are positive quantities $m(\Upsilon)$, $M(\Upsilon)$ such that for all Borel subset $A\subset\cS{\Dset}$
    \begin{equation} \label{informal0}
    m(\Upsilon)
    \int_A
    \exp\left(-\mathcal{H}^{\Phi,\Psi}[\s]\right) \Poisson_\Dset^{\numberofpoints(\Upsilon),\Sigma(\Upsilon)}(\s)
    \le
    \PROB_{\Dset|\Dset^\c}\left[ A \big| \Upsilon \right] 
    \le 
    M(\Upsilon) 
    \int_A 
    \exp\left(-\mathcal{H}^{\Phi,\Psi}[\s]\right)
    \Poisson_\Dset^{\numberofpoints(\Upsilon),\Sigma(\Upsilon)}(\s)\;.
    \end{equation}
    
In the same setting, the approximate Gibbsian property for $\XX$ and a sequence $(\xn)$ of finite particle approximations of $\XX$ can be motivated as follows. Suppose $\xn\to\XX$ a.s.\ and let $\xnin=\xn\cap\Dset;\; \xnout=\xn\cap\Dset^\c$. Consider Borel subsets $A\subset\cS{\Dset}$ and $B\subset\cS{\Dset^\c}$. Denoting by $\PROB^{(n)}_{\Dset|\Dset^\c}[\cdot|\cdot]$ the conditional distribution of $\xnin$ given $\xnout$ and by $\PROB_{n,\out}$ the marginal law of $\xnout$ we can write a canonical expression 
\begin{equation}\label{informal1}
\Prob[\Big]{\,\fbr[\big]{\,\xnin\in A\,}\,\cap\,\fbr[\big]{\,\xnout \in B\,}\,} 
= 
\int_B \PROB^{(n)}_{\Dset|\Dset^\c}\big[A|\Upsilon\big]\deri\PROB_{n,\out}[\Upsilon] 
=
\Exp\tbr[\Big]{\PROB^{(n)}_{\Dset|\Dset^\c}\big[A|\xnout\big]\cdot\bigindicator{\xnout\in B}}\;.
\end{equation}
Let $\nu_{\Phi,\Psi,\Dset}:\Borel{\cS{\Dset}}\times\cS{\Dset^\c}\mapsto[0,1]$ be the probability kernel (c.f.\cite{Kallenberg2017RandomApplications}) given by 
\[
\nu_{\Phi,\Psi,\Dset}\kernel*{A}{\Upsilon}\coloneqq
\frac{1}{Z(\Upsilon)}\int_A \exp\left(-\calH^{\Phi,\Psi}[\s]\right)
\Poisson_\Dset^{\numberofpoints(\Upsilon),\Sigma(\Upsilon)}(\s)\;.
\]  
By way of an approximate Gibbsian structure, we may begin with a somewhat naive criterion that the quantity in \eqref{informal1} is comparable to (i.e., bounded from above and below up to suitable multiplicative factors) the quantity 
\begin{equation}\label{informal2}
    \Exp\tbr[\Big]{\nu_{\Phi,\Psi,\Dset}\kernel*{A}{\xout}\cdot\bigindicator{\xnout\in B}}\;.
\end{equation}
This is, however, too strong a restriction to demand of the finite particle conditional laws $\PROB^{(n)}_{\Dset|\Dset^\c}[\cdot|\cdot]$, and fail to hold, especially in settings of our interest where the eventual infinite particle limit $\XX$ has singular conditional distributions. To mitigate this difficulty, we posit that the terms in \eqref{informal1} is comparable, in the sense of upper and lower bounds, a term like in \eqref{informal2} but only in a \textit{weak sense}. To be more specific, we posit that for a certain collection of \textit{good events} $\{\onj\}_{j \ge 1}$ (that are measurable with respect to $\xnout$), and a rich enough class of events $A\subset\cS{\Dset}$ and $B\subset\cS{\Dset^\c}$, the quantity
\begin{equation}\label{informal3}
\PROB\fbr[\Big]{\,
\fbr[\big]{\,\xnin\in A\,}
\,\cap\, 
\fbr[\big]{\,\xnout \in B\,} 
\,\cap\,\onj\,}
\end{equation}
is comparable, via matching upper and lower bounds, to 
\begin{equation}\label{informal4}
\Exp\tbr[\Big]{
\nu_{\Phi,\Psi,\Dset}\kernel{A}{\xout}
\cdot\bigindicator{\xnout \in B}
\cdot\bigindicator{\xnout \in \onj}} + \vartheta(j,n)\;,
\end{equation}
where $\vartheta(j,n)$ is an additive error term that converges to $0$ for each fixed $j$ as $n\to\infty$. The $\onj$-s, for each $j$, are finite $n$-particle approximation to certain events $\oj$ (measurable with respect to $\xout$), which themselves have a desirable asymptotic behaviour (as $j \to \infty$) in the context of the spatial dependency structure of the infinite volume point process $\XX$. For a detailed, rigorous description of these notions, we refer the reader to Section~\ref{sec:technique}.

A key result that we establish in this article is that the approximate Gibbs property implies the generalized Gibbs property, thereby enabling us to deduce the Gibbs-type comparisons on infinite volume conditional measures  \eqref{informal0} from the estimates on the finite particle approximations as laid out above. In fact, we are able to deal with more general classes of comparing measures that the restricted Gibbs-type potentials such as $\Dset$ within the ambit of our general framework. Further, it suffices that the approximate Gibbs comparison inequalities hold only for a subsequence of $\{n_k\}_{k \ge 1}$ in the variable $n$. For a rigorous discussion of this result and its attributes, we refer the reader in particular to \ref{thm:abscont}, and  to Section~\ref{sec:technique} in general. For many random point fields of interest, such as the $\alpha$-GAFs, these comparisons hold for all bounded measurable domains in the ambient space (with appropriate choices of the potentials $(\Phi,\Psi)$), whence we say that the point field satisfies the generalized Gibbs property with respect to the potentials $(\Phi,\Psi)$.

We observe that the comparison of the terms in \eqref{informal3} and \eqref{informal4} is occurring in a \textit{weak sense} in two major ways: first, the inequalities hold only up to an additive correction that decays with growing system size ; and secondly, the comparison holds only on certain \textit{good events} (namely, the $\Omega_n(j)$-s), and not in general. These \textit{good events} occupy an increasingly large fraction of the probability space with growing system size, but yet not all of it for any finite $n$-particle system - indeed, for many point processes that have strong rigidity properties in the infinite volume limit, such a requirement on finite particle approximations would simply not be true.

\subsection{Implications for strongly singular point fields}

A fundamental implication of our approach to approximate Gibbs structures is that it allows us to obtain comparison inequalities for spatial conditioning on particle systems in the infinite volume limit, even when the latter have strong spatial singularities and might not have analytically tractable forms. 
While the \textit{regular conditional distributions} for spatial conditionings in point processes exist by abstract theory,  
it is generally very difficult to deduce any concrete information about them in the infinite volume limit, except in special cases, such as systems with an exact (or quasi) Gibbs structure. 
This is because the singularity of the conditional distribution in the infinite volume limit is usually not observed in the finite particle approximation, where a joint density for the entire particle system would normally exist (see, e.g., the standard planar GAF zero process \cite{GP}). 

Our approach to approximate Gibbs structures is able to address this problem in a broad class of strongly singular point processes. This is encapsulated in Theorem~\ref{thm:abscont}, wherein a very general technique is demonstrated for transitioning from comparison inequalities for finite particle systems to those for a limiting infinite particle system. 
To our knowledge, such results pertaining to strongly singular point processes are unknown in the literature. In fact, the most singular processes for which Gibbs-type bounds on conditional measures are known all exhibit no further rigidity that the rigidity of local particle numbers (in the sense of \cite{GP}); it may be noted that the spatial conditioning $\Pi_{\Dset|\Dset^\c}$ for such processes is usually absolutely continuous with respect to the canonical Poisson process (conditioned to have the right particle number $N$); see e.g. the literature on the quasi-Gibbs property (c.f. \cite{Os21survey} and the references therein). This latter distribution is reasonably tractable; taken in uniform random order, the points are uniformly distributed on the appropriate power $\Dset^N$ of the domain $\Dset$. Further,  the number of points in $\Dset$ being a desired value $N$ for the finite particle system is usually an event of positive probability, and conditioning on this event gives a good approximation to the infinite particle system conditioned to have $N$ points in $\Dset$. No such advantages are available for higher order rigidity : e.g., for the finite particle approximations to the planar GAF zero ensemble, the event that the centre of mass of the particles in $\Dset$ equals a desired value $s$ is an event of zero probability.

Significant models of random point fields for which our framework sheds particularly useful light include zeros of the standard planar GAF, and more generally, the zero ensembles of $\alpha$-GAFs, which are canonical generalizations of the standard GAF into a one-parameter family. Since the $\alpha$-GAFs exhibit increasingly singular conditional structure as $\alpha$ varies (roughly, $\lfloor\frac{1}{\alpha}\rfloor$ moments of the point configuration in $\Dset$ are determined a.s.\ by that in $\Dset$), the effectiveness of our approach for $\alpha$-GAFs demonstrates its ability to address highly singular spatial structures. In general, for any $\alpha$ we are able to show that the conditional law $\Pi_{\Dset|\Dset^\c}$ for such a process has a density (with respect to a canonical background measure on its support) that is comparable to the squared Vandermonde density. Thus, we establish in particular that even under spatial conditioning, the close-range repulsion structure of such a process is preserved, wherein the joint density decays like the square of the Euclidean separation between neighboring particles. The detailed statement of these results maybe obtained in Theorem~\ref{thm:alphagaf-main}. It goes without saying that our results are also able to address the case of less singular processes, such as those with the exact Gibbs property or random matrix type ensembles such as the Ginibre ensemble, sine and Airy random point fields, among others. In particular, the case of the Ginibre ensemble has been discussed in Section~\ref{sec:ginibre} as a demonstration of some of the main features of our approach in a relatively simple scenario.

Indeed, one may observe that once the comparison inequalities for the finite particle ensembles are available, Theorem~\ref{thm:abscont} can be invoked as a black box in order to deduce comparison inequalities for the infinite particle system which is of main interest. To our knowledge, this is the arguably the first result that provides a general, principled toolbox to directly access Gibbs-type properties of strongly singular particle systems in the infinite volume  limit. This opens the door to potential applications to very general classes of random point fields, where the infinite particle system of interest might be intractable but analytical estimates on finite ensemble approximations are nonetheless available.

\subsection{The emergence of singularity for limits of spatially conditioned point fields}

A fundamental problem in Gibbs-type comparisons for strongly singular processes is that the conditional measure $\PROB_{\Dset|\Dset^\c}$ and its approximation $\PROB_{\Dset|\Dset^\c}^{(n)}$ are supported on different sets, with the support of $\PROB_{\Dset|\Dset^\c}$ often being a singular manifold  (see Theorem~\ref{thm:app:gaf:tol} for reference). As such, a direct comparison inequality on the conditional density for finite ensembles is of limited value in such a situation, since conditional measure for the limiting infinite ensemble will live on a different support. However, known Gibbs-type comparison results are structurally unable to address this problem. In this context, it may be worthwhile to note that our finite system comparison inequalities are structured in the form of upper on lower bounds on \textit{conditional  probabilities} of only \textit{certain particular events}, and \textit{not} on the conditional densities per se (roughly speaking, the latter would entail comparison inequalities on probabilities of all events). This enables us to mitigate the problem of differing support sets for $\PROB_{\Dset|\Dset^\c}$ and $\PROB_{\Dset|\Dset^\c}^{(n)}$.

Another major difficulty in dealing with conditional measures for random point fields of growing size is that, the conditioning events (defined in terms of the point configuration on $\Dset^\c$) \textit{do not} have good consistency properties in general. To be more explicit, let us consider the situation where the set of point configurations $A$ in \eqref{informal1} contains point configurations with a fixed number of points $m$. Then, for $n_1\neq n_2>m$, on the event $\bbX_{n_1,\inn}\in A$, we must have $|\bbX_{n_1,\out}|=n_1-m$, and on the event $\bbX_{n_2,\inn}\in A$, we must have  $|\bbX_{n_2,\out}|=n_2-m$. This implies, in particular, that the integral in \eqref{informal1} has to be taken over disjoint subsets of $\cS{\Dset^\c}$. This means that the conditional probabilities $\PROB^{(n)}_{\Dset|\Dset^\c}[A|\Upsilon]$ considered as functions of $\Upsilon\in\cS{\Dset^\c}$, are supported on disjoint subsets of $\cS{\Dset^\c}$ as $n$ varies. This poses a challenge in understanding the limiting behavior of these conditional measures as $n\to\infty$. 

A key contribution of the present work is to introduce an architecture and a toolbox to understand such limiting of conditional measures in a rigorous manner, especially in a setting with a singular infinite volume limit. We subsequently use this analysis to obtain results on the infinite volume conditional law $\PROB_{\Dset|\Dset^\c}[\cdot|\cdot]$, which is the main goal from a statistical mechanical point of view. We believe that this toolbox can be effectively used for studying other problems (beyond Gibbs-type properties) for strongly singular infinite particle systems; this includes but is not limited to potential applications to invariant dynamics thereon. 

While previous works, such as those on quasi-Gibbs properties, often used finite system comparison inequalities largely as tools to study certain specific aspects (such as dynamics) for the infinite volume limit (without drawing direct statistical mechanical conclusions about the limiting particle system), in this article we obtain direct comparison inequalities for the conditional laws of infinite volume limit. This raises the possibility of an application of our results for the study statistical mechanical properties of strongly singular particle systems (including, in particular, their dynamics) by directly working with the infinite volume limit.

\subsection{Applications of approximate Gibbsianity} 

\subsubsection{A precise hierarchy in levels of rigidity}

We demonstrate the broad scope of our approach by using it to settle an open question on the existence of infinite point processes at arbitrary levels of rigidity. To lay out the problem, we briefly recall the phenomena of rigidity and tolerance in random point fields and the hierarchical structure thereof. To be succinct, we will focus on the setting of rigidity of moments for point processes defined on the complex plane $\CC$, with regard to a bounded domain $\Dset$. Suppose for a point process $\Pi$ there are $k$ rigid moments $\{\sum_{j=1}^{M_0} z_j^p = M_p \, ; \, 0 \le p \le k-1 \}$ of the points $\{z_j\}_{j \in M_0}$ of $\Pi$ inside $\Dset$. Then the phenomenon of tolerance subject to these $k$ rigid moments entails that, on the set $\Sigma_{(k)}:= \cap_{p=0}^{k-1} \Sigma_p$ with $\Sigma_p:=\{\sum_{j=1}^{M_0} z_j^p = M_p\}$, the conditional measure $\Pi_{\Dset|\Dset^\c}$ is mutually absolutely continuous with respect to the Lebesgue measure on $\Sigma^{(k)}$. If a point process on $\CC$ satisfies this condition for all bounded measurable sets $\Dset$, then the point process $\Pi$ is said to be \textit{rigid at level} $k$. 

It was established in \cite{GP} that the Ginibre ensemble is rigid at level $0$, whereas the zeros of the standard planar GAF are rigid at level $1$. It is a natural question as to whether there exist point processes that are rigid at level $k$, for any given $k\in\NN$. This question turns out to be surprisingly challenging; while it would be of great interest to show the existence of $k$-level rigid random point fields that have close connections to important models in statistical physics, even toy examples are in fact hard to come by. In \cite{GK}, a one-parameter family of generalized Gaussian analytic functions was introduced, referred to as the $\alpha$-GAFs, with the parameter $\alpha \in (0,\infty)$. In explicit terms, the $\alpha$-GAF is defined as the random entire function $\sum_{k=0}^\infty \xi_k \frac{z^k}{(k!)^{\alpha/2}}$ where $\xi_k$s are i.i.d.\ standard complex Gaussian random variables. The family of $\alpha$-GAFs includes, in particular, the standard planar GAF for the particular choice of parameter $\alpha=1$, and thus its zero set belongs to the wider class of point processes pertaining to Coulomb type processes and their generalizations. 

It was demonstrated in \cite{GK} that the zeros of the $\alpha$-GAF have $k$ rigid moments, where $k=\lfloor\frac{1}{\alpha}\rfloor$. 
However, rigidity at level $k$ involves demonstrating that, subject to $k$ rigid moments, there is \textit{tolerance}, as discussed above. Establishing tolerance is, in general, a challenging problem, and this was left open for $\alpha$-GAFs in \cite{GK}, thereby leaving the program of investigating general $k$-rigidity incomplete. 

In this paper, we settle this problem, by demonstrating that for any bounded measurable subset $\Dset\subset\CC$, the conditional density of the zeros of $\alpha$-GAF is in fact comparable to the squared Vandermonde density with respect to the Lebesgue measure on $\Sigma^{(k)}$ defined as above with $k=\lfloor\frac{1}{\alpha}\rfloor$. For a complete and rigorous statement, we refer the reader to Theorem~\ref{thm:alphagaf-main}. This, in particular, implies tolerance for the $\alpha$-GAF zeros, subject to the first $\lfloor\frac{1}{\alpha}\rfloor$ moments. \textit{A fortiori}, this answers the question raised in \cite{GK} in the affirmative, and establishes the $\alpha$-GAFs as a one-parameter family of planar point processes that exhibit a complete hierarchy of rigidity structures with all possible levels of rigidity between $1$ and $\infty$ attainable by tuning the parameter $\alpha$ appropriately. In particular, the present work fully subsumes as a special case the quantitative estimates obtained for the standard planar GAF in the preprint \cite{Quantitative-Estimates} by the second-named author. More generally, our analysis in this paper establishes a systematic framework to investigate the technically challenging tolerance phenomena and problems of spatial conditioning at large for general classes of point processes.

\subsubsection{Bounds on relative energies}

Investigating the relative energies of configurations is an important tool in statistical mechanics, especially those with long range or higher order dependencies where a simple Gibbs structure is absent (c.f. \cite{Lew} and the references therein for a discussion in the context of Coulomb type systems; see also \cite{Ruelle}). While the absolute energy of particular configuration may be difficult to make rigorous sense of, it is often technically simpler to consider the energy difference between two configurations. A  setting of great significance in this context would perhaps be the energy difference between two configurations that are \textit{local perturbations} of each other, i.e., the result of a transformation of point configurations that acts as identity outside a suitably large compact set of the ambient space. 

For a system that accords a simple local Gibbs structure with a local Hamiltonian $\calH_\Dset$ (on a bounded domain $\Dset$) (c.f.\ Section~\ref{ss:localconditioning}), the energy can be taken to be simply $\calH_\Dset$. Additionally, we may consider the canonical decomposition $\calH_\Dset[\s] =  \calH_1[\s_{|\Dset}] - \beta \calH_2[\s_{|\Dset},\s_{|\Dset^\c}]$, with the boundary effect $\calH_2[\cdot,\cdot]$ being controlled (possibly in terms of the outside configuration $\s_{|\Dset^\c}$). In this setting, the energy difference between two configurations $\s^{(1)}$ and $\s^{(2)}$ with $\s^{(1)}_{|\Dset^\c}=\s^{(2)}_{|\Dset^\c}=\omega$ would simply be $\calH_1[\s^{(1)}_{|\Dset}]-\calH_1[\s^{(2)}_{|\Dset}]$ (up to a bounded additive term that depends on $\omega$). This form is particularly convenient, since $\calH_1$ is an internal energy term that depends only on the finite configuration of points inside the bounded domain $\Dset$, and in well-structured models, can often have a very explicit and tractable algebraic form (c.f.\ the Ising spin system in the discrete setting). For more general point fields, log conditional density of a configuration on a bounded domain would be a natural substitute for the local energy, and the difference between such log conditional densities would provide a indication of the change in energy between two comparable configurations. 

Random point fields with long-ranged correlations and singularities in their conditional structure admit hardly any of this simplistic description. Yet, for motivations stemming from statistical physics as well as stochastic geometry, it would be of great interest to obtain similar tractable bounds on the relative energy of configurations when they are obtained from each other via a local perturbation (as discussed above). While the physical motivations are of classical interest (\cite{Lew,Ruelle}), the stochastic geometric considerations are also significant in the context of recent advances in those directions, for instance see \cite{GKP} for an investigation of continuum percolation on the  Ginibre and  Gaussian zero models. The classic Burton and Keane argument, as an illustrative example, obtains stochastic geometric consequences (in particular, uniqueness of infinite cluster in percolation) via local perturbations of configurations in a bounded domain while freezing the \textit{environment} (i.e., the configuration outside the said domain). In spatially singular models, such as the Ginibre or  Gaussian zeros process, this can only be done in a manner that respects the rigidity structure of local statistics of these processes (e.g., preserving the mass and the centre of mass for the GAF zeros) \cite{GKP}. While the vanilla Burton and Keane type argument relies on existence of such desirable local perturbations, the study of finer, quantitative stochastic geometric properties would call forth estimates on the energy cost of such local perturbations, which is our object of interest herein.

It is a consequence of Theorem~\ref{thm:alphagaf-main} that the relative energy between two configurations $(\alpha^{(1)},\omega)$ and $(\alpha^{(2)},\omega)$ (with $\alpha^{(1)},\alpha^{(2)}$ supported on a bounded domain $\Dset$ and $\omega$ supported on $\Dset^\c$) will be bounded  by $|\log \nu(\cdot ; \omega) - \log \nu(\cdot ; \omega)| + A(\omega)$, where $\nu(\cdot ; \omega)$ is the comparing conditional density and $A(\omega)$ is an additive constant that depends only on the environment $\omega$. This  assumes particular significance in a setting where, as we shall see in the case of $\alpha$-GAF zero processes, the comparing  density $\nu$ has a simple, tractable form.

For the $\alpha$-GAF zero processes, Theorem~\ref{thm:alphagaf-main} (in particular, Corollary~\ref{cor:energy} thereof) implies that the relative energy would be bounded above, upto an additive constant that depends only on the environment $\omega$, by the difference between the \textit{logarithmic energies} of the configurations $\alpha^{(1)}$ and $\alpha^{(2)}$. Here the logarithmic energy of a finite point configuration $\alpha$, denoted by $\mathcal{E}_{\log}(\alpha)$, is defined as $-\sum_{x,y \in \alpha ; x \ne y} \log |x-y|$. This demonstrates the fact that, while the $\alpha$-GAF zeros have an intricate, many body interaction structure entailing arbitrary orders of spatial singularity, perturbations between legitimate local configurations are nonetheless energetically inexpensive, and their relative energies are bounded above by that of simple 2D Coulomb type system with only two body interactions. This reveals an intriguing interplay between the spatial rigidities of the $\alpha$-GAF zeros on a global scale, and a certain regularity on the local scale, wherein the impact of the strong dependency structure can nonetheless be effectively dominated locally by a simple logarithmic Coulomb system.

\section{The generalized and approximate Gibbs properties} \label{sec:technique}

In this section, our objective is to state the key technical result, which is Theorem~\ref{thm:abscont}. 
We prove this result in Section~\ref{sec:pf:thm:abscont}.
We describe the general setup of this result now.
Consider a probability space $\prsp$ equipped with a probability measure $\PROB$.
We will consider point processes from $\prsp$ to $\RR^d$ for some $d\in\NN$.
For any Borel subset $F\subset\RR^d$, let $\cS{F}$ denote the space of locally finite point configurations on $F$.
Let $\Borel{\cS{F}}$ denote the Borel sigma-algebra on $\cS{F}$.
Let $\bfD$ denote the set of all bounded open subsets of $\RR^d$ whose boundary has zero Lebesgue measure.
Let us recall the definition of \textbf{rigidity}.

\begin{definition}[Rigidity]\label{def:rigidity}
Consider a point process $\bbX:\prsp\to\cS{\RR^d}$ whose first intensity measure is absolutely continuous with respect to the Lebesgue measure on $\RR^d$. Consider a set $\Dset\in\bfD$. A measurable function $\phi:\cS{\Dset}\to\RR$ is said to be \emph{rigid} with respect to $\bbX$ if there exists a measurable function $\psi:\cS{\Dset^\c}\to\RR$ such that $\psi\fbr[\big]{\bbX\cap\Dset^\c} = \phi\fbr[\big]{\bbX\cap\Dset}$ a.s. The point process $\bbX$ is said to be number-rigid if the number of points of $\bbX\cap\Dset$ is rigid for all $\Dset\in\bfD$.
\end{definition}

\subsection{The generalized Gibbs property}

In this section, our objective is to define the generalized Gibbs property. Let $\bbX:\prsp\to\cS{\RR^d}$ be a point process whose first intensity measure is absolutely continuous with respect to the Lebesgue measure on $\RR^d$. Consider $\Dset\in\bfD$. 
Let $\PROB_{\inn}$ be the distribution of $\bbX_{\inn}\coloneqq\bbX\cap\Dset$.
Let $\PROB_{\out}$ be the distribution of $\bbX_{\out}\coloneqq\bbX\cap\Dset^\c$. 
Let $\PROB_{\Dset|\Dset^\c}[\cdot|\cdot]$ be the conditional measure of $\bbX_{\inn}$ given $\bbX_{\out}$.

\begin{definition}[The generalized Gibbs property with respect to a probability kernel]\label{def:gengibbs} 
Let $\nu:\Borel{\cS{\Dset}}\times\cS{\Dset^\c}\to[0,1]$ be a probability kernel (see \cite{Kallenberg2021FoundationsProbability} for reference).  
We say $\bbX$ satisfies the generalized Gibbs property with respect to $\nu$ on the domain $\Dset$
if there exists measurable functions $\mm,\MM:\cS{\Dset^\c}\to(0,\infty)$ such that
\begin{equation}
\mm\fbr[\big]{\bbX_{\out}}\,\nu\kernel{\cdot}{\bbX_{\out}} 
\,\le\,
\PROB_{\Dset|\Dset^\c}[\cdot|\bbX_{\out}] 
\,\le\,
\MM\fbr[\big]{\bbX_{\out}}\,\nu\kernel{\cdot}{\bbX_{\out}}\quad\mbox{a.s.}
\label{targeteq}
\end{equation}
\end{definition}

\begin{definition}[The generalized Gibbs property with respect to potentials]\label{def:gengibbsII}
Assume that $\bbX$ is number-rigid i.e., $\abs{\bbX_{\inn}}=\numberofpoints\fbr{\bbX_{\out}}$ for some measurable function $\numberofpoints:\cS{\Dset^\c}\to\nat$. Suppose that, conditioned on $\bbX_{\out}=\Upsilon\in\cS{\Dset^\c}$, the points of $\bbX_{\inn}$ considered as a vector in $\CC^{\numberofpoints(\Upsilon)}$, live on a smooth symmetric submanifold $\Sigma(\Upsilon)\subset\CC^{\numberofpoints(\Upsilon)}$. Let  $\Poisson_\Dset^{\numberofpoints(\Upsilon),\Sigma(\Upsilon)}$ be the standard Poisson point process on $\Dset$ conditioned to have $\numberofpoints(\Upsilon)$ points and for those points as a vector in $\CC^{\numberofpoints(\Upsilon)}$ to lie on the submanifold $\Sigma(\Upsilon)$. Let $\Phi,\Psi:\RR^d\to\RR\cup\{\infty\}$ be two \textit{potential functions}. For a finite point configuration $\sigma\subset\RR^d$, define the Hamiltonian 
\[
\mathcal{H}^{\Phi,\Psi}[\sigma]=\sum_{x\in\sigma}\Phi(x)+\sum_{x,y\in\sigma}\Psi(x-y)\;.
\]
We say $\bbX$ satisfies the \textit{generalized Gibbs property} with the potentials $(\Phi,\Psi)$ on the domain $\Dset$ if there exists measurable functions $\mm,\MM:\cS{\Dset^\c}\to(0,\infty)$ such that for $\PROB_{\infty,\out}$-a.s.\ $\Upsilon$ we have for all $A\in\Borel{\cS{\Dset}}$
\[  
\mm(\Upsilon)
\int_{A}  
\exp\left(-\mathcal{H}^{\Phi,\Psi}[\sigma]\right)\deri\Poisson_\Dset^{\numberofpoints(\Upsilon),\Sigma(\Upsilon)}(\sigma)
\,\le\,
\PROB_{\Dset|\Dset^\c}[A|\Upsilon] 
\,\le\, 
\MM(\Upsilon)
\int_{A}  
\exp\left(-\mathcal{H}^{\Phi,\Psi}[\sigma]\right)\deri\Poisson_\Dset^{\numberofpoints(\Upsilon),\Sigma(\Upsilon)}(\sigma)\;.
\]
We say $\bbX$ satisfies the generalized Gibbs property with the potentials $(\Phi,\Psi)$ if it satisfies the generalized Gibbs property with the potentials $(\Phi,\Psi)$ on all $\Dset\in\bfD$.
\end{definition}

\begin{remark}
If $\bbX$ satisfies the generalized Gibbs property with respect to the potentials $(\Phi,\Psi)$ as in Definition~\ref{def:gengibbsII}, then $\bbX$ satisfies the generalized Gibbs property, as in Definition~\ref{def:gengibbs}, with respect to the probability kernel $\nu_{\Phi,\Psi,\Dset}:\Borel{\cS{\Dset}}\times\cS{\Dset^\c}\to[0,1]$ given by
\begin{equation}\label{eq:def-kernel}
    \nu_{\Phi,\Psi,\Dset}\kernel{A}{\Upsilon} = 
    \frac{1}{Z(\Upsilon)}
    \int_A\exp\left(-\mathcal{H}^{\Phi,\Psi}[\sigma]\right)\deri\Poisson_\Dset^{\numberofpoints(\Upsilon),\Sigma(\Upsilon)}(\sigma)\;,
\end{equation} 
where $Z(\Upsilon)$ is the appropriate normalizing factor.
\end{remark}

\subsection{The approximate Gibbs property}

In this section, our objective is to define the approximate Gibbs property.
Consider $\Dset\in\bfD$. 
First We introduce some notations related to point configurations inside $\Dset$. 
\begin{notation}\label{n:insidetopology}
Consider $m\in\NN$.
\begin{enumerate}[(a),font=\normalfont\bfseries,topsep=\parskip,itemsep=\parskip]
\item Consider $m$-disjoint open balls $B_1,\ldots,B_m$ with rational centers and rational radii in $\Dset$. Let
\[
\Psi(B_1,\ldots,B_m) \coloneqq \Set*{ \Upsilon\in\cS{\Dset} \given 
|\Upsilon\cap B_i| = 1 \mbox{ for all } 1\leq i\leq m}
\]
Then the countable collection
\[
\BasisInsidempts \coloneqq \Set*{ \Psi(B_1,\ldots, B_m) \given B_1,\ldots, B_m \mbox{ as above} }
\]
is a countable basis of the Borel $\sigma$-algebra of point configurations on $\Dset$ with exactly $m$ points. 

\item Let $\FUBasisInsidempts$ be the collection of sets which are finite union of the elements of the basis $\BasisInsidempts$ i.e.,
\[
\FUBasisInsidempts\coloneqq
\Set[\bigg]{\bigcup_{i=1}^k A_i \given A_i\in\BasisInsidempts \mbox{ for all } k\geq 1 \mbox{ and }1\leq i\leq k}\;.
\]

\item Let $\BorelInsidempts$ be the corresponding Borel $\sigma$-algebra. 
\end{enumerate}
\end{notation}

Now let us introduce some notations related to point configurations outside $\Dset$.

\begin{notation}\label{n:outsidetopology}\text{}
\begin{enumerate}[(a),font=\normalfont\bfseries,topsep=0pt]
\item Let $n$ be a positive integer.  
Let $B\subset\Dset^\c$ be a closed annulus whose center is the origin and whose inradius and outradius are both rational. 
Consider a collection of $n$ disjoint open balls $B_i$ with rational radii and centres having rational co-ordinates such that $B_i \cap \Dset^\c\subset B$.
Let $\Psi\fbr[\big]{n,B,B_1, \cdots, B_n}$ be the Borel subset of $\cS{\Dset^\c}$ defined as follows:
\[ 
\Psi\fbr[\big]{n,B,B_1,\cdots,B_n}
\coloneqq 
\Set[\big]{\,\Upsilon \in \cS{\Dset^\c} \given \lvert\Upsilon \cap B\rvert=n, \lvert\Upsilon \cap {B_i}\rvert = 1 \,}\;.
\] 
Then the countable collection 
\[
\BasisOutside
\coloneqq
\Set[\Big]{\,\Psi\fbr[\big]{n,B,B_1,\ldots,B_n}\given n,B, B_i \text{ as above}\,}
\]
is a basis for the topology of $\cS{\Dset^\c}$.

\item Let $\FUBasisOutside$ be the collection of sets which are finite union of the sets in the basis $\BasisOutside$ i.e.,  
\[
\FUBasisOutside
\coloneqq
\Set[\big]{\,\cup_{i=1}^k \Psi_i\given \Psi_i\in\BasisOutside\mbox{ for all }
k\geq 1\mbox{ and } 1\leq i\leq k\,}\;.
\]
\item Let $\BorelOutside$ be the corresponding $\sigma$-algebra.
\end{enumerate}
\end{notation}

\begin{notation}\label{n:asymp}
Let $p$ and $q$ be indices 
    which take values in potentially infinite abstract sets. 
Let $F(p,q)$ and $G(p,q)$ be non-negative functions of these indices. 
We write 
    \[
    F(p,q) 
    \,\overset{q}{\scaleto{\asymp}{8pt}}\, 
    G(p,q)
    \]
    if there exist positive numbers
    $\underline{k}(q)$ and $\overline{k}(q)$ such that 
    for all $p$ and $q$
    \[ 
    \underline{k}(q) \, F(p,q) 
    \,\le\,
    G(p,q) 
    \,\le\, 
    \overline{k}(q) \, F(p,q)\;.
    \] 
\end{notation}

\begin{definition}[A sequence of events exhausts another event]
A sequence of events $\seq*{\calE_j}_{j \ge 1}$ is said to \emph{exhaust} another event $\calE$ if $\calE_j\subset\calE_{j+1}\subset\calE$ for all $j\ge 1$, and $\Prob{\,\calE\setminus\calE_j\,}\rightarrow 0$ as $j\to\infty$.
\end{definition}

The approximate Gibbs property is defined for a sequence of point processes $\seq{\xn}_{n=1}^\infty$ and a limiting point process $\XX$ such that $\xn\to\XX$ a.s. All the point processes are from $\prsp$ to $\cS{\RR^d}$. We assume that the first intensity measures of these processes are absolutely continuous with respect to the Lebesgue measure on $\RR^d$. Further, we assume that $\XX$ is number-rigid i.e., for every $\Dset\in\bfD$, there is a measurable function $\numberofpoints:\cS{\Dset^\c}\to\nat$ such that $\lvert\XX\cap\Dset\rvert = \numberofpoints\fbr{\XX\cap\Dset^\c}$ a.s. Now consider a fixed $\Dset\in\bfD$. 

\begin{notation}\label{n:genpp}
For $n\in\NN\cup\Set{\infty}$, let $\xnin\coloneqq\xn\cap\Dset$, $\xnout\coloneqq\xn\cap\Dset^\c$, $\xin\coloneqq\XX\cap\Dset$, $\xout\coloneqq\XX\cap\Dset^\c$, and let $\PROB_n$, $\PROB_{n,\inn}$, $\PROB_{n,\out}$ be the distributions of $\xn$, $\xnin$, $\xnout$ respectively.
\end{notation}

\begin{definition}[The events $\om$ and $\omn$]
\label{d:om:ommn}
For $m\in\NN$, let $\om$ be the event that the number of points in $\xin$ is $m$. 
For $n,m\in\NN$ with $n\geq m$ let $\omn$ be the event that the number of points in $\xnin$ is $m$.
\end{definition}

\begin{definition}[The approximate Gibbs property with respect to a probability kernel]\label{def:approximategibbs}
We say that $\fbr*{\seq*{\xn}_{n=1}^\infty,\XX}$ satisfies the approximate Gibbs property on the domain $\Dset$ with respect to the probability kernel $\nu:\Borel{\cS{\Dset}}\times\cS{\Dset^\c}\to[0,1]$ if for every $m$ for which $\Prob{\,\om\,}>0$, we have the following: 
\begin{enumerate}[(a),font=\normalfont\bfseries,topsep=0pt]
\item There exists a sequence of events $\seq*{\oj}_{j=1}^\infty$ such that each $\oj$ is measurable with respect to $\xout$ and $\seq*{\oj}_{j=1}^\infty$ exhausts $\om$.

\item\label{c:ag:2}  For each $j\geq 1$ there exists a sequence of positive integers $(n_k)_{k=1}^\infty$, a sequence of events $\fbr[\big]{\onkj}_{k=1}^{\infty}$, and a sequence of real numbers $\fbr[\big]{\vartheta(j,k)}_{k=1}^\infty$ such that the following hold:
\begin{enumerate}[(1),font=\normalfont\bfseries,topsep=0pt]
\item For each $j\geq 1$ we have $\lim_{k\to\infty}\vartheta(j,k)=0$.\label{c:ag:2:1}
\item For each $j\geq 1$ and $k\geq 1$ the event $\onkj$
is measurable with respect to $\xoutnk$.
\item For each $j\geq 1$ and $k\geq 1$ we have $\onkj\subset\omnk$.
\item For each $j\geq 1$ we have $\oj\subset\liminf_{k\to \infty}\onkj$.
\item\label{c:ag:2:5} As functions of the quantities 
    $A\in\FUBasisInsidempts$, 
    $B\in\FUBasisOutside$, 
    $j\geq 1$, 
    $k\geq 1$ we have
\begingroup
\addtolength{\jot}{0.25em}
\begin{align*}
& \Prob[\Big]{\,
\fbr[\big]{\,\xinnk\in A\,}
\,\cap\, 
\fbr[\big]{\,\xoutnk \in B\,} 
\,\cap\,\onkj\,}\\
\overset{j}{\scaleto{\asymp}{8pt}} \;
& 
\Exp\fbr[\Big]{
    \nu\kernel{A}{\xout}
    \bigindicator{\xoutnk\in B}
    \bigindicator{\onkj} 
    }
\,+\, 
\vartheta(j,k)\;.
\numberthis
\label{eq:main}
\end{align*}
\endgroup
That is, 
    the ratio of the left hand side and the right hand side is bounded
    above and below by functions of only $j$ - the quantities $A$, $B$, 
    $k$ are not involved in the bounds. 
\end{enumerate}
\end{enumerate}
\end{definition}

\begin{definition}[The approximate Gibbs property with respect to 
potentials] We say $\fbr*{\seq*{\xn}_{n=1}^\infty,\XX}$ 
satisfies the approximate Gibbs property on the domain $\Dset$ 
with the potentials $(\Phi,\Psi)$ if $\fbr*{\seq*{\xn}_{n=1}^\infty,\XX}$
satisfies the approximate Gibbs property on the domain $\Dset$ with the
probability kernel $\nu_{\Phi,\Psi,\Dset}$ given by \eqref{eq:def-kernel}.  
We say $\fbr*{\seq*{\xn}_{n=1}^\infty,\XX}$ satisfies the approximate Gibbs
property with the potentials $(\Phi,\Psi)$ if it satisfies the approximate
Gibbs property with the potentials $(\Phi,\Psi)$ on all domains $\Dset\in\bfD$.
\end{definition}

\begin{remark}
Note that the left hand side of \eqref{eq:main} can
be written as 
\[
\Exp\fbr[\Big]{
\PROB^{(n_k)}_{\Dset|\Dset^\c}\tbr[\big]{A|\xoutnk}
\bigindicator{\xoutnk\in B}
\bigindicator{\onkj}}
\]
where $\PROB^{(n_k)}_{\Dset|\Dset^\c}\tbr[\big]{\cdot|\cdot}$ is the conditional distribution of $\xinnk$ given $\xoutnk$. This demonstrates that the two sides of \eqref{eq:main} are of a similar nature, with the left hand side being computable purely in terms of the distribution of the finite particle system $\bbX_{n_k}$. On the other hand, if the target conditional measure $\nu\kernel{\cdot}{\xout}$ has a reasonably tractable form, then the right hand side may also be well-estimated (up to additive and multiplicative errors), and thus \eqref{eq:main} can be verified. This programme is indeed possible to carry out for a substantial class of models, including those with arbitrarily high levels of spatial rigidity, as we shall see later in this article.
\end{remark}

\begin{remark}
Observe that the right hand side of \eqref{eq:main} involves the full point process $\bbX_{\infty}$ in the form of $\xout$. This is, in fact, essential, since for strongly rigid point processes $\bbX$, any measure that is comparable to the infinite volume conditional law $\PROB_{\Dset|\Dset^\c}[\cdot|\xout]$ must almost surely be supported on the non-trivial submanifold $\Sigma(\xout)$ (as in Definition~\ref{def:gengibbs}). This information must somehow be incorporated into the finite particle comparison inequalities \eqref{eq:main}, which is the reason for this phenomenon.
\end{remark}

\subsection{From approximate Gibbs to generalized Gibbs}

We are now ready to state a key technical theorem that drives the subsequent major results in the paper.
Broadly speaking, it connects the approximate Gibbs property,
    which is largely dependent on the finite particle approximations, 
    to the generalized Gibbs property, 
    which entails a comparison for the spatially conditional distribution for the infinite volume limit.

\begin{theorem}\label{thm:abscont}
If $\fbr*{\seq*{\xn}_{n=1}^\infty,\XX}$ satisfies the approximate Gibbs property with respect to a probability kernel $\nu$ on a domain $\Dset$, then $\XX$ satisfies the generalized Gibbs property with respect to $\nu$ on the domain $\Dset$.
\end{theorem}

\begin{corollary}
If $\fbr*{\seq*{\xn}_{n=1}^\infty,\XX}$ satisfies the approximate Gibbs property with the potentials $(\Phi,\Psi)$ on a domain $\Dset$, then $\XX$ satisfies the generalized Gibbs property with the potentials $(\Phi,\Psi)$ on the domain $\Dset$. 
\end{corollary}

\begin{remark}
From the proof of Theorem~\ref{thm:abscont}
    it will be clear that Theorem~\ref{thm:abscont} remains valid 
    if in the definition of approximate Gibbs property (Definition~\ref{def:approximategibbs}) 
    we replace the condition that 
    $\oj$ is measurable with respect to $\xout$ 
    by the condition that
    there exists an event 
    $\widetilde{\Omega}(j)$ 
    which is measurable with respect to $\xout$ 
    and which satisfies 
    $\Prob*{\,\oj\,\symmdiff\,\widetilde{\Omega}(j)\,}=0$.
    Additionally, Theorem~\ref{thm:abscont} remains valid if the condition $\Omega(j)\subset\liminf_{k\to\infty}\Omega_{n_k}(j)$ only holds a.s.
\label{rmabs-1}
\end{remark}

\section{Gibbsian structures for strongly singular point fields on \texorpdfstring{$\CC$}{Complex}}\label{s:setup}

In this article, we will investigate approximate and generalized Gibbsianity in the context of point processes on $\CC$; we note in passing that our approach in fact applies to point processes on very general spaces. In particular, we will establish generalized Gibbsian structure for point processes with arbitrarily high levels of spatial rigidity (equivalently, arbitrarily high degrees of singularity in their spatial conditioning), thereby demonstrating the power and scope of the approach outlined in this article.

As a preparation to tackle the higher order singularities present, e.g., in the $\alpha$-GAF zeros, we first discuss the approximate Gibbsian structure of Ginibre ensemble. The Ginibre ensemble accords only a mild degree of spatial singularity (namely, the rigidity of numbers); being a determinantal point process its spatial conditioning can also be accessed via other methods (see, e.g., \cite{BuQiSha} among others). Nonetheless, a discussion on approximate Gibbsianity of the Ginibre ensemble allows us to lay out some of the major ingredients of our approach, and prepares the reader for the more delicate considerations that are called forth by the $\alpha$-GAF zero ensembles in the subsequent sections.

\subsection{Gibbsian structure of the Ginibre ensemble}\label{ss:ginisetup}

Consider the Ginibre ensemble $\GIF$ (see Appendix~\ref{a:ginibre} for reference). Consider $\Dset\in\bfD$ (as an abuse of notation, here we treat $\bfD$ as consisting of subsets of $\CC$ as opposed to $\RR^2$.) Let $\GIFIN\coloneqq\GIF\cap\Dset$ i.e., it is the restriction of $\GIF$ inside the domain $\Dset$. Let $\GIFOUT\coloneqq\GIF\cap\Dset^\c$ i.e., it is the restriction of $\GIF$ outside the domain $\Dset$. Let $\numberofpoints(\GIFOUT)$ be the number of points in $\GIFIN$. The number of points in $\GIFIN$ is measurable with respect to $\GIFOUT$ due to the number rigidity of the Ginibre ensemble (see Theorem~\ref{t:1.1} in Appendix~\ref{a:ginibre} for reference). Let $\frho{\cdot}{\GIFOUT}$ be the conditional distribution of $\GIFIN$ given $\GIFOUT$ where we identify the configuration $\GIFIN$ with an element in $\Dset^{\numberofpoints(\GIFOUT)}$ by taking the points in uniform random order. This distribution has a density with respect to the Lebesgue measure on $\Dset^{\numberofpoints(\GIFOUT)}$ (see Theorem~\ref{thm:Gini:tol} in Appendix~\ref{a:ginibre} for reference). 

\begin{notation}\label{n:vandermonde}
For a vector $\underline{x}=(x_1,\dots,x_N)\in\CC^N$ let
\[
\van{\underline{x}}\coloneqq\prod_{1\leq i<j\leq N}(x_i-x_j)\;.
\]
For two vectors 
$\underline{x}=(x_1,\dots,x_{N_1})\in\CC^{N_1}$
and
$\underline{y}=(y_1,\dots,y_{N_1})\in\CC^{N_2}$
let 
\[
\vancross{\underline{x}}{\underline{y}}\coloneqq
\prod_{i=1}^{N_1}\prod_{j=1}^{N_2}(x_i-y_j)\;.
\]
Therefore
\[
\van{\underline{x}\cc\underline{y}}=
\van{\underline{x}}\cdot
\van{\underline{y}}\cdot
\vancross{\underline{x}}{\underline{y}}\;.
\]
Here $\underline{x}\cc\underline{y}$ denotes the concatenated vector.
\end{notation}

\begin{theorem}[The generalized Gibbsian structure of the Ginibre ensemble]
\label{thm:ginibre-main}
There exist positive quantities $\mm(\GIFOUT)$ and $\MM(\GIFOUT)$, 
    measurable with respect to $\GIFOUT$, 
    such that $\GIFOUT$-a.s.\ we have 
    \begin{equation}\label{eq:target2}
    \mm(\GIFOUT)
    \abs*{\van{\uz}}^2  
    \le \frac{\deri\frho{\cdot}{\GIFOUT}}{\deri\el}(\uz) 
    \le \MM(\GIFOUT)\abs*{\van{\uz}}^2
    \end{equation}
    for a.e.\ $\uz$ with respect to the Lebesgue measure $\el$ on $\Dset^{\numberofpoints(\GIFOUT)}$.    
    In other words, the generalized Gibbs property is satisfied with the potentials $(\Phi,\Psi)$ given by 
    $\Phi\equiv 0$ and $\Psi(z)=-2\log|z|$.
\end{theorem}

Theorem~\ref{thm:ginibre-main} shows that even after the configuration outside $\Dset$ that is 
$\Gini_{\infty,\out}$ is fixed, the points inside $\Dset$ namely the points of 
$\Gini_{\infty,\inn}$ have repulsion among them. The nature of the repulsion is similar
to the repulsion between points of a generic $\GIF$ configuration.

\subsection{Gibbsian structure of the zeros of \texorpdfstring{$\alpha$}{alpha}-GAF}\label{ss:agafsetup}

We are now ready to delve into the approximate Gibbsian structure of the $\alpha$-GAF zero ensemble.
Recall that it is known that the $\alpha$-GAF zero ensemble has $\lfloor \frac{1}{\alpha} \rfloor$ rigid moments, so these point fields can be highly singular depending on the value of the parameter $\alpha$. The following theorem demonstrates the full power and generality of our approach in the context of these highly singular processes, and unveils an approximate Gibbsian structure for them via a comparison of their spatially conditioned densities (on appropriate submaniforlds) to the squared Vandermonde density.

\begin{notation}\label{n:manifold}
For a vector of complex numbers $\us=\fbr*{s_1,\cdots,s_k}\in\CC^k$ and $m\in\NN$, define the manifold 
\[
\Sigma_{\usm}\coloneqq
\Set*{\fbr*{\z_1,\cdots,\z_m}\in\Dset^m\given
\sum_{i=1}^m\z_i^j=s_j\mbox{ for all }1\leq j\leq k}\;.
\]
\end{notation}

Consider the ensemble $\alphagaf$ of zeros of $\alpha$-GAF $\falphagaf$ (see Appendix~\ref{a:agaf} for reference).
Consider $\Dset\in\bfD$.
Let $\alphagafin$ be the restriction of the point configuration $\alphagaf$ inside the domain $\Dset$.
Let $\alphagafout$ be the restriction of the point configuration $\alphagaf$ outside the domain $\Dset$.
Let $\numberofpoints(\alphagafout)$ be the number of points in $\alphagafin$.
The number of points in $\alphagafin$ is measurable with respect to $\alphagafout$ due to the number rigidity of the ensemble $\alphagaf$ (see Theorem~\ref{thm:app:alphagaf} in Appendix~\ref{a:agaf} for reference).
Let $\constraintexternal(\alphagafout)$ be the array of the first $\lfloor 1/\alpha\rfloor$ moments of $\alphagafin$.
This is measurable with respect to $\alphagafout$ due to rigidity of the ensemble $\alphagaf$ up to order $\rigidity\coloneqq 1+\lfloor 1/\alpha\rfloor$ (see Theorem~\ref{thm:app:alphagaf} in Appendix~\ref{a:agaf} for reference).
To be more precise,
if $(\zeta_1,\dots,\zeta_m)$ are the points of $\alphagafin$ in some order, 
then $\constraintexternal(\alphagafout) = (s_1,\dots,s_{\rigidity-1})$
where for $1\leq j\leq \rigidity - 1$, $s_j = \zeta_1^j+\cdots+\zeta_m^j$. 
Let $\frho{\cdot}{\alphagafout}$ be the conditional distribution of $\alphagafin$ given $\alphagafout$ where we identify $\alphagafin$ with a vector in $\Dset^{\numberofpoints(\alphagafout)}$ by taking the points in uniform random order.
This distribution is supported on the set $\Sigma_{\usm}$ where $m=\numberofpoints(\alphagafout)$, $\us=\constraintexternal(\alphagafout)$.
This distribution has a density with respect to the Lebesgue measure on $\Sigma_{\usm}$ (see Theorem~\ref{thm:app:alphagaf} in Appendix~\ref{a:agaf} for reference). 

\begin{theorem}[The generalized Gibbsian structure of the $\alpha$-GAF zero ensembles]
\label{thm:alphagaf-main}
There exist \emph{positive} quantities 
$\mm\fbr[\big]{\alphagafout}$ and $\MM\fbr[\big]{\alphagafout}$, 
measurable with respect to $\alphagafout$, 
such that $\alphagafout$-a.s.\ we have:
\begin{equation}\label{eq:target:GAF}
\mm\fbr[\big]{\alphagafout} 
\abs*{\van{\uz}}^2 
\le 
\frac{\deri\frho{\cdot}{\alphagafout}}{\deri\el}(\uz)
\le
\MM\fbr[\big]{\alphagafout} 
\abs*{\van{\uz}}^2 
\end{equation}
for a.e.\ $\uz$ with respect to
    the measure $\el$ on $\Sigma_{\usm}$,
    where $m=\numberofpoints(\alphagafout)$,
    and $\us=\constraintexternal(\alphagafout)$.
In other words, 
    the generalized Gibbs property is satisfied
    with the potentials $(\Phi,\Psi)$ given by
    $\Phi\equiv 0$ and $\Psi(z)=-2\log|z|$.
\end{theorem}

\begin{remark}\label{remark:disk}
It was shown in \cite{GP} that for proving the rigidity and tolerance of the Ginibre ensemble and the ensemble of the roots of standard GAF, it is enough to consider $\Dset$ to be a disk centered at the origin. This is also true for establishing the generalized Gibbs property as in Theorems~\ref{thm:ginibre-main} and \ref{thm:alphagaf-main}. We show this in Appendix~\ref{a:disk}. Therefore, in the proofs of Theorems~\ref{thm:ginibre-main} and \ref{thm:alphagaf-main} we will assume that $\Dset$ is a disk centered at the origin. We will denote the radius of the disk by $r_0$.
\end{remark}

As an immediate consequence of Theorem~\ref{thm:alphagaf-main} we get the following.

\begin{corollary}\label{cor:energy}
For a point configuration $\uz=(\zeta_1,\ldots,\zeta_m)\in\Dset^m$, define the logarithmic energy as
\[
\mathcal{E}_{\log}\fbr*{\uz} \coloneqq \sum_{i\neq j} \log\frac{1}{\abs*{\zeta_i-\zeta_j}}\;.
\]
This is well-defined for a.e.\ $\uz$ with respect to the Lebesgue measure on $\Dset^m$.
Then, for $\alphagafout$-a.s.\ the following is true. 
Let $m = \numberofpoints(\alphagafout)$ and $\us = \constraintexternal(\alphagafout)$. 
Let $\varrho(\cdot;\alphagafout)$ be the density of the conditional measure of $\alphagafin$ given $\alphagafout$ with respect to the Lebesgue measure on $\Sigma_{\usm}$. 
Then for any two configurations $\uz,\uzp\in\Sigma_{\usm}$ we have 
\[
\abs*{ \log\varrho\fbr*{\uz;\alphagafout} - \log\varrho\fbr*{\uzp;\alphagafout} }
\leq
2\abs*{\mathcal{E}_{\log}(\uz) - \mathcal{E}_{\log}(\uzp)} + C
\]
where $C$ depends only on the conditioning $\alphagafout$ configuration and not on $\uz,\uzp$. 
\end{corollary}

\section{Proof of Theorem~\ref{thm:abscont}}\label{sec:pf:thm:abscont}

In this section our objective is to prove Theorem~\ref{thm:abscont}.
We need to prove \eqref{targeteq} from \eqref{eq:main}.
We rewrite \eqref{eq:main} as
\begingroup
\addtolength{\jot}{0.25em}
\begin{align*}
& \Prob[\Big]{\,
    \fbr[\big]{\,\xinnk\in A\,}
    \,\cap\, 
    \fbr[\big]{\,\xoutnk\in B\,} 
    \,\cap\,\onkj\,}\\
    \overset{j}{\scaleto{\asymp}{8pt}}\; 
& \bigg(\,
    \int_{\xoutnk^{-1}(B)\;\cap\;\onkj}
    \nu\kernel{A}{\xout}
    \deri\PROB
    \,\bigg)
\,+\, 
\vartheta(j,k)\;.
\numberthis
\label{abscond}
\end{align*}
\endgroup
First, we propose a sufficient criterion for \eqref{targeteq} to hold. 
From \eqref{abscond} we get that there exists functions $L,U:\NN\to(0,\infty)$ 
such that for $A\in\FUBasisInsidempts$, $B\in\FUBasisOutside$, $j\geq 1$,
$k\geq 1$
\begin{equation}\label{eq:AC}
L(j)\leq 
\frac{\displaystyle
    \Prob[\Big]{\,
     \fbr[\big]{\,\xinnk\in A\,}
     \,\cap\, 
     \fbr[\big]{\,\xoutnk \in B\,} 
     \,\cap\,
     \onkj
     \,}}
    {\displaystyle\fbr[\Big]{\,
    \int_{\xoutnk^{-1}(B)\;\cap\;\onkj}
    \nu\kernel{A}{\xout}\deri\PROB
 }
 + \vartheta(j,k)}
 \leq U(j)\;.
\end{equation}
Recall that $\oj$ is measurable with respect 
    to $\xout$. Let $\ej\subset\cS{\Dset^\c}$ be such that $\xout^{-1}(\ej)=\oj$.
The sufficient criterion is the following. 
\begin{lemma}\label{Lemma4.1}
If for all $j\in\NN$, 
$A\in\FUBasisInsidempts$, 
and $B\in\BorelOutside$ we have
\begin{equation}\label{abs1} 
L(j)\leq 
\frac{\displaystyle
\Prob[\Big]{\,
    \fbr[\big]{\,\xin\in A\,}
    \,\cap\, 
    \fbr[\big]{\,\xout\in B\,}
    \,\cap\,
    \oj
    \,}}
{\displaystyle
\int_{\xout^{-1}(B)\;\cap\;\oj}
\nu\kernel{A}{\xout}\deri\PROB}
\leq U(j)\;,
\end{equation}
then \eqref{targeteq} holds with
\[
\mm(\xout)=L(\mathscr{J}(\xout)), \quad  
\MM(\xout)=U(\mathscr{J}(\xout)), \quad 
\mathscr{J}(\xout)=\inf\Set*{j\geq 1\given\xout\in\ej}\;.
\]
\end{lemma}

\begin{proof}[\textbf{Proof of Theorem~\ref{thm:abscont}}]
By Lemma~\ref{Lemma4.1} it is enough to prove \eqref{abs1} for all 
$j\in\NN$, 
$A\in\FUBasisInsidempts$, 
and 
$B\in\BorelOutside$.
Let us fix $j\in\NN$, 
$A\in\FUBasisInsidempts$, 
$B\in\BorelOutside$.
Since $A\in\FUBasisInsidempts$, we have
\begin{equation}\label{eq:pre-dct-A}
\lim_{k\to\infty}
\bigindicator{\xinnk\in A}
=
\bigindicator{\xin\in A}\quad\mbox{a.s.}
\end{equation}
Given any $\epsilon>0$ there exists 
a set $B_\epsilon\in\FUBasisOutside$
such that 
\begin{equation}\label{eq:tildeB1}
\Prob[\bigg]{\,
        \fbr[\Big]{\,
            \fbr[\big]{\,\xout\in B\,}
            \cap\,\oj\,
        }
    \;\symmdiff\;
    \fbr[\big]{\,\xout\in B_\epsilon\,}
    \,} < \epsilon\;.
\end{equation}
Here we use the fact that $\oj$ is measurable with respect to $\xout$.
This step also justifies Remark~\ref{rmabs-1}.
The set $B_\epsilon$ depends on $B$, $\epsilon$, and $j$.
But since we are treating $j$ as fixed, we suppress it from the notation.
Since $B_\epsilon\in\FUBasisOutside$, 
we have
\begin{equation}\label{eq:tildeB}
\lim_{k\to\infty}
\bigindicator{\xoutnk\in B_\epsilon}
=
\bigindicator{\xout\in B_\epsilon}\quad\mbox{a.s.}
\end{equation}
We start from \eqref{eq:AC} 
applied to $A$, $B_\epsilon$, $j$.
That is,
\begin{equation}\label{eq:abscond2.1}
L(j)\leq 
\frac{\displaystyle\Prob[\Big]{\,
    \fbr[\big]{\,\xinnk\in A\,}
    \,\cap\, 
    \fbr[\big]{\,\xoutnk \in B_\epsilon\,} 
    \,\cap\,
    \onkj
    \,}}
{\displaystyle\fbr[\Big]{\,
    \int_{\xoutnk^{-1}(B_\epsilon)\;\cap\;\onkj}
    \nu\kernel{A}{\xout}\deri\PROB
}
+ \vartheta(j,k)}
\leq U(j)\;.
\end{equation}
We want to derive \eqref{abs1} for $A$, $B$, $j$.
We begin with the numerator of the term in the middle of \eqref{eq:abscond2.1}.
Using \eqref{eq:pre-dct-A} and \eqref{eq:tildeB} we get
\begingroup
\addtolength{\jot}{0.25em}
\begin{align*}
& \abs[\bigg]{\Prob[\Big]{\,\fbr[\big]{\,\xinnk\in A\,} 
\,\cap\,\fbr[\big]{\,\xoutnk\in B_\epsilon\,} 
\,\cap\,\onkj\,}
- 
\Prob[\Big]{\,\fbr[\big]{\,\xin\in A\,}
\,\cap\,\fbr[\big]{\,\xout\in B_\epsilon\,} 
\,\cap\,\onkj\,}}\\
\leq{} & 
\Prob[\bigg]{\,
\fbr[\Big]{\,\fbr[\big]{\,\xinnk\in A\,}
\,\cap\,\fbr[\big]{\,\xoutnk\in B_\epsilon\,}
\,}
\triangle 
\fbr[\Big]{\,\fbr[\big]{\,\xin\in A\,}
\,\cap\,\fbr[\big]{\,\xout\in B_\epsilon\,} 
}\,}\\
={} & \okbe\;,
\numberthis
\label{eq:mm1}
\end{align*}
\endgroup
where by $\okbe$ we mean a term which goes to $0$ as $k\to\infty$ for every fixed $\epsilon$. Using \eqref{eq:tildeB1} we get
\begingroup
\addtolength{\jot}{0.25em}
\begin{align*}
& \abs[\bigg]{
\Prob[\Big]{\,\fbr[\big]{\,\xin\in A\,}
\,\cap\,\fbr[\big]{\,\xout\in B_\epsilon\,} 
\,\cap\,\onkj\,}
- 
\Prob[\Big]{\,\fbr[\big]{\,\xin\in A\,} 
\,\cap\,\fbr[\big]{\,\xout\in B\,} 
\,\cap\,\oj 
\,\cap\,\onkj\,}
}\\
\leq{} & 
\Prob[\bigg]{\,\fbr[\Big]{\,\fbr[\big]{\,\xout\in B\,}\,\cap\,\oj\,}
\,\symmdiff\,
\fbr[\big]{\,\xout\in B_\epsilon\,}\,}
\\
={} & \oep\;,
\numberthis
\label{eq:mm2}
\end{align*}
\endgroup
where $\oep$ denotes a term which goes to $0$ as $\epsilon$ goes to 
$0$. Combining \eqref{eq:mm1} and \eqref{eq:mm2} we get
\begingroup
\addtolength{\jot}{0.25em}
\begin{align*}
& \Prob[\Big]{\,\fbr[\big]{\,\xinnk\in A\,} 
\,\cap\,\fbr[\big]{\,\xoutnk\in B_\epsilon\,} 
\,\cap\,\onkj\,}\\
={} & \Prob[\Big]{\,\fbr[\big]{\,\xin\in A\,}
\,\cap\,\fbr[\big]{\,\xout\in B_\epsilon\,} 
\,\cap\,\onkj\,}
+ \okbe\\
={} & \Prob[\Big]{\,\fbr[\big]{\,\xin\in A\,} 
\,\cap\,\fbr[\big]{\,\xout\in B\,} 
\,\cap\,\oj 
\,\cap\,\onkj\,}  
+ \oep 
+ \okbe\;.
\numberthis
\label{eq:r1}
\end{align*}
\endgroup
Recall from the statement of this theorem that
\begin{equation}\label{eq:b1}
\oj\subset\liminf_{k \to \infty}\onkj\;.
\end{equation}
Also observe that
\begin{equation}\label{eq:b2}
\Prob[\Big]{\,
\fbr[\big]{\,\liminf_{l\to\infty}\onlj\,}
\,\symmdiff\,
\fbr[\big]{\,\cap_{l\ge k}\onlj\,}\,}
=\ok\;,
\end{equation}
where $\ok$ denotes a quantity that tends to $0$ as $k\to\infty$.
Therefore using \eqref{eq:b1} and \eqref{eq:b2}
\begingroup
\addtolength{\jot}{0.25em}
\begin{align*}
& \Prob[\Big]{\,\fbr[\big]{\,\xin\in A\,}
\,\cap\, \fbr[\big]{\,\xout\in B\,} 
\,\cap\, \oj 
\,\cap\, \onkj\,}\\ 
={} & \Prob[\Big]{\,\fbr[\big]{\,\xin\in A\,}
\,\cap\, \fbr[\big]{\,\xout\in B\,} 
\,\cap\, \oj 
\,\cap\, \onkj 
\,\cap\, \fbr[\big]{\,\liminf_{l \to \infty}\onlj\,}\,}\\
={} & \Prob[\Big]{\,\fbr[\big]{\,\xin\in A\,}
\,\cap\, \fbr[\big]{\,\xout\in B\,} 
\,\cap\, \oj 
\,\cap\, \onkj 
\,\cap\, \fbr[\big]{\,\cap_{l \ge k} \onlj\,} 
\,} + \ok\\
={} & \Prob[\Big]{\,\fbr[\big]{\,\xin\in A\,} 
\,\cap\,\fbr[\big]{\,\xout\in B\,}
\,\cap\,\oj 
\,\cap\,\fbr[\big]{\,\cap_{l \ge k} \onlj\,} 
\,} + \ok\\
={} & \Prob[\Big]{\,\fbr[\big]{\,\xin\in A\,} 
\,\cap\,\fbr[\big]{\,\xout\in B\,}
\,\cap\,\oj 
\,\cap\,\fbr[\big]{\,\liminf_{l \to \infty}\onlj\,} 
\,}  + \ok\\ 
={} & \Prob[\Big]{\,\fbr[\big]{\,\xin\in A\,} 
\,\cap\,\fbr[\big]{\,\xout\in B\,}
\,\cap\,\oj 
\,}  + \ok\;.
\numberthis\label{eq:r2}
\end{align*}
\endgroup
Combining \eqref{eq:r1} and \eqref{eq:r2} we get
\begingroup
\addtolength{\jot}{0.25em}
\begin{align*}
 & \Prob[\Big]{\,\fbr[\big]{\,\xinnk\in A\,} 
\,\cap\,\fbr[\big]{\,\xoutnk\in B_\epsilon\,}
\,\cap\,
\onkj\,}\\
={} & \Prob[\Big]{\,\fbr[\big]{\,\xin\in A\,}  
\,\cap\,
\fbr[\big]{\,\xout\in B\,} 
\,\cap\, 
\oj\,} 
+ \okbe + \ok + \oep\;.\numberthis\label{reduction} 
\end{align*}
\endgroup
Now we will carry out a similar procedure for 
the denominator of the term in the middle of \eqref{eq:abscond2.1}. 
Using \eqref{eq:tildeB} we get
\begingroup
\addtolength{\jot}{0.25em}
\begin{align*}
& \abs[\bigg]{
\int_{\xoutnk^{-1}(B_\epsilon)\;\cap\;\onkj}
\nu\kernel{A}{\xout}\deri\PROB
- 
\int_{\xout^{-1}(B_\epsilon)\;\cap\;\onkj}
\nu\kernel{A}{\xout}\deri\PROB
}\\
\leq{} & 
\int\abs[\Big]{
\bigindicator{\xoutnk^{-1}(B_\epsilon)}
\bigindicator{\onkj}
\nu\kernel{A}{\xout}
-
\bigindicator{\xout^{-1}(B_\epsilon)}
\bigindicator{\onkj}
\nu\kernel{A}{\xout}}
\deri\PROB
\\
\leq{} & 
\int\abs[\Big]{
\bigindicator{\xoutnk^{-1}(B_\epsilon)}
-
\bigindicator{\xout^{-1}(B_\epsilon)}
}
\deri\PROB
\\
={} & \okbe\;,
\numberthis
\label{eq:mm4}
\end{align*}
\endgroup
Using \eqref{eq:tildeB1} we get
\begingroup
\addtolength{\jot}{0.25em}
\begin{align*}
& \abs[\bigg]{
\int_{\xout^{-1}(B_\epsilon)\;\cap\;\onkj}
\nu\kernel{A}{\xout}\deri\PROB
- 
\int_{\xout^{-1}(B)\;\cap\;\onkj\;\cap\;\oj}
\nu\kernel{A}{\xout}\deri\PROB}\\
\leq{} & 
\int_{(\xout^{-1}(B_\epsilon))\;\symmdiff\;(\xout^{-1}(B)\;\cap\;\oj)}
\nu\kernel{A}{\xout}\deri\PROB\\
\leq{} & 
\Prob[\bigg]{\,
    \fbr[\Big]{\,\fbr[\big]{\,\xout\in B\,}\,\cap\,\oj\,}
    \,\symmdiff\,
    \fbr[\big]{\,\xout\in B_\epsilon\,}\,}
\\
={} & \oep\;,
\numberthis
\label{eq:mm5}
\end{align*}
\endgroup
Combining \eqref{eq:mm4} and \eqref{eq:mm5} we get
\begingroup
\addtolength{\jot}{0.25em}
\begin{align*}
& \int_{\xoutnk^{-1}(B_\epsilon)\;\cap\;\onkj}
\nu\kernel{A}{\xout}\deri\PROB
\\
={} & \int_{\xout^{-1}(B_\epsilon)\;\cap\;\onkj}
\nu\kernel{A}{\xout}\deri\PROB
+ \okbe\\
={} & \int_{\xout^{-1}(B)\;\cap\;\onkj\;\cap\;\oj}
\nu\kernel{A}{\xout}\deri\PROB
+ \oep 
+ \okbe\;.
\numberthis
\label{eq:r11}
\end{align*}
\endgroup
Using \eqref{eq:b1} and \eqref{eq:b2} we get
\begingroup
\addtolength{\jot}{0.25em}
\begin{align*}
& \int_{\xout^{-1}(B)\;\cap\;\onkj\;\cap\;\oj}
\nu\kernel{A}{\xout}\deri\PROB\\ 
={} & 
\int_{\xout^{-1}(B)\;\cap\;\onkj\;\cap\;\oj\;\cap\;\liminf_{l\to\infty}\onlj}
\nu\kernel{A}{\xout}\deri\PROB\\
={} & 
\int_{\xout^{-1}(B)\;\cap\;\onkj\;\cap\;\oj\;\cap_{l \ge k} \onlj}
\nu\kernel{A}{\xout}\deri\PROB + \ok\\
={} & 
\int_{\xout^{-1}(B)\;\cap\;\oj\;\cap_{l \ge k} \onlj}
\nu\kernel{A}{\xout}\deri\PROB + \ok\\
={} & 
\int_{\xout^{-1}(B)\;\cap\;\oj\;\cap\;\liminf_{l\to\infty}\onlj}
\nu\kernel{A}{\xout}\deri\PROB + \ok\\
={} & \int_{\xout^{-1}(B)\;\cap\;\oj}
\nu\kernel{A}{\xout}\deri\PROB + \ok\;.
\numberthis\label{eq:r21}
\end{align*}
\endgroup
Combining \eqref{eq:r11} and \eqref{eq:r21} we get
\begingroup
\addtolength{\jot}{0.25em}
\begin{align*}
 & \int_{\xoutnk^{-1}(B_\epsilon)\;\cap\;\onkj}
\nu\kernel{A}{\xout}\deri\PROB\\
={} & \int_{\xout^{-1}(B)\;\cap\;\oj}
\nu\kernel{A}{\xout}\deri\PROB
+ \okbe + \ok + \oep\;.\numberthis\label{reduction2} 
\end{align*}
\endgroup
Thus, combining \eqref{reduction},
\eqref{reduction2}, and \eqref{eq:abscond2.1}
we get
\begingroup
\addtolength{\jot}{0.25em}
\begin{align*}
L(j) \leq 
\frac{\displaystyle\Prob[\Big]{\,\fbr[\big]{\,\xin\in A\,} 
\,\cap\,\fbr[\big]{\,\xout\in B\,} 
\,\cap\,\oj\,} + 
+ \oep 
+ \okbe 
+ \ok
}
{\displaystyle\fbr*{\int_{(\xout)^{-1}(B)\;\cap\;\oj} 
\nu\kernel{A}{\xout}\deri\PROB} 
+ \oep 
+ \okbe 
+ \ok
+ \correctionterm{j}{k}}
\leq U(j)\;.
\numberthis\label{eq:lp}
\end{align*}
\endgroup
Recall from the statement of this theorem that 
$\correctionterm{j}{k} \to 0$ as $k \to \infty$. 
First letting $k\to\infty$ in \eqref{eq:lp}, 
with $\epsilon$ held fixed,
and then letting $\epsilon\to 0$ we get
\[
L(j) 
\leq 
\frac{\displaystyle\Prob[\Big]{\,\fbr[\big]{\,\xin\in A\,}\,\cap\,\fbr[\big]{\,\xout\in B\,}\,\cap\,\oj\,}}
{\displaystyle\int_{(\xout)^{-1}(B)\;\cap\;\oj}\nu\kernel{A}{\xout}\deri\PROB}
\leq 
U(j)\;.
\]
This concludes the proof of Theorem~\ref{thm:abscont}.
\end{proof}

Now it remains to prove Lemma~\ref{Lemma4.1} for which we will need another lemma.

\begin{lemma}\label{lemma:meas1}
Let $(\mathcal{A},\mathcal{T})$ be a second countable topological space. Let $\mathcal{B}$ be a countable basis of open sets. Let $\mathcal{F}\coloneqq\Set*{\cup_{i=1}^k B_i \given B_i\in\mathcal{B} \mbox{ for all } 1\leq i\leq k, k \ge 1}$. Let $\mu_1$ and $\mu_2$ be two non-negative, regular, Borel measures on $(\mathcal{A},\mathcal{T})$. If for some constant $c>0$, $\mu_1(B)\leq c\mu_2(B)$ for all $B\in\mathcal{F}$, then $\mu_1(B)\leq c \mu_2(B)$ for all $B\in\mathcal{T}$.
\end{lemma}

Now we prove Lemma~\ref{Lemma4.1} using Lemma~\ref{lemma:meas1}.

\begin{proof}[\textbf{Proof of Lemma~\ref{Lemma4.1}}]
For $j\geq 1$,
    let $\xo(j)$ be the $\sigma$-algebra
    formed by intersecting the sets of $\BorelOutside$ 
    with $\ej$. 
Consider the finite measure space
    $\fbr*{\ej,\xo(j),\PROB_{\infty,\out}}$.
For all $A\in\FUBasisInsidempts$ and 
    $B\in\BorelOutside$ we have 
\[
\Prob[\Big]{\,
    \fbr[\big]{\,\xin\in A\,}
    \,\cap\,\fbr[\big]{\,\xout\in B\,}
    \,\cap\,\oj\,} 
    =
    \int_{B\;\cap\;\ej}
    \frho{A}{\Upsilon}\deri\PROB_{\infty,\out}[\Upsilon]\;.
\]
Thus, from \eqref{abs1} we get that 
    for all $A\in\FUBasisInsidempts$ 
    and $B\in\BorelOutside$
    \[
    L(j)\leq 
    \frac{\displaystyle
        \int_{B\;\cap\;\ej}
        \frho{A}{\Upsilon}
        \deri\PROB_{\infty,\out}[\Upsilon]}
    {\displaystyle
    \int_{B\;\cap\;\ej}
    \nu\kernel{A}{\Upsilon}
    \deri\PROB_{\infty,\out}[\Upsilon]}
    \leq U(j)\;.
    \]
Thus, for all $A\in\FUBasisInsidempts$ we get   
    \[
    L(j)
    \leq
    \frac{\frho{A}{\Upsilon}}{\nu\kernel{A}{\Upsilon}}
    \leq 
    U(j)
    \numberthis\label{eq:lubd}
    \]
    for $\PROB_{\infty,\out}$-a.s.\ $\Upsilon\in\ej$. 
    Since $\FUBasisInsidempts$ is countable, 
    we get (\ref{eq:lubd}) holds for $\PROB_{\infty,\out}$-a.s. $\Upsilon\in\ej$,   
    and for all $A\in\FUBasisInsidempts$.
    Using Lemma~\ref{lemma:meas1} on the measure space $\fbr*{\ej,\xo(j),\PROB_{\infty,\out}}$ we have (\ref{eq:lubd}) holds for $\PROB_{\infty,\out}$-a.s. $\Upsilon$, for all $A\in\BasisInsidempts$. 
    Therefore, 
    we get \eqref{targeteq} with 
    $\mm(\xout)=L(\mathscr{J}(\xout))$ and  
    $\MM(\xout)=U(\mathscr{J}(\xout))$ where
    $\mathscr{J}(\xout)=\inf\Set*{j\geq 1\given\xout\in\ej}$.
\end{proof}

Now we prove Lemma~\ref{lemma:meas1}.

\begin{proof}[\textbf{Proof of Lemma~\ref{lemma:meas1}}]
Suppose $\mu_1(B)\leq c\mu_2(B)$
for all $B\in\mathcal{F}$.
Since $\mathcal{B}$ is a basis,
all $U\in\mathcal{T}$ are
countable union of sets in $\mathcal{B}$.
Therefore, $\mu_1(U)\leq c\mu_2(U)$ for all $U\in\mathcal{T}$.
Therefore, for any Borel set $B$ and an open set $U$ containing $B$ 
we have $\mu_1(B)\leq \mu_1(U)\leq c \mu_2(U)$.
Since $\mu_2$ is a regular measure, 
for any Borel set $B$, 
we have $\mu_2(B)=\inf\Set*{\mu_2(V)\given B\subset V, V\in\mathcal{T}}$. 
Therefore, we get $\mu_1(B)\leq c\mu_2(B)$. 
\end{proof}

\section{The approximate Gibbsian structure of the Ginibre Ensemble}\label{sec:ginibre}

In this section our objective is to prove Theorem~\ref{thm:ginibre-main}.
Recall from Remark~\ref{remark:disk} that we assume $\Dset$ to be a disk of radius $r_0$ centered at the origin.
Recall that $\frho{\cdot}{\GIFOUT}$ is 
    the conditional measure of $\GIFIN$ given $\GIFOUT$
    where we identify $\GIFIN$ with a vector in 
    $\Dset^{\numberofpoints(\GIFOUT)}$.
Also recall that $\om$ is the event that $\numberofpoints(\GIFOUT)=m$.    
For each $m\in\NN$ we have $\Prob{\,\om\,}>0$, 
    because for the ensemble $\GIF$ 
    the number of points in a domain 
    is a sum of Bernoulli random variables, 
    each with success probability strictly 
    between $0$ and $1$. 
Our objective is to show that $\GIFOUT$-a.s.\ 
    $\frho{\cdot}{\GIFOUT}$ has a density 
    with respect to the Lebesgue measure on $\Dset^{\numberofpoints(\GIFOUT)}$, 
    and this density satisfies \eqref{eq:target2}.
The fact that this density exists is already known from \cite{GP}.
It is also known from \cite{GP} 
    that a bound similar to \eqref{eq:target2} 
    but without the Vandermonde terms holds. 
We want to show that the same procedure yields the stronger 
    bound in \eqref{eq:target2} by using Theorem~\ref{thm:abscont}.

\begin{notation}\label{n:projection}
Let $\aspc{\cdot}$ be the map from $\sqcup_{m=1}^\infty\CC^m\to\cS{\CC}$ which takes a vector (of variable length) to the point configuration on $\CC$ consisting of the coordinates of the vector. If two coordinates of the vector are same, we do not distinguish them in the point configuration.
\end{notation}    

Consider the potentials $(\Phi,\Psi)$ given by $\Phi\equiv 0$ and $\Psi(z)=-2\log|z|$.
The corresponding probability kernel $\nu_{\Phi,\Psi,\Dset}$ is as follows.
Given $\Upsilon\in\cS{\Dset^\c}$, 
the measure $\nu_{\Phi,\Psi,\Dset}\kernel{\cdot}{\Upsilon}$ is supported on 
    point configurations having $m=\numberofpoints(\Upsilon)$ number of points,
    and for $A\in\BorelInsidempts$ 
    we have 
\begin{equation}\label{eq:ginicandidate}
\nu_{\Phi,\Psi,\Dset}\kernel{A}{\Upsilon} = 
    \frac{\displaystyle\int_{\asvector{A}}
    \abs*{\van{\uz}}^2\deri\el(\uz)}
    {\displaystyle\int_{\Dset^m} 
    \abs*{\van{\uz}}^2\deri\el(\uz)}\;,
\end{equation}
where $\el$ is the Lebesgue measure on $\Dset^m$. 
Recall that $\PROB_{\Dset|\Dset^\c}[\cdot|\Upsilon]$ is supported on 
point configurations having $m$ points, and for $A\in\BorelInsidempts$
\[
\PROB_{\Dset|\Dset^\c}[A|\Upsilon] = \frho{\asvector{A}}{\Upsilon}\;. 
\]
If we can establish the generalized Gibbs property, then we get that for all $m\in\NN$ and all $A\in\BorelInsidempts$
    \[
    \mm_1\fbr[\big]{\GIFOUT}\nu_{\Phi,\Psi,\Dset}\kernel{A}{\GIFOUT} 
    \;\le\; 
    \PROB_{\Dset|\Dset^\c}[A|\GIFOUT] 
    \;\le\; 
    \MM_1\fbr[\big]{\GIFOUT}\nu_{\Phi,\Psi,\Dset}\kernel{A}{\GIFOUT}\;,  
    \]
    for some measurable functions $\mm_1,\MM_1:\cS{\Dset^\c}\to(0,\infty)$. 
    Then, by the Radon-Nikodym Theorem we will have
    \[
    \mm_2\fbr[\big]{\GIFOUT}\abs*{\van{\uz}}^2  
    \;\le\; 
    \frac{\deri\frho{\cdot}{\GIFOUT}}{\deri\el}(\uz) 
    \;\le\; 
    \MM_2\fbr[\big]{\GIFOUT}\abs*{\van{\uz}}^2\;,
    \]
for some measurable functions $\mm_2,\MM_2:\cS{\Dset^\c}\to(0,\infty)$. Thus we get \eqref{eq:target2}. Consider the finite dimensional approximations of the Ginibre ensemble $\seq*{\Gini_n}$ which converge to $\GIF$ a.s.\ (see Appendix~\ref{a:ginibre} for reference). Due to Theorem~\ref{thm:abscont}, it is enough to verify the conditions of approximate Gibbsianity for $\fbr*{\seq*{\Ginin}_{n=1}^\infty,\GIF}$ with respect to $\nu_{\Phi,\Psi,\Dset}$ on the domain $\Dset$. We will not present the proof in full detail because we will see that the procedure that was used in \cite{GP} for proving tolerance of the Ginibre ensemble also yields the approximate Gibbsian structure. 

\subsection{The limiting procedure for the Ginibre ensemble}\label{ss:ginilimiting}

We will present the proof in three steps. Here we outline the three steps. For $n\geq 1$, let $\Gininin\coloneqq\Ginin\cap\Dset$, $\Gininout\coloneqq\Ginin\cap\Dset^\c$.

\begin{enumerate}[label = Step~\arabic*, font = \normalfont\bfseries,  
    labelsep = \parindent,
    itemindent = 2\parindent,
    leftmargin = 2\parindent,
    topsep = 0pt]
    
    \item\label{c:gini:1} In this step we will define the sequence of events $\seq*{\oj}_{j=1}^{\infty}$.
    We will also define, for each $j\geq 1$, the sequence of positive integers $\seq*{n_k}_{k=1}^{\infty}$ and the sequence of events $\seq*{\onkj}_{k=1}^\infty$. 
    We will verify the following conditions:
    \begin{enumerate}[(i),leftmargin=*,font = \normalfont\bfseries] 
        \item\label{c:gini:1:1} $\oj\subset\Omega(j+1)$ for all $j\geq 1$;
        \item\label{c:gini:1:2} $\oj\subset\om$ for all $j\geq 1$;
        \item\label{c:gini:1:3} $\onkj\subset\omnk$ for all $j\geq 1$ and $k\geq 1$;
        \item\label{c:gini:1:4} $\onkj$ is measurable with respect to $\Gininkout$ for all $j\geq 1$ and $k\geq 1$;
        \item\label{c:gini:1:5} $\oj\subset\liminf_{k\to\infty} \onkj$ for all $j\geq 1$.
    \end{enumerate}
    The event $\oj$ is not going to be measurable with respect to $\GIFOUT$.
    We will define another event $\ocj$, which is measurable with respect to $\GIFOUT$, and which satisfies $\Prob{\,\oj\,\symmdiff\,\ocj\,}=0$.
    Here we utilize Remark~\ref{rmabs-1}.
    
    \item\label{c:gini:2} In this step we will establish that \eqref{eq:main} for some $\seq*{\correctionterm{j}{k}}_{k\geq 1,j\geq 1}$ satisfying $\lim_{k\to\infty}\correctionterm{j}{k}=0$ for all $j\geq 1$. Recall that the probability kernel $\nu$ is given by \eqref{eq:ginicandidate}. 
    
    \item\label{c:gini:3} In this step we will show that $\lim_{j\to\infty}\Prob{\,\om\,\setminus\,\oj\,}=0$.
\end{enumerate}

\subsubsection{Step 1}

\begin{notation}\label{n:ginisum}
For $n\in\NN$ let 
\[
\ginisumone(n)\coloneqq\sum_{\omega\in\Gininout}\frac{1}{\omega}\;,
\qquad
\ginisumtwo(n)\coloneqq\sum_{\omega\in\Gini_{n,\out}}\frac{1}{\omega^2}\;,
\qquad
\ginisumthree(n)\coloneqq\sum_{\omega\in\Gini_{n,\out}}\frac{1}{|\omega|^3}\;.
\]
Let 
\[
\X_n\coloneqq\abs*{\ginisumone(n)}+\abs*{\ginisumtwo(n)}+\ginisumthree(n)\;. 
\]
\end{notation}

\begin{definition}[The events $\OGm{n}{\thetagap}$ and $\OGm{\infty}{\thetagap}$]\label{d:gap:gini}
For $n,m\in\NN$, and $\thetagap\in(0,1)$, let $\OGm{n}{\thetagap}$ be the event that: 
\begin{enumerate}[(i),font=\normalfont\bfseries,topsep=0pt] 
    \item the number of points in $\Gininin$ is $m$;
    \item and there is a gap of at least $\thetagap$ between the boundary of $\Dset$ and the points of $\Gininout$ i.e., for all $x\in\partial\Dset$ and $y\in\Gininout$, $\abs*{x-y}\geq\thetagap$.
\end{enumerate}
Let $\OGm{\infty}{\thetagap}$ be the analogous event for $\GIF$. 
The event $\OGm{n}{\thetagap}$ is a subset of the event $\omn$ and is measurable with respect to $\Gininout$.
Similarly, the event $\OGm{\infty}{\thetagap}$ is a subset of the event $\om$ and is measurable with respect to $\GIFOUT$.
\end{definition}

In \cite{GP} (see Appendix~\ref{a:ginibre-estimates} for reference), it was shown that there exists random variables $\ginisumone$, $\ginisumtwo$, $\ginisumthree$, such that 
\[
\ginisumone(n)\xrightarrow[]{}\ginisumone\;,\quad
\ginisumtwo(n)\xrightarrow[]{}\ginisumtwo\;,\quad
\ginisumthree(n)\xrightarrow[]{}\ginisumthree\quad\mbox{in probability.}
\]
Thus, we get a sequence $\seq*{n_k}_{k=1}^\infty$ such that 
\[
\ginisumone(n_k)\xrightarrow[]{}\ginisumone\;,\quad
\ginisumtwo(n_k)\xrightarrow[]{}\ginisumtwo\;,\quad 
\ginisumthree(n_k)\xrightarrow[]{}\ginisumthree\quad\mbox{a.s.} 
\]
Further, we choose an increasing sequence $\seq*{M_j}_{j=1}^\infty$ diverging to $\infty$ such that none of the $M_j$-s is an atom of the distributions of $\abs{\ginisumone}$, $\abs{\ginisumtwo}$, $\ginisumthree$. We also choose a sequence $\seq*{\thetagap_j}_{j=1}^\infty$ such that $\thetagap_j\geq\thetagap_{j+1}>0$ for all $j\in\NN$ and $\thetagap\to 0$ as $j\to\infty$. Now we define the events $\seq*{\onkj}_{j\geq 1, k\geq 1}$.

\begin{definition}[The event $\onkj$]
For $j\geq 1$ and $k\geq 1$, let $\onkj$ be the
event in which all of the following conditions are
satisfied: 
\begin{enumerate}[(i),font=\normalfont\bfseries,topsep=0pt]
\item $\OGm{n_k}{\thetagap_j}$ occurs; 
\item 
$\abs{\ginisumone(n_k)}\leq M_j$, 
$\abs{\ginisumtwo(n_k)}\leq M_j$, 
$\ginisumthree(n_k)\leq M_j$.
\end{enumerate}
\end{definition}

The event $\onkj$ is measurable with respect to $\Gininkout$, as required in condition~\ref{c:gini:1:4} in \ref{c:gini:1}.
Also, the event $\onkj$ is a subset of the event $\omnk$, as required in condition~\ref{c:gini:1:3} in \ref{c:gini:1}.

Next, we define the events $\seq*{\oj}_{j\geq 1}$.

\begin{definition}[The event $\oj$]
For $j\geq 1$, let $\oj\coloneqq\liminf_{k\to\infty}\onkj$. Thus, $\oj$ is the event in which all of the following conditions are satisfied:
\begin{enumerate}[(i),font=\normalfont\bfseries,topsep=0pt]
\item $\OGm{n_k}{\thetagap_j}$ occurs for all large enough (random) $k$;
\item for all large enough (random) $k$, 
$|\ginisumone(n_k)|\leq M_j$, 
$|\ginisumtwo(n_k)|\leq M_j$, 
$\ginisumthree(n_k)\leq M_j$.
\end{enumerate}
\end{definition}

Thus, condition~\ref{c:gini:1:5} in \ref{c:gini:1} is satisfied by definition. 
And since $\seq*{M_j}_{j=1}^\infty$ is increasing, condition~\ref{c:gini:1:1} in \ref{c:gini:1} is also satisfied.
Since $\Ginin\to\GIF$ a.s., we have $\liminf_{k\to\infty}\omnk\subset\om$.
Since each $\onkj$ is a subset of $\omnk$, we have $\oj\subset\om$, as required by condition~\ref{c:gini:1:2} in \ref{c:gini:1}.

The event $\oj$ is not measurable with respect to $\GIFOUT$. So we construct another event $\ocj$, which is measurable with respect to $\GIFOUT$, and which also satisfies $\Prob{\,\ocj\,\symmdiff\,\oj\,}=0$. This is sufficient by Remark~\ref{rmabs-1}.

\begin{definition}[The event $\ocj$]
For $j\geq 1$, let $\ocj$ be the event in which all of the following conditions are satisfied:
\begin{enumerate}[(i),font=\normalfont\bfseries,topsep=0pt]
\item $\OGm{\infty}{\thetagap_j}$ occurs; 
\item $\abs{\ginisumone}\leq M_j$, $\abs{\ginisumtwo}\leq M_j$, $\ginisumthree\leq M_j$.
\end{enumerate}
The event $\ocj$ is measurable with respect to $\GIFOUT$.
\end{definition}

To show $\Prob{\,\ocj\,\symmdiff\,\oj\,}=0$, we construct another family of events $\sbr*{\calA_{\epsilon}(j):\epsilon>0}$. Each $\calA_{\epsilon}(j)$ depends on a function $\delta(\epsilon)$ of $\epsilon$. This function $\delta(\cdot)$ needs to be chosen appropriately.

\begin{definition}[The event $\calA_\epsilon(j)$]
For $j\geq 1$ and $\epsilon>0$ let $\calA_\epsilon(j)$ be the event in which all of the following conditions are satisfied: 
\begin{enumerate}[(i),font=\normalfont\bfseries,topsep=0pt]
\item $\OGm{\infty}{\thetagap_j}$ occurs; 
\item 
$\abs*{\ginisumone} < M_j - \delta(\epsilon)$, 
$\abs*{\ginisumtwo} < M_j - \delta(\epsilon)$, 
$\ginisumthree < M_j - \delta(\epsilon)$.
\end{enumerate}
Each $\calA_\epsilon(j)$ is measurable with respect to $\GIFOUT$, and $\calA_\epsilon(j)\subset\ocj$.
\end{definition}

The function $\delta(\cdot)$ can be chosen in such a way that  
\[
\Prob[\big]{\,\calA_\epsilon(j)\,\symmdiff\,\oj\,}\;\leq\;\epsilon\;.
\]
If we let $\epsilon\to 0$ through the sequence $\seq*{2^{-n}}_{n=1}^\infty$, then we get $\seq*{\calA_{2^{-n}}(j)}_{n=1}^\infty$ exhausts $\ocj$.
Therefore  
\[
\Prob[\big]{\,\ocj\,\symmdiff\,\oj\,}=0\;.
\]
For the detailed procedure of choosing $\delta(\cdot)$ we refer to \cite{GP}.

\subsubsection{Step 2}\label{ss:GiniStep2}

In this step we verify \eqref{eq:main} holds for some $\seq*{\vartheta(j,k)}_{k\geq 1,j\geq 1}$ satisfying $\lim_{k\to\infty}\vartheta(j,k)=0$ for all $j\geq 1$. Suppose that the event $\onkj$ occurs for some $j\geq 1$, $k\geq 1$. 
Let $\uo$ be a vector consisting of the points of $\Gininkout$.
Let $\drho{\uo}{n_k}{\cdot}$ denote the density with respect to the Lebesgue measure on $\Dset^m$ of the conditional distribution of a vector consisting of the points of $\Gininkout$ taken in uniform random order given $\Gininkout$.
Since the event $\onkj$ is a subset of the event $\OGm{n_k}{\thetagap_j}$, using Proposition~\ref{prop:ginicondratio} we get that for all $\uz,\uzp\in\Dset^m$
\[
\exp\fbr*{ - m \usekd{gst} \thetagap_j^{-1} M_j }
\abs*{\frac{\van{\uzp}}{\van{\uz}}}^2
\leq
\frac{\drho{\uo}{n_k}{\uzp}}{\drho{\uo}{n_k}{\uz}}
\leq
\abs*{\frac{\van{\uzp}}{\van{\uz}}}^2
\exp\fbr*{ m \usekd{gst} \thetagap_j^{-1} M_j}\;,
\]
for some constant $\usekd{gst}>0$. Thus on the event $\onkj$ we have for all $\uz,\uzp\in\Dset^m$  
\[
\drho{\uo}{n_k}{\uzp} 
\cdot 
\abs*{\van{\uz}}^2 
\leq
\exp\fbr*{ m \usekd{gst} \thetagap_j^{-1} M_j }
\cdot
\drho{\uo}{n_k}{\uz} 
\cdot 
\abs*{\van{\uzp}}^2\;.
\]
Thus for all $A\in\FUBasisInsidempts$
\[
\abs*{\van{\uz}}^2
\cdot
\int_{\asvector{A}}\drho{\uo}{n_k}{\uzp}
\deri\el\fbr*{\uzp} 
\leq
\exp\fbr*{ m \usekd{gst} \thetagap_j^{-1} M_j}
\cdot
\drho{\uo}{n_k}{\uz}\cdot
\int_{\asvector{A}}\abs*{\van{\uzp}}^2
\deri\el(\uzp)\;.
\]
Thus for all $B\in\FUBasisOutside$ we get 
\begingroup
\addtolength{\jot}{0.25em}
\begin{align*}
& \fbr*{\int_{\Dset^m}\abs*{\van{\uz}}^2\deri\el(\uz)}
\cdot 
\Prob[\Big]{\,\fbr[\big]{\,\Gini_{n_k,\inn}\in A\,}\,\cap\,\fbr[\big]{\,\Gini_{n_k,\out}\in B\,}\,\cap\,\onkj\,}\\
\leq{} & 
\exp\fbr[\Big]{m\usekd{gst} \thetagap_j^{-1} M_j}
\cdot
\int_{(\Gini_{n_k,\out})^{-1}(B)\;\cap\;\onkj}
\fbr*{\int_{\asvector{A}}\abs*{\van{\uzp}}\deri\el(\uzp)}\deri\PROB\;.
\end{align*}
\endgroup
Thus, using \eqref{eq:ginicandidate} we get
\begingroup
\addtolength{\jot}{0.25em}
\begin{align*}
& \Prob[\Big]{\,\fbr[\big]{\,\Gininkin\in A\,}\,\cap\,\fbr[\big]{\,\Gini_{n_k,\out}\in B\,}\,\cap\,\onkj\,}\\
\leq{} & 
\exp\fbr[\Big]{m \usekd{gst} \thetagap_j^{-1} M_j}
\cdot
\int_{(\Gini_{n_k,\out})^{-1}(B)\;\cap\;\onkj}
\nu_{\Phi,\Psi,\Dset}\kernel{A}{\GIFOUT}\deri\PROB\;.
\numberthis\label{eq:g:2:1}
\end{align*}
\endgroup
Similarly, we also get
\begingroup
\addtolength{\jot}{0.25em}
\begin{align*}
& \Prob[\Big]{\,\fbr[\big]{\,\Gini_{n_k,\inn}\in A\,}\,\cap\,\fbr[\big]{\,\Gini_{n_k,\out}\in B\,}\,\cap\,\onkj\,}\\
\geq{} & 
\exp\fbr[\Big]{ - m \usekd{gst} \thetagap_j^{-1} M_j}
\cdot 
\int_{(\Gini_{n_k,\out})^{-1}(B)\;\cap\;\onkj}
\nu_{\Phi,\Psi,\Dset}\kernel{A}{\GIFOUT}\deri\PROB\;.
\numberthis\label{eq:g:2:2}
\end{align*}
\endgroup
Combining \eqref{eq:g:2:1} and \eqref{eq:g:2:2} we get \eqref{eq:main} with $\vartheta(j,k)=0$ for all $j\geq 1$, $k\geq 1$.
This concludes \ref{c:gini:2}.

\subsubsection{Step 3}

In this step we will verify that $\lim_{j\to\infty}\Prob*{\,\om\,\setminus\,\oj\,}=0$. 
We need to define a sequence of events $\seq*{\calA(j)}_{j\geq 1}$.

\begin{definition}[The event $\calA(j)$] 
For $j\in\NN$ let $\calA(j)$ be the event in which all of the following conditions are satisfied:
\begin{enumerate}[(i),font=\normalfont\bfseries,topsep=0pt]
\item $\OGm{\infty}{\thetagap_j}$ occurs;
\item $\abs*{\ginisumone} < M_j - 1$, 
$\abs*{\ginisumtwo} < M_j - 1$, 
$\ginisumthree < M_j - 1$.
\end{enumerate}
\end{definition}

We have $M_{j+1}\geq M_j$ for all $j\in\NN$ and $M_j\to\infty$ as $j\to\infty$.
We also have $\thetagap_j\geq\thetagap_{j+1}>0$ for all $j\in\NN$ and $\thetagap_j\to 0$ as $j\to\infty$. 
Therefore, we get $\seq*{\calA(j)}_{j=1}^\infty$ exhausts $\om$.
From the definition of the events $\calA(j)$, $\calA_\epsilon(j)$, and $\ocj$ we get 
\[
\calA(j)\subset\calA_\epsilon(j)\subset\ocj\subset\om\;.
\]
Therefore $\seq*{\ocj}_{j=1}^\infty$ exhausts $\om$.
Since $\Prob*{\,\ocj\,\symmdiff\,\oj\,}=0$ and $\oj\subset\om$, we have $\lim_{j\to\infty}\Prob{\,\om\,\setminus\,\oj\,}=0$.
This concludes \ref{c:gini:3}.
The relationships between the events defined in this step are as follows:
\[
\begin{tikzcd}
\om & & & & \oj\arrow[l l l l, "\text{as } j\to\infty", "\text{exhaust}"']\\
& & & & \\
\calA(j) \arrow[u u, "\text{exhaust}", "\text{as } j\to\infty"'] 
\arrow[r r, "\subset"] & & 
\calA_{2^{-n}}(j) \arrow[r r, "\text{exhaust}", "\text{as } n\to\infty"'] & & 
\ocj \arrow[u u, leftrightarrow, "\text{almost equal}"]
\end{tikzcd}
\]

\section{Estimates for the finite-dimensional \texorpdfstring{$\alpha$}{alpha}-GAFs}\label{sec:finitegaf}

In Sections~\ref{sec:finitegaf} and \ref{sec:limitinggaf} our objective is to lay down the steps of proving Theorem~\ref{thm:alphagaf-main} using Theorem~\ref{thm:abscont}. The detailed proofs are in Sections~\ref{sec:proof1}-\ref{sec:step3}. In Section~\ref{sec:limitinggaf}, we will carry out a three step procedure for verifying the conditions of Theorem~\ref{thm:abscont}, analogous to what we did for the Ginibre ensemble in Section~\ref{sec:ginibre}. In this section we derive some estimates regarding finite dimensional approximations of the $\alpha$-GAF which we will use in Section~\ref{sec:limitinggaf}. These are estimates analogous to estimates for the finite dimensional approximations of the Ginibre ensemble in Appendix~\ref{a:ginibre-estimates}. 

\subsection{Ratio of conditional densities}\label{ss:agaf:ratio}

Consider the sequence $\seq*{\falphagafn}_{n=1}^\infty$ of finite dimensional approximations of $\falphagaf$ (see Appendix~\ref{a:agaf} for reference). For $n\geq 1$, $\alphagafn$ denotes the ensemble of roots of $\falphagafn$. Let $\alphagafnin\coloneqq\alphagafn\cap\Dset$, $\alphagafnout\coloneqq\alphagafn\cap\Dset^\c$.
Consider $m\geq\rigidity$, $\uo\in\fbr*{\Dset^\c}^{n-m}$, and $\us\in\CC^{\rigidity-1}$.
The conditional density of $\alphagafnin$ given $\alphagafnout=\uo$ at some $\uz\in\Sigma_{\usm}$ is
\begin{equation}\label{eq:411}
\drho{\uo\,,\,\us}{n}{\uz} 
\; \coloneqq \;
C(\uo,\us) \; \abs*{\van{\uz\,,\,\uo}}^2 \;
\fbr*{\sum_{k=0}^{n}
\abs*{\frac{\sym{k}{\uz\,,\,\uo}}{\nka}}^2}^{-(n+1)}\;.
\end{equation}
Here $\sym{k}{\cdot}$ is the $k$th elementary symmetric function of degree $k$. Let 
\begin{equation}\label{eq:412}
\fD{\uz}{\uo} \; \coloneqq \; 
\sum_{k=0}^n\abs*{\frac{\sym{k}{\uz\cc\uo}}{\nka}}^2\;.
\end{equation}
So the ratio of the conditional densities at two locations $\uz,\uzp\in\Sigma_{\usm}$ is
\begin{equation}\label{eq:413}
\frac{\drho{\uo\,,\,\us}{n}{\uzp}}
{\drho{\uo\,,\,\us}{n}{\uz}} 
=\abs*{\frac{\van{\uzp\,,\,\uo}}{\van{\uz\,,\,\uo}}}^2
\fbr*{\frac{\fD{\uzp}{\uo}}{\fD{\uz}{\uo}}}^{-(n+1)}\;.
\end{equation}
We need bounds for this ratio. We will bound the ratio of the Vandermonde 
terms and the ratio of the symmetric functions separately. To bound the 
ratio of the Vandermonde terms we need some estimates for sum of inverse 
powers of zeros of $\falphagafn$ and $\falphagaf$.

\begin{notation}\label{n:vecsym}
Let $n,N\in\NN$ and let $\underline{v}\in\CC^N$. 
We define 
$\vecsym{n}{i}{\underline{v}},\vecsymb{n}{i}{\underline{v}}\in\CC^{n+1}$
as
\[
\vecsym{n}{i}{\underline{v}}
\;\coloneqq\;
\fbr*{\frac{\sym{k-i}{\underline{v}}}{\nckat}}_{k=0}^n, 
\quad
\vecsymb{n}{i}{\underline{v}} 
\;\coloneqq\;
\fbr*{\frac{\overline{\sym{k-i}{\underline{v}}}}{\nckat}}_{k=0}^n\;,
\]
where we use the convention that if $k-i<0$ or if $k-i>N$ then 
$\sym{k-i}{\underline{v}}=0$. Utilizing this notation we can 
write $\fD{\uz}{\uo}$ from \eqref{eq:412} as
\[
\fD{\uz}{\uo} = \vecsymn{n}{0}{(\uz\cc\uo)}.
\]
Thus, an alternative form of \eqref{eq:411} is 
\[
\drho{\uo\,,\,\us}{n}{\uz}
 = C(\uo,\us)\,\abs*{\van{\uz\cc\uo}}^2\,
\fbr*{\vecsymn{n}{0}{(\uz\cc\uo)}}^{-(n+1)}\;.
\]
\end{notation}

\subsection{Estimates for inverse power sums of zeroes}\label{ss:EIPZ}

Now our objective is to state some estimates for the sum of inverse power of zeros. We start by constructing a partition of unity on $\Dset^\c$. Recall from Remark~\ref{remark:disk} that we assume $\Dset$ to be a disk of radius $r_0$ centered at the origin.

\paragraph{The functions $\varphi$, $\widetilde{\varphi}$ and the partition of unity $\seq*{\phi_j}_{j\geq 0}$ on $\Dset^\c$:}
Let 
\[
x_1\coloneqq\frac{1+e}{2}\;,
\quad x_2\coloneqq e\;,
\quad x_3\coloneqq\frac{e(1+e)}{2}\;.
\]
Let $\varphi$ be a non-negative radial $C_c^{\infty}$-function supported on $[r_0,x_3 r_0]$ such that $\varphi=1$ on $[x_1 r_0, x_2 r_0]$ and $\varphi(r_0 + y r_0)=1-\varphi(e r_0 + y e r_0)$, for $0\le y\le e/2$. Let $\widetilde{\varphi}$ be another non-negative radial $C_c^{\infty}$ function with the same support as $\varphi$, satisfying $\widetilde{\varphi}(r_0 + y r_0)=1$ for $0\leq y\leq e/2$ and $\widetilde{\varphi}=\varphi$ otherwise. Let $\phi_0=\widetilde{\varphi}$ and for $j\geq 1$ let $\phi_j(z)=\varphi(|z|/e^j)$. Then the collection of functions $\fbr*{\phi_j}_{j\geq 0}$ is a partition of unity on $\Dset^\c$. 

\begin{notation}
For $j\geq 0$, and for $\bbF$ either $\falphagafn$ for some $n\in\NN$ or $\falphagaf$, we define 
\[
\inv{\bbF}{l}{j}{j+1}
\; \coloneqq \; 
\sum_{\substack{\bbF(\omega)=0\\
\phi_j(\omega)\neq 0}}
\frac{\phi_j(\omega)}{\omega^l}\,,
\quad
\invabs{\bbF}{l}{j}{j+1}
\; \coloneqq \; 
\sum_{\substack{\bbF(\omega)=0\\
\phi_j(\omega)\neq 0}}
\frac{\phi_j(\omega)}{|\omega|^l}\,.
\]
For $j\geq 1$, and $\bbF$ same as above, let 
\[
\inv{\bbF}{l}{0}{j}
\coloneqq\sum_{j^\prime=0}^{j-1}
\inv{\bbF}{l}{j^\prime}{j^\prime+1}\,,
\quad
\invabs{\bbF}{l}{0}{j}
\coloneqq\sum_{j^\prime=0}^{j-1}
\inv{\bbF}{l}{j^\prime}{j^\prime+1}\,.
\]
\end{notation}

Suppose $\bbF=\falphagafn$ for some $n\in\NN$. 
Since $\fbr*{\phi_j}_{j=0}^{\infty}$ is a partition of unity, 
for sufficiently large $j$,
$\inv{\bbF}{l}{0}{j}$ is the sum of inverse $l$'th power of all the zeroes of 
$\falphagafn$ of which there are finitely many. 
So we can define $\inv{\falphagafn}{l}{0}{\infty}$ as the limit of $\inv{\falphagafn}{l}{0}{j}$ as $j\to\infty$ (in a.s.\ sense). 
We cannot do this immediately for $\bbF=\falphagaf$.
To do this we derive some $L^1$ bounds.

\begin{proposition}
Let $\Phi$ be a $C_c^\infty$ radial function supported on the annulus $\Set*{z\in\CC : r_0\leq |z|\leq x_3 r_0}$. Then, for $R \ge 1$, $l\in\NN$, $n\in\NN$ we have:
\begingroup
\addtolength{\jot}{0.25em}
\begin{align}
\Exp\tbr*{\abs*{\int\Phi\fbr*{\frac{z}{R}}\frac{1}{z^l}\deri[\alphagafn](z)}}
\leq{} & \usec{CPhi}(\Phi) \cdot l^2 \cdot \frac{1}{r_0^l} \cdot \frac{1}{R^l}\;,
\label{eq:invroot1}\\
\Exp\tbr*{\int\Phi\fbr*{\frac{z}{R}}\frac{1}{\abs*{z}^l}\deri[\alphagafn](z)}
\leq{} & \usec{CPhi}(\Phi) \cdot l^2\cdot \frac{1}{r_0^l} \cdot \frac{1}{R^{-l+\fta}}\;.
\label{eq:invroot2}
\end{align} 
\endgroup
Here $\usec{CPhi}(\Phi)>0$ is a constant depending on $\Phi$. The same bounds also hold for $\alphagaf$.
\label{p:inv:1}
\end{proposition}

We prove Proposition~\ref{p:inv:1} in Section~\ref{ss:pf:p:inv:1}.

\begin{proposition}
For $l\in\NN$ and $n\in\NN$
\[ 
\Exp\tbr*{\abs*{\;
\int\frac{\widetilde{\varphi}(z)}{z^l}
\deri[\alphagafn](z)}
\;}
\leq
\Exp\tbr*{\;
\int\frac{\widetilde{\varphi}(z)}{|z|^l}
\deri[\alphagafn](z)
\;}
\leq \usec{p42}\cdot l^2\cdot \frac{1}{r_0^l}\;,
\] 
for some constant $\usec{p42}>0$. The same bounds also hold for $\alphagaf$.
\label{p:inv:2}\end{proposition}

We prove Proposition~\ref{p:inv:2} in Section~\ref{ss:pf:p:inv:2}.

\begin{proposition}[Uniform bounds on sum of inverse powers]

For $l\geq 1$, and $\bbF$ either $\falphagafn$ for 
some $n\in\NN$ or $\falphagaf$, the following are true:
\begin{enumerate}[(a),font=\normalfont\bfseries,topsep=0pt]
\item For all $j\geq 0$ the following is well-defined 
\[
\inv{\bbF}{l}{j}{\infty}
\coloneqq\sum_{j^\prime=j}^{\infty} 
\inv{\bbF}{l}{j^\prime}{j^\prime+1}
\]
i.e., the infinite sum converges absolutely a.s.

\item For all $j\geq 0$
\[
\Exp\tbr*{\;\absinv{\bbF}{l}{j}{\infty}\;}
< \usec{p43}\cdot l^2 \cdot \frac{1}{r_0^l} \cdot 
\exp\fbr[\big]{ - j l }
\]
for some constant $\usec{p43}>0$.

\item For all $j\geq 0$
\[
\Prob*{\,\absinv{\bbF}{l}{j}{\infty}>\usec{p43}\cdot l^2 \cdot \frac{1}{r_0^l} \cdot \exp\fbr*{-j\frac{l}{2}}\,}
\leq\exp\fbr*{-j\frac{l}{2}}\;.
\]
\end{enumerate}
\label{p:inv:3}
\end{proposition}

We prove Proposition~\ref{p:inv:3} in Section~\ref{ss:pf:p:inv:3}.

\begin{proposition}[Uniform bounds on sum of 
absolute value of inverse powers]
Let 
\[
\moment\coloneqq1+\fta. 
\]
For $l\geq\moment$, and $\bbF$ either $\falphagafn$ 
or $\falphagaf$, the following are true: 
\begin{enumerate}[(a),font=\normalfont\bfseries,topsep=0pt]
\item For all $j\geq 0$ the following is well-defined
\[
\invabs{\bbF}{l}{j}{\infty}
\coloneqq\sum_{j^\prime=j}^{\infty} 
\invabs{\falphagafn}{l}{j^\prime}{j^\prime+1}\;.
\]
i.e., the sum converges a.s.

\item For all $j\geq 0$ 
\[
\Exp\tbr*{\;
\invabs{\bbF}{l}{j}{\infty}
\;}
< \usec{p44} \cdot l^2 \cdot \frac{1}{r_0^l} \cdot 
\exp\fbr[\big]{ - j l}
\]
for some constant $\usec{p44}>0$.

\item For all $j\geq 0$ 
\[
\Prob*{\,\invabs{\bbF}{l}{j}{\infty}>\usec{p44} \cdot l^2 \cdot \frac{1}{r_0^l} \cdot \exp\fbr*{-j\frac{l}{2}}\,}\leq\exp\fbr*{-j\frac{l}{2}}\;.
\]
\end{enumerate}
\label{p:inv:4}
\end{proposition}

Proposition~\ref{p:inv:4} can be proved in a 
way similar to the proof of Proposition~\ref{p:inv:3}.
So we omit the proof.

\subsection{Bound on the ratio of the Vandermonde terms}

\begin{definition}[The events $\OGm{n}{\thetagap}$ and $\OGm{\infty}{\thetagap}$]
Let $\OGm{n}{\thetagap}$ be the event that $\Omega_n^m$ 
occurs and the points of $\alphagafn$ outside $\Dset$ 
are at least $\thetagap$ distance away from the 
boundary of $\Dset$. Let $\OGm{\infty}{\thetagap}$
be the event that $\Omega_\infty^m$ 
occurs and the points of $\alphagaf$ outside $\Dset$ 
are at least $\thetagap$ distance away from the 
boundary of $\Dset$.
\label{d:gaf:gap}
\end{definition}

\begin{notation}\label{n:gafsum}
For $n\in\NN$ let
\begingroup
\addtolength{\jot}{0.25em}
\begin{align*}
\X_n \coloneqq{} & 
\sum_{k=1}^{\moment-1} 
\abs*{\sum_{\omega_j\in\alphagafnout} 
\frac{1}{\omega_j^k}}
+ \sum_{\omega_j\in\alphagafnout} 
\frac{1}{|\omega_j|^{\moment}}\\
={}& \sum_{k=1}^{\moment-1}\absinv{\falphagafn}{k}{0}{\infty}+
\invabs{\falphagafn}{\moment}{0}{\infty}\;.
\end{align*}
\endgroup
\end{notation}

\begin{proposition}[Bounding ratio of the Vandermonde terms]

Suppose for some $n,m\in\NN$ with $n\geq m$,
and $\thetagap\in(0,1)$, the event $\OGm{n}{\thetagap}$ 
occurs. Let $\uo$ be a vector consisting of the roots of 
$\falphagafn$ outside $\Dset$. 
Then for all $\uz,\uzp\in\Dset^m$
\[
\exp\fbr[\Big]{-m\usekd{p45}\thetagap^{-1}\X_n} 
\abs*{\frac{\van{\uzp}}{\van{\uz}}}
\le\abs*{\frac{\van{\uzp\,,\,\uo}}{\van{\uz\,,\,\uo}}}
\le\exp\fbr[\Big]{m\usekd{p45}\thetagap^{-1}\X_n}
\abs*{\frac{\van{\uzp}}{\van{\uz}}}
\]
for some constant $\usekd{p45}>0$.
\label{p:van}
\end{proposition}

We prove Proposition~\ref{p:van} in Section~\ref{ss:pf:p:van}. 

\subsection{Bound on the ratio of the symmetric functions}

In this section our objective is to bound the ratio 
$\fD{\uzp}{\uo}/\fD{\uz}{\uo}$
appearing \eqref{eq:412}.

\begin{proposition}

Consider $n,m\in\NN$ with $n\geq m\geq\rigidity$.
Let $\us=(s_1,\dots,s_{\rigidity-1})$ be an element of 
$\CC^{\rigidity-1}$. 
Let $\uo$ be an element of $(\Dset^\c)^{n-m}$.
Let $\uz$ and $\uzp$ be elements of $\Sigma_{\usm}$.
For $0\leq i\leq m$ and $\rigidity\leq j\leq m$ let 
\begin{equation}
\fDij{\uz}{\uo}
\,\coloneqq\,
\frac{\abs*{\Big\langle\vecsymb{n}{i}{\uo},\vecsym{n}{j}{\uo}\Big\rangle}}
{\vecsymn{n}{0}{(\uz\cc\uo)}}
\,=\,
\frac{\displaystyle\abs*{
    \sum_{k=0}^n
        \frac{\overline{\sym{k-i}{\uo}}}{\nckat}
        \frac{\sym{k-j}{\uo}}{\nckat}}
    }
    {\displaystyle
        \sum_{k=0}^n\abs*{
        \frac{\sym{k}{\uz\cc\uo}}{\nckat}
        }^2
    }\;.
\label{eq:dijdef}\end{equation}
Let
\begin{equation}\label{eq:Dhatdef}
\fDhat{\uz}{\uo}\coloneqq
\max_{0\leq i\leq m}\;
\max_{\rigidity\leq j\leq m}\; 
\fDij{\uz}{\uo}\;. 
\end{equation}
Then, there exists a constant $\usekdm{ratio:sym:agaf}>0$ such that 
\[
1 - \usekdm{ratio:sym:agaf}\cdot\fDhat{\uz}{\uo}
\,\leq\,
\frac{\fD{\uzp}{\uo}}{\fD{\uz}{\uo}}
\,\leq\, 
1 + \usekdm{ratio:sym:agaf}\cdot\fDhat{\uz}{\uo}\;.
\]
\label{p:sym}
\end{proposition}

We prove Proposition~\ref{p:sym} in Section~\ref{ss:pf:p:sym}.
In the next proposition we bound the terms
$\sym{k-i}{\uo}$ which appears in the 
expression of $\fDij{\uz}{\uo}$ in \eqref{eq:dijdef}. 
We introduce a notation first.

\begin{notation}\label{n:factorial}
For non-negative integer $i$ and positive integer $j$ let
\[
\ffa{i}{j} \coloneqq \fbr[\Big]{(i+1)\cdots(i+j)}^{\alpha/2}\;.
\]
When $i=0$, we have $\pi_\alpha[0;j]=(j!)^{\alpha/2}$. Note that, for fixed $j$, 
\[
\lim_{i\to\infty} \frac{\pi_\alpha[i;j]}{i^{j\alpha/2}} = 1\;.
\]
Thus $\sum_{i=1}^{\infty} 1/{\fbr*{\ffa{i}{j}}^2}$ and 
$\sum_{i=1}^{\infty} 1/{\fbr[\big]{i^{j\alpha/2}\pi_\alpha[i:j]}}$
are finite if and only if $j\geq\rigidity$.
\end{notation}

\begin{proposition}[Expansion of the elementary symmetric functions of the outside roots]
Let $m$ be a positive integer. 
There exist positive constants 
$\usekdm{471}$, 
$\usekdm{472}$,
$\usekdm{473}$,
such that the following holds.
Consider the ensemble $\alphagafn$ 
for some $n\geq m$.  
There exists random variables $(\g_r)_{r=0}^{n}$  
such that the following hold: 

\begin{enumerate}[(i),font=\normalfont\bfseries,topsep=0pt]
\item On the event $\omn$ (ref. Definition~\ref{d:om:ommn}) 
we have for all $m\leq l\leq n$
\[ 
\frac{\sym{n-l}{\uo}}
{\fbr*{\binom{n}{n-l}(n-l)!}^{\alpha/2}}
=\frac{1}{\xi_n}
\sum_{r=0}^{n-l}
(-1)^{n-l-r}
\cdot\g_r
\cdot\frac{\xi_{l+r}}{\ffa{l}{r}}\;,
\]
where $\uo$ is the vector consisting of points of the 
ensemble $\alphagafnout$ taken in uniform random order.

\item On the event $\omn$ we have 
$|\g_r|\leq \fbr[\big]{\usekdm{471}}^r$ 
for all $0\leq r\leq n$.

\item For $m\leq l\leq n-\lev$ define 
\[
\tail{l}{n}
\coloneqq \sum_{r=\lev}^{n-l}
(-1)^{r}
\frac{\g_r\cdot\xi_{l+r}}{\ffa{l}{r}}\;. 
\]
For $l>n-\rigidity$ let $\tail{l}{n}\coloneqq0$.
Therefore, for all $m\leq l\leq n$
\[ 
\frac{\sym{n-l}{\uo}}
{\fbr*{\binom{n}{n-l}(n-l)!}^{\alpha/2}}
=\frac{(-1)^{(n-l)}}{\xi_n}\tbr*{
\sum_{r=0}^{(n-l)\wedge(\lev-1)}
(-1)^{r}
\frac{\g_r\cdot\xi_{l+r}}{\ffa{l}{r}} + \tail{l}{n}}\;.
\]
\item If $ l \geq \usekdm{472} $, then
\[
\fbr*{\Exp\fbr*{\abs*{\tail{l}{n}}^2\indicator{\Omega_n^m}}}^{1/2}
\leq
\frac{\fbr[\big]{ \usekdm{473} }^{\lev}}{l^{\lev\alpha/2}}\;.
\]
\end{enumerate}
\label{p:rec}
\end{proposition}

We prove Proposition~\ref{p:rec} in Section~\ref{ss:pf:p:rec}.

\section{The limiting procedure for \texorpdfstring{$\alpha$}{alpha}-GAF zeroes}\label{sec:limitinggaf}

In Sections~\ref{sec:finitegaf} and \ref{sec:limitinggaf} our objective is to lay down the steps of proving Theorem~\ref{thm:alphagaf-main} using Theorem~\ref{thm:abscont} with $\XX=\alphagaf$ and $\xn=\alphagafn$. In Section~\ref{sec:finitegaf} we have obtained the necessary estimates for $\alphagafn$. In this section we will lay down the steps to verify the conditions of Theorem~\ref{thm:abscont}, and obtain \eqref{eq:target:GAF} from \eqref{targeteq}. The detailed proofs are in later sections.

\subsection{Overview}

For $\Upsilon\in\cS{\Dset^\c}$, the probability measure $\PROB_{\Dset|\Dset^\c}[\cdot|\Upsilon]$ is supported on the set of configurations $\Set{\sigma\in\cS{\Dset}\given\asvector{\sigma}\subset\Sigma_{\usm}}$, where $m=\numberofpoints(\Upsilon)$ and $\us=\constraintexternal(\Upsilon)$. Therefore, for $A\in\Borel{\cS{\Dset}}$ satisfying $\asvector{A}\subset\Sigma_{\usm}$
\[
\PROB_{\Dset|\Dset^\c}[A|\Upsilon] = \frho{\asvector{A}}{\Upsilon}\;. 
\]
Let $(\Phi,\Psi)$ be the potentials given by $\Phi\equiv 0$ and $\Psi(z) = -2\log|z|$. 
Consider the corresponding probability kernel $\nu_{\Phi,\Psi,\Dset}$.
For $\Upsilon\in\cS{\Dset^\c}$, 
the probability measure $\nu_{\Phi,\Psi,\Dset}\kernel{\cdot}{\Upsilon}$ 
is also supported on $\Set{\sigma\in\cS{\Dset}\given\asvector{\sigma}\subset\Sigma_{\usm}}$, 
where $m=\numberofpoints(\Upsilon)$ and $\us=\constraintexternal\fbr[\big]{\Upsilon}$.
For $A\in\Borel{\cS{\Dset}}$ satisfying $\asvector{A}\subset\Sigma_{\usm}$ we have from \eqref{eq:def-kernel}
\begin{equation}\label{eq:candidate}
\nu_{\Phi,\Psi,\Dset}\kernel{A}{\Upsilon}
\coloneqq
\frac{\displaystyle\int_{\asvector{A}}\abs*{\van{\uz}}^2\deri\leb(\uz)}
    {\displaystyle\int_{\Sigma_{\usm}}\abs*{\van{\uz}}^2\deri\leb(\uz)}\;,
\end{equation}
where $\leb$ is the Lebesgue measure on $\Sigma_{\usm}$. 
If we can establish the generalized Gibbs property,
    then we get that
    for all $m\in\NN$ and all $A\in\BorelInsidempts$
    \[
    \mm_1\fbr[\big]{\alphagafout}\nu_{\phi,\psi,\Dset}\kernel{A}{\alphagafout} 
    \;\le\; 
    \PROB_{\Dset|\Dset^\c}[A|\alphagafout] 
    \;\le\; 
    \MM_1\fbr[\big]{\alphagafout}\nu_{\phi,\psi,\Dset}\kernel{A}{\alphagafout}\;,  
    \]
for some measurable functions $\mm_1,\MM_1:\cS{\Dset^\c}\to(0,\infty)$. 
Then, by the Radon-Nikodym Theorem we will have
    \[
    \mm_2\fbr[\big]{\alphagafout}\abs*{\van{\uz}}^2  
    \;\le\; 
    \frac{\deri\frho{\cdot}{\alphagafout}}{\deri\el}(\uz) 
    \;\le\; 
    \MM_2\fbr[\big]{\alphagafout}\abs*{\van{\uz}}^2\;,
    \]
for some measurable functions $\mm_2,\MM_2:\cS{\Dset^\c}\to(0,\infty)$. Thus we get \eqref{eq:target2}. Due to Theorem~\ref{thm:abscont}, it is enough to verify the conditions of approximate Gibbsianity for $\xn=\alphagafn$ and $\XX=\alphagaf$ with respect to the probability kernel $\nu_{\Phi,\Psi,\Dset}$. As in the case of the Ginibre ensemble, here also we treat the $\alphagafout$ configurations separately depending on $m=\numberofpoints\fbr[\big]{\alphagafout}$. Since the $\alpha$-GAF zero ensemble is rigid up to order $\rigidity$, the cases $m=0,\dots,\rigidity-1$ are trivial {($\Sigma_{\usm}$ is either the empty set or a singleton.)} So we consider $m\geq\rigidity$. From now on we fix a value of $m\geq\rigidity$.

\paragraph{The three step procedure:} 
We will verify the conditions of approximate Gibbsianity in three steps. 
\begin{itemize}[label=$\blacktriangleright$,leftmargin=*,topsep=0pt]
\item In the first step we define the events $\fbr*{\oj}_{j\geq 1}$ and
also define, for each $j\geq 1$, the sequence $(n_k)_{k=1}^\infty$
and the sequence of events $\fbr*{\onkj}_{k=1}^\infty$. We also verify that 
\begin{itemize}[label=-, leftmargin=*]
\item $\oj\subset\Omega(j+1)$ for all $j$;
\item $\oj\subset\om$;
\item $\oj$ is measurable with respect to $\alphagafout$;
\item $\onkj$ is measurable with respect to $\alphagafnkout$;
\item $\onkj\subset\omnk$;
\item $\oj\subset\liminf_{k\to\infty}\onkj$ a.s.
\end{itemize}

\item In the second step we verify condition~\ref{c:ag:2}-\ref{c:ag:2:5} of the definition of approximate Gibbsianity holds with respect to the probability kernel $\nu_{\Phi,\Psi,\Dset}$ and some $\seq*{\correctionterm{j}{k}}_{j\geq 1,k\geq 1}$ which satisfies condition~\ref{c:ag:2}-\ref{c:ag:2:1}.

\item In the third step we verify $\Prob[\big]{\,\om\,\setminus\,\oj\,}\to 0$.
\end{itemize}

\paragraph{Parameters $M$, $\thetagap$, $\delta$:}
In each step we will define various events and prove relations between them.
These events will involve three parameters $M>3$, $\thetagap\in(0,1)$, and
$\delta\in(0,1)$.
The parameter $M$ is to be thought of as large.
The parameter $\thetagap$ is to be thought of as small.
The parameter $\delta$ is also to be thought of as small. 
Our analysis will proceed by first fixing $M$ and $\thetagap$, 
and then taking $\delta$ sufficiently small. 

\paragraph{Convention of naming events:}
We will name the events in the form $\OMG{\thetagap}{n}{i:j}{M}{\delta}$.
The subscript $n$ indicates that the event is measurable with respect to
the roots of $\falphagafn$. 
Here the subscript $n$ can be either a positive integer or $\infty$.
The superscript $i:j$ indicates that the event is the 
$j$'th event with the subscript $n$ defined in the $i$'th step.
In step 3 we define the event $\OMG{\thetagap}{\infty}{3:2}{M}{\bullet}$
which doesn't depend on the parameter $\delta$. 
The sign $\bullet$ indicates that the parameter $\delta$ is not involved.

\paragraph{$\delta^3$-negligible events and $\delta^3$-inclusion:}
Let us now introduce the notion of
``$\delta^3$-negligible event''
and ``$\delta^3$-inclusion,'' 
where $\delta\in(0,1)$. The 
reason for using $\delta^3$ 
as opposed to just $\delta$ 
will be clear from Theorem~\ref{t:s:2}.

\begin{definition}[$\delta^3$-negligible event]
For $\delta\in(0,1)$, we say an event $\mathcal{E}$ is $\delta^3$-negligible if there exists a constant $C>0$ 
such that 
\[
\PROB\fbr*{\mathcal{E}}\leq C\delta^3.
\]
The constant $C$ can involve $m$ and $\Dset$. 
\end{definition}

\begin{definition}[$\delta^3$-inclusion]
For $\delta\in(0,1)$ we say an event $\mathcal{E}_1$ 
is $\delta^3$-included in an
event $\mathcal{E}_2$ if there exists a $\delta^3$-negligible event $\mathcal{E}_3$ such that
\[
\mathcal{E}_1\setminus\mathcal{E}_3\subset
\mathcal{E}_2.
\]
We denote this relation by
\[
\begin{tikzcd}
\mathcal{E}_1\arrow[r, "\mathcal{E}_3"] & \mathcal{E}_2\;.
\end{tikzcd}
\]
If $\mathcal{E}_2$ is obtained by excluding the $\delta^3$-negligible event $\mathcal{E}_3$ from 
$\mathcal{E}_1$ i.e., $\mathcal{E}_1\setminus\mathcal{E}_3=
\mathcal{E}_2$, then we denote this relation by
\[
\begin{tikzcd}
\mathcal{E}_1\arrow[r, "\mathcal{E}_3", "\supset"'] & \mathcal{E}_2\;.
\end{tikzcd}
\]
\end{definition}

\subsubsection{Outline of Step 1}

In the first step our objectives are: 
to define the sequence of events 
$\fbr*{\oj}_{j=1}^\infty$; 
for each $j\geq 1$ define the sequence $(n_k)_{k=1}^\infty$; 
for each $j$ and $k$ define the event $\onkj$; 
show the conditions on these events as stated in
Theorem~\ref{abscond} are satisfied;
show that $\oj\subset\liminf_{k\to\infty}\onkj$. 
\begin{enumerate}[label = Step~1.\arabic*:, 
font = \normalfont\bfseries, 
labelsep = \parindent,
itemindent = 3\parindent,
leftmargin = 3\parindent,
topsep = 0pt
]
\item In Definition~\ref{d:i:1:1}
we introduce the event
$\OMG{\thetagap}{\infty}{1:1}{M}{\delta}$. 
This event is measurable with respect to the 
roots of $\falphagaf$ in the annulus 
\begin{equation}\label{eq:def-annulus}
\mathrm{An}(\delta)
\coloneqq\Set[\Big]{z\in\CC\given r_0<|z|<\Rdel}\;,
\end{equation}
where we define $\Rdel$ appropriately in Section~\ref{ss:Rdkd}.
Recall that $r_0$ is the radius of $\Dset$. 
Hence $\OMG{\thetagap}{\infty}{1:1}{M}{\delta}$
is measurable with respect to the roots of $\falphagaf$
outside $\Dset$.

\item In Definition~\ref{d:n:1:1} 
we define the event 
$\OMG{\thetagap}{\nd}{1:1}{M}{\delta}$.
This event is measurable with respect to the roots of  
$\fand$ in the annulus $\mathrm{An}(\delta)$, 
where $\nd$ is defined in Section~\ref{subsec:parameters}.
Thus, this event is measurable with 
respect to the roots of $\fand$ outside $\Dset$.
The parameter $\nd$ depends $\thetagap$,
but in the notation we suppress the dependence on $\thetagap$
since our analysis proceeds by first 
fixing $M$ and $\thetagap$ and then 
taking $\delta$ sufficiently small.

\item In Proposition~\ref{p:1:1} 
we show that for all $M>3$, $\thetagap\in(0,1)$ and for $\delta>0$ 
sufficiently small depending on $M$, the event 
$\OMG{\thetagap}{\infty}{1:1}{M}{\delta}$ 
is $\delta^3$-included in the event $\OMG{\thetagap}{\nd}{1:1}{M}{\delta}$.
So there is a $\delta^3$-negligible event $\eneg{1:1}$ such that 
\[
\OMG{\thetagap}{\infty}{1:1}{M}{\delta}
\,\setminus\,
\eneg{1:1}
\;\subset\;
\OMG{\thetagap}{\nd}{1:1}{M}{\delta}\;.
\]
Note that, the event $\eneg{1:1}$ depends 
on $\theta$. But as in the case of $\nd$, 
we refrain from writing it explicitly. 
Although the event $\eneg{1:1}$ doesn't 
depend on $M$, the relationship between the 
events as above holds for $\delta$ small
enough depending on $M$.

\item Let $(M_j)_{j=1}^\infty$ be a sequence of positive real numbers
which is monotonically increasing and diverges to $\infty$. 
Let $(\thetagap_j)_{j=1}^\infty$ be a sequence of positive real numbers 
which is monotonically decreasing and converges to $0$. 
Let $(\delta_k)_{k=1}^\infty$ be a sequence of positive real numbers 
which is monotonically decreasing, converges to $0$, and 
\begin{equation}\label{eq:deltakchoice}
    \sum_{k}\delta_k<\infty\;.
\end{equation}
For each $j\geq 1$ let 
\begin{equation}\label{eq:omegajdef} 
\oj
\;\coloneqq\;
\liminf_{k \to \infty}\, 
\OMG{\thetagap_j}{\infty}{1:1}{M_j}{\delta_k}\;.
\end{equation}
For each $j\geq 1$, let $(n_k)_{k=1}^\infty$ be the sequence
$(\ndk)_{k=1}^\infty$ where we construct $\nd$ using $\theta_j$.
Finally, for each $j\geq 1$ and $k\geq 1$ let 
\begin{equation}\label{eq:omegajdef2} 
\onkj
\;\coloneqq\;
\OMG{\thetagap_j}{\ndk}{1:1}{M_j}{\delta_k}\;.
\end{equation}
Recall from the statement of Theorem~\ref{thm:abscont} that 
we need $\oj$ to be measurable with respect to $\alphagafout$.
This is true because for each $j$ and $k$ the event 
$\OMG{\thetagap_j}{\infty}{1:1}{M_j}{\delta_k}$
is measurable with respect to $\alphagafout$.
Similarly, we need $\onkj$ to be measurable with respect to 
$\alphagafnkout$. 
This is true because for each $j$ and $k$ the event 
$\OMG{\thetagap_j}{\ndk}{1:1}{M_j}{\delta_k}$
is measurable with respect to $\alphagafnkout$.
Each $\oj$ is a subset of $\om$. 
And each $\onkj$ is a subset of $\omnk$.

\item In Theorem~\ref{t:s:1} we show that $\oj\subset\liminf_{k\to\infty}\onkj$ a.s. 
\end{enumerate}

The relationship between the events defined in this step is
\[
\begin{tikzcd}
\OMG{\thetagap}{\infty}{1:1}{M}{\delta}
\arrow{r}{\eneg{1:1}}
&\OMG{\thetagap}{\nd}{1:1}{M}{\delta}\;.
\end{tikzcd}
\]

\subsubsection{Outline of Step 2}

In this step our objective is to show that 
\eqref{abscond} holds 
for some $\fbr*{\vartheta(j,k)}_{j\geq 1, k\geq 1}$
satisfying $\lim_{k\to\infty}\vartheta(j,k)=0$ for all $j\geq 1$.
\begin{enumerate}[
label = Step~2.\arabic*:, 
font = \normalfont\bfseries, 
labelsep = \parindent,
itemindent = 3\parindent,
leftmargin = 3\parindent,
topsep = 0pt
]

\item 
In Definition~\ref{d:n:2:1}
we introduce an event 
$\OMG{\thetagap}{\nd}{2:1}{M}{\delta}$. 
This event is measurable with respect to the roots 
of $\fand$ outside $\Dset$.

\item 
In Proposition~\ref{p:2:1}
we show that
there exists a $\delta^3$-negligible event 
$\eneg{2:1}$ such that 
\[
\OMG{\thetagap}{\nd}{1:1}{M}{\delta}
\,\setminus\,
\eneg{2:1}
\;\subset\;
\OMG{\thetagap}{\nd}{2:1}{M}{\delta}\;. \quad\text{Thus }\;
\begin{tikzcd}
\OMG{\thetagap}{\nd}{1:1}{M}{\delta}
\arrow[r, "\eneg{2:1}"]&
\OMG{\thetagap}{\nd}{2:1}{M}{\delta}\;.
\end{tikzcd}
\]
Both events 
$\OMG{\thetagap}{\nd}{1:1}{M}{\delta}$
and 
$\OMG{\thetagap}{\nd}{2:1}{M}{\delta}$
are measurable with respect to the roots of $\fand$ outside $\Dset$.
But the crucial difference between the events is that
the event $\OMG{\thetagap}{\nd}{1:1}{M}{\delta}$
involves the roots of $\fand$ in the annulus $\mathrm{An}(\delta)$ (defined in \eqref{eq:def-annulus}), 
whereas the event $\OMG{\thetagap}{\nd}{2:1}{M}{\delta}$
involves all the roots of $\fand$ outside $\Dset$.

\item 
In Proposition~\ref{p:2:2} we define the event 
$\OMG{\thetagap}{\nd}{2:2}{M}{\delta}$
as a subset of the event 
$\OMG{\thetagap}{\nd}{2:1}{M}{\delta}$
obtained by removing a $\delta^3$-negligible set $\eneg{2:2}$:
\[
\OMG{\thetagap}{\nd}{2:2}{M}{\delta}
\;=\;
\OMG{\thetagap}{\nd}{2:1}{M}{\delta}
\,\setminus\,\eneg{2:2}\;.
 \quad\text{Thus }\;
\begin{tikzcd}
\OMG{\thetagap}{\nd}{2:2}{M}{\delta}
\arrow[r, "\eneg{2:2}", "\supset"'] &
\OMG{\thetagap}{\nd}{2:1}{M}{\delta}\;.
\end{tikzcd}
\]

\item In Proposition~\ref{p:2:3} we define the event
$\OMG{\thetagap}{\nd}{2:3}{M}{\delta}$ 
as a subset of the event 
$\OMG{\thetagap}{\nd}{2:2}{M}{\delta}$ 
obtained by removing a $\delta^3$-negligible set $\eneg{2:3}$:
\[
\OMG{\thetagap}{\nd}{2:3}{M}{\delta}
\;=\;
\OMG{\thetagap}{\nd}{2:2}{M}{\delta}
\,\setminus\,
\eneg{2:3}\;.
 \quad\text{Thus }\;
\begin{tikzcd}
\OMG{\thetagap}{\nd}{2:2}{M}{\delta}
\arrow[r, "\eneg{2:3}", "\supset"'] &
\OMG{\thetagap}{\nd}{2:3}{M}{\delta}\;.
\end{tikzcd}
\]

\item In Proposition~\ref{p:2:4} we define 
$\OMG{\thetagap}{\nd}{2:4}{M}{\delta}$ 
as a subset of 
$\OMG{\thetagap}{\nd}{2:3}{M}{\delta}$ 
obtained by removing a $\delta^3$-negligible set $\eneg{2:4}$:
\[
\OMG{\thetagap}{\nd}{2:4}{M}{\delta}
\;=\;
\OMG{\thetagap}{\nd}{2:3}{M}{\delta}
\,\setminus\,
\eneg{2:4}\;.
 \quad\text{Thus }\;
\begin{tikzcd}
\OMG{\thetagap}{\nd}{2:3}{M}{\delta}
\arrow[r, "\eneg{2:4}", "\supset"'] &
\OMG{\thetagap}{\nd}{2:4}{M}{\delta}\;.
\end{tikzcd}
\]

\item 
In Proposition~\ref{p:2:5} we show that on the event $\OMG{\thetagap}{\nd}{2:4}{M}{\delta}$ the following holds.
Consider the conditional density $\drho{\uo\,,\,\us}{\nd}{\cdot}$, 
where $\uo$ is a vector consisting of the points of $\alphagafndout$  
and $\us$ is the vector of power sums up to order $\rigidity-1$ of the points of $\alphagafndin$. 
We show, in \eqref{eq:Bound-target-alternative}, that for every $A\in\Borel{\Sigma_{\usm}}$ the ratio of 
\[
\int_A \drho{\uo\,,\,\us}{\nd}{\uz}\deri\leb(\uz) \quad\mbox{and}\quad 
\frac{\displaystyle\int_A \abs{\van{\uzp}}^2 \deri\el_{\usm}(\uzp) }{\displaystyle\int\abs{\van{\uzp}}^2 \deri\el_{\usm}(\uzp)}
\]
is bounded by positive functions of $M$ and $\thetagap$, uniformly in $\nd$ and $A$. 

\item In Theorem~\ref{t:s:2} we establish 
that \eqref{abscond} is satisfied.
\end{enumerate}

The relation between the events can be summarized as follows:
\[
\begin{tikzcd}
\OMG{\thetagap}{\nd}{1:1}{M}{\delta} \arrow[r, "\eneg{2:1}"] & 
\OMG{\thetagap}{\nd}{2:1}{M}{\delta} \arrow[r, "\eneg{2:2}", "\supset"'] & 
\OMG{\thetagap}{\nd}{2:2}{M}{\delta} \arrow[r, "\eneg{2:3}", "\supset"'] & 
\OMG{\thetagap}{\nd}{2:3}{M}{\delta} \arrow[r, "\eneg{2:4}", "\supset"'] & 
\OMG{\thetagap}{\nd}{2:4}{M}{\delta}\,. 
\end{tikzcd}
\]

\subsubsection{Outline of Step 3}

In this step our objective is to show that 
$\PROB\fbr[\big]{\om\setminus\oj}\to 0$.

\begin{enumerate}[label = Step~3.\arabic*:, 
font = \normalfont\bfseries, 
labelsep = \parindent,
itemindent = 3\parindent,
leftmargin = 3\parindent,
topsep = 0pt
]
\item In Definition~\ref{d:n:3:1} we define the event $\OMG{\thetagap}{\nd}{3:1}{M}{\delta}$. 

\item In Proposition~\ref{p:3:1} 
we show that there is a $\delta^3$-negligible 
event $\eneg{3:1}$ such that
\[
\OMG{\thetagap}{\nd}{3:1}{M}{\delta}
\,\setminus\,\eneg{3:1}
\;\subset\;
\OMG{\thetagap}{\infty}{1:1}{M}{\delta}\;.
 \quad\text{Thus }\;
\begin{tikzcd}
\OMG{\thetagap}{\nd}{3:1}{M}{\delta}
\arrow[r, "\eneg{3:1}"]&
\OMG{\thetagap}{\infty}{1:1}{M}{\delta}\;.
\end{tikzcd}
\]

\item In Definition~\ref{d:i:3:1} we define the event $\OMG{\thetagap}{\infty}{3:1}{M}{\delta}$. 

\item In Proposition~\ref{p:3:2} 
we show that there is a
$\delta^3$-negligible event $\eneg{3:2}$ such that
\[
\OMG{\thetagap}{\infty}{3:1}{M}{\delta}
\,\setminus\,
\eneg{3:2}
\;\subset\;
\OMG{\thetagap}{\nd}{3:1}{M}{\delta}\;.
 \quad\text{Thus }\;
\begin{tikzcd}
\OMG{\thetagap}{\infty}{3:1}{M}{\delta}
\arrow[r, "\eneg{3:2}"]&
\OMG{\thetagap}{\nd}{3:1}{M}{\delta}\;.
\end{tikzcd}
\]

\item In Definition~\ref{d:i:3:2} we define the event $\OMG{\thetagap}{\infty}{3:2}{M}{\bullet}$.

\item In Proposition~\ref{p:3:3}
we show that there is a 
$\delta^3$-negligible event $\eneg{3:3}$ such that
\[
\OMG{\thetagap}{\infty}{3:2}{M}{\bullet}
\,\setminus\,
\eneg{3:3}
\;\subset\;
\OMG{\thetagap}{\infty}{3:1}{M}{\delta}\;.
 \quad\text{Thus }\;
\begin{tikzcd}
\OMG{\thetagap}{\infty}{3:2}{M}{\bullet}
\arrow[r, "\eneg{3:3}"]&
\OMG{\thetagap}{\infty}{3:1}{M}{\delta}\;.
\end{tikzcd}
\]
\item In Theorem~\ref{t:s:3} we show that 
$\PROB\fbr[\big]{\om\setminus\oj}\to 0$.
\end{enumerate}

The relationship between the events in 
this step is
\[
\begin{tikzcd}
\OMG{\thetagap}{\infty}{3:2}{M}{\bullet} 
\arrow[r, "\eneg{3:2}"]& 
\OMG{\thetagap}{\infty}{3:1}{M}{\delta}
\arrow[r, "\eneg{3:1}"]& 
\OMG{\thetagap}{\nd}{3:1}{M}{\delta}
\arrow[r, "\eneg{3:4}"]& 
\OMG{\thetagap}{\infty}{1:1}{M}{\delta}\;.
\end{tikzcd}
\]

The relationship between the events defined in all the steps is:
\[
\begin{tikzcd}
&\OMG{\thetagap}{\infty}{3:2}{M}{\bullet} \arrow[r, "\eneg{3:3}"]& 
\OMG{\thetagap}{\infty}{3:1}{M}{\delta} \arrow[r, "\eneg{3:2}"]& 
\OMG{\thetagap}{\nd}{3:1}{M}{\delta} \arrow[r, "\eneg{3:1}"]& 
\OMG{\thetagap}{\infty}{1:1}{M}{\delta} \arrow[d, "\eneg{1:1}"]\\
\OMG{\thetagap}{\nd}{2:4}{M}{\delta}&
\OMG{\thetagap}{\nd}{2:3}{M}{\delta}\arrow[l, "\eneg{2:4}"', "\subset"]& 
\OMG{\thetagap}{\nd}{2:2}{M}{\delta}\arrow[l, "\eneg{2:3}"', "\subset"]& 
\OMG{\thetagap}{\nd}{2:1}{M}{\delta}\arrow[l, "\eneg{2:2}"', "\subset"]& 
\OMG{\thetagap}{\nd}{1:1}{M}{\delta}\arrow[l, "\eneg{2:1}"']\;.
\end{tikzcd}
\]

\subsection{The parameters}\label{subsec:parameters}

\subsubsection{The function \texorpdfstring{$h$}{h}} 

Let $h:\NN\to\NN$ be a function such that 
\[
\lim_{L\to\infty}\sum_{l=L}^{L+h(L)}\frac{1}{l^{\alpha/8}}=0\;.
\]
For the sake of definiteness, we take
\[
h(L)\coloneqq\lfloor L^{\alpha/16}\rfloor\;.
\]

\subsubsection{The constant \texorpdfstring{$C_0$}{Czero}} 

Let $C_0$ be a positive real constant such that for $l\geq 1$ 
\begin{equation}\label{d:czero}
\sum_{r=1}^\infty 
\frac{1}{l^{r\alpha/8}\;(r!)^{\alpha/4}}\leq\frac{C_0}{l^{\alpha/8}}\;.
\end{equation}

\subsubsection{The function \texorpdfstring{$\Ld$}{Ld}} 

Let $\Ld$ be an integer satisfying the following conditions: 
\begin{enumerate}[(i),font=\normalfont\bfseries,topsep=0pt]
\item \label{c:Ld:1}
\[
\sum_{l=\Ld}^{\Ld+\hLd}
\frac{1}{l^{\alpha/8}}<\frac{\delta^3}{C_0}\;;
\]

\item\label{c:Ld:2} 
\[
\Prob*{\,
\frac{1}{2}<\frac{1}{\hLd}
\sum_{l=\Ld}^{\Ld+\hLd}
\abs*{\xi_l}^2 <\frac{3}{2}\,}>1-\delta^3\;. 
\]
\end{enumerate}

Clearly, $\Ld\to\infty$ as $\delta\to 0$.

\subsubsection{The function \texorpdfstring{$\Cdel$}{Cdel}}

Let $\Cdel$ be a positive integer such that the following conditions hold:
\begin{enumerate}[(i),font=\normalfont\bfseries,topsep=0pt]
\item\label{c:Cd:1} $\Cdel>\Ld+\hLd$.

\item\label{c:Cd:2} 
For all 
$0\leq i\leq m$, 
$\rigidity\leq j\leq m$, 
$0\leq r_1\leq m$, 
$0\leq r_2\leq m$:
\begin{align*}
& \sum_{l=\Cdel+1}^{\infty} 
\frac{1}{\fbr[\big]{\ffa{l}{i+r_1}\cdot
\ffa{l}{j+r_2}}^{2}}
\leq\delta^5\;,\\
& \sum_{l=\Cdel+1}^{\infty} 
\frac{1}
{l^{\lev\alpha/2}\cdot\ffa{l}{i}
\cdot\ffa{l}{j+r_2}}
\leq\delta^4\;,\\
& \sum_{l=\Cdel+1}^{\infty} 
\frac{1}{l^{\lev\alpha/2}\cdot\ffa{l}{j}\cdot\ffa{l}{i+r_1}}
\leq\delta^4\;,\\
& \sum_{l=\Cdel+1}^{\infty} 
\frac{1}{l^{\lev\alpha}\cdot\ffa{l}{i}\cdot\ffa{l}{j}}
\leq\delta^4\;.
\end{align*}
\end{enumerate}

Clearly, $\Cdel\to\infty$ as $\delta\to 0$.

\subsubsection{The function \texorpdfstring{$\gdel$}{gdel}} 

Let $\gdel$ be such that the following conditions hold: 
\begin{enumerate}[(i),font=\normalfont\bfseries,topsep=0pt]
\item\label{c:gd:1} for $1\leq l\leq \Cdel$ and
for $\bbF$ either $\falphagafn$ for some $n\in\NN$ or $\falphagaf$ 
\[
\PROB\fbr*{\sum_{j=0}^{\infty}
\absinv{\bbF}{l}{j}{j+1}>\gdel-1
}<\frac{\delta^3}{\Cdel}\;,
\]
\item\label{c:gd:2} for $\bbF$ either $\falphagafn$ for some $n\in\NN$ or $\falphagaf$ 
\[
\PROB\fbr*{\invabs{\bbF}{\moment}{0}{\infty}>\gdel-1}<\delta^3\;.
\]
\end{enumerate}
The uniform $L^1$ bounds obtained in Propositions~\ref{p:inv:3} and
\ref{p:inv:4} implies the existence of $\gdel$ satisfying these
conditions.

\subsubsection{The functions \texorpdfstring{$\QFIop$}{QFIop} and \texorpdfstring{$\QFZop$}{QFZop}}

For $0\leq i\leq m$ and $\rigidity\leq j\leq m$ define the function
$\QFIop:\CC^{\Cdel}\to\RR$ such that
\begin{equation}\label{eq:fijdef}
\QFI{x_1,\ldots,x_{\Cdel}}
\coloneqq
\abs*{\sum_{l=\mij}^{\Cdel}\overline{x_{l+i-m}} \cdot x_{l+j-m} \cdot \fbr*{l!}^{\alpha}}\;.
\end{equation}
Also define $\QFZop:\CC^{\Cdel}\to\RR$ such that
\begin{equation}\label{eq:fzdef}
\QFZ{x_1,\cdots,x_{\Cdel}}\coloneqq
\frac{1}{\hLd} 
\sum_{l=\Ld}^{\Ld+\hLd}
\;\abs*{x_{l-m}}^2 \cdot \fbr*{l!}^{\alpha}\;.
\end{equation}

\subsubsection{The function \texorpdfstring{$\edel$}{edel}} 

\begin{notation}
\text{}
\begin{enumerate}[(i),font=\normalfont\bfseries,topsep=0pt] 
\item Let $P_k$ be the $k$-th Newton 
polynomial expressing the elementary 
symmetric function of order $k$ in 
terms of power-sums of order 
$1,2,\cdots,k$. That is, for 
complex numbers 
$x_1,\cdots,x_n$, let 
$s_k\coloneqq\sum_{j=1}^n x_j^k$
and $e_k\coloneqq\sum_{i_1<\cdots<i_k} 
x_{i_1}x_{i_2}\cdots x_{i_k}$ for $1\leq k\leq n$.
Then $e_k=P_k(s_1,\cdots,s_k)$ for $1\leq k\leq n$.
As a polynomial of $k$ variables, $P_k$ does not depend on $n$. 

\item For a vector $\uw=(\mathrm{w}_1,\dots,\mathrm{w}_N)\in\CC^N$ 
and for $1\leq k\leq N$, let $\uw^{\tn{k}}$ denote the 
vector $(w_1,\dots,w_k)$. 

\item Let $\linf{\cdot}$ denote the $L^\infty$ norm.
\end{enumerate}
\end{notation}
For $\delta\in(0,1)$ let $\edel\in(0,1)$ be such that the following conditions hold:
\begin{enumerate}[(i),font=\normalfont\bfseries,topsep=0pt]
\item\label{c:ed:1}
for all $\uw_1,\uw_2\in\CC^{\Cdel}$ satisfying 
$\linf{\uw_1-\uw_2}<\edel$ and $\linf{\uw_1}\vee\linf{\uw_2}<\gdel$ 
we have: 
\[
\max_{0\leq i\leq m}\;
\max_{\rigidity\leq j\leq m}\;
\abs*{\QFI{\fbr*{ P_k\fbr*{\uw_1^{\tn{k}}}}_{k=1}^{\Cdel}}
-\QFI{\fbr*{ P_k\fbr*{\uw_2^{\tn{k}}}}_{k=1}^{\Cdel}}}
\;<\; \delta\;;
\]
\item\label{c:ed:2} for all $\uw_1,\uw_2\in\CC^{\Cdel}$ satisfying 
$\linf{\uw_1-\uw_2}<\edel$ and $\linf{\uw_1}\vee\linf{\uw_2}<\gdel$ 
we have: 
\[
\abs*{\QFZ{\fbr*{P_k\fbr*{\uw_1^{\tn{k}}}}_{k=1}^{\Cdel}}
-\QFZ{\fbr*{P_k\fbr*{\uw_2^{\tn{k}}}}_{k=1}^{\Cdel}}}\;<\;\delta\;.
\]
\end{enumerate}

\subsubsection{The functions \texorpdfstring{$\Rdel$}{Rdel} and \texorpdfstring{$\kdel$}{kdel}}\label{ss:Rdkd}

Let $\kdel$ a positive integer such that: 
\begin{enumerate}[(i),font=\normalfont\bfseries,topsep=0pt] 
    \item\label{c:kd:1} 
    \[
    \sum_{l\geq 1} \exp\fbr*{-\kdel\frac{l}{2}}
    \;\leq\;
    \delta^3\;;
    \]
    \item\label{c:kd:2}
    \[
    \max_{1\leq l\leq \Cdel}\;
    \fbr[\big]{ \usec{p43} \vee \usec{p44} }
    \cdot l^2 \cdot \frac{1}{r_0^l} \cdot
    \exp\fbr*{-\kdel\frac{l}{2}}
    \;\leq\;
    \delta\wedge\edel\;,
    \]
    where $\usec{p43}$ is the constant introduced in 
    Proposition~\ref{p:inv:3}, 
    and $\usec{p44}$ is the constant introduced in 
    Proposition~\ref{p:inv:4}.
\end{enumerate}
Let $\Rdel \coloneqq r_0\; \exp\fbr*{\kdel}$, where recall that $r_0$ is the radius of $\Dset$.

\subsubsection{The function \texorpdfstring{$\nd$}{nd}} 

\begin{notation}\label{n:product}
Let $\InProd$ and $\InProdNd$ be the product of the roots inside $\Dset$ of $\falphagaf$ and $\fand$ respectively.
\end{notation}

Let $\nd$ be such that the following hold:

\begin{enumerate}[(i),font=\normalfont\bfseries,topsep=0pt]
\item\label{c:nd:1}
In the complement of a $\delta^3$-negligible event we have 
\begin{gather*}
\max_{l\leq\Cdel}\;\max_{0\leq k\leq \kdel}\;
\abs[\Bigg]{\inv{\falphagaf}{l}{k}{k+1}-
\inv{\fand}{l}{k}{k+1}}
\;<\;
\frac{\delta\wedge\edel}{\kdel+1}\;;\\
\max_{0\leq k\leq \kdel}\;
\abs[\Bigg]{\invabs{\falphagaf}{\moment}{k}{k+1}
-\invabs{\fand}{\moment}{k}{k+1}}
\;<\;
\frac{\delta\wedge\edel}{\kdel+1}\;.
\end{gather*}

\item\label{c:nd:2}
In the complement of a $\delta^3$-negligible event we have  
\[
\frac{1}{2}
\;\leq\; 
\frac{1}{\nd}\sum_{l=1}^{\nd}\abs*{\xi_l}^2
\;<\;
\frac{3}{2}\;.
\]

\item\label{c:nd:3}In the complement of a $\delta^3$-negligible event we have 
\[
\frac{1}{\abs*{\xi_0}^2}\abs*{\;\InProdSqNd-\InProdSq\;}\;<\;\delta\;.
\]

\item\label{c:nd:4}
\[
\Prob[\Big]{\,\OGm{\infty}{\thetagap}\,\symmdiff\,\OGm{\nd}{\thetagap}\,}\;<\;\delta^3\;.
\]
\end{enumerate}

\subsubsection{The vectors \texorpdfstring{$\zzv$}{zzv}, \texorpdfstring{$\zpv$}{zpv}, \texorpdfstring{$\zppv$}{zppv}}

For $1\leq l\leq \Cdel$ define
\begin{align*}
\zzl &\;\coloneqq\; P_l\fbr*{\inv{\falphagaf}{1}{0}{\kdel}
,\dots,\inv{\falphagaf}{l}{0}{\kdel}}\;,\\
\zpl &\;\coloneqq\; P_l\fbr*{\inv{\fand}{1}{0}{\kdel}\;,
\dots,\inv{\fand}{l}{0}{\kdel}},\\
\zppl &\;\coloneqq\; P_l\fbr*{\inv{\fand}{1}{0}{\infty},
\dots,\inv{\fand}{l}{0}{\infty}}\;.
\end{align*}

Let 
\[
\zzv\;\coloneqq\;\fbr[\Big]{\zzl}_{l=1}^{\Cdel}\,,\quad
\zpv\;\coloneqq\;\fbr[\Big]{\zpl}_{l=1}^{\Cdel}\,,\quad
\zppv\;\coloneqq\;\fbr[\Big]{\zppl}_{l=1}^{\Cdel}\;.
\]

\begin{remark} 
Suppose $m$ is the number of roots of $\fand$ inside $\Dset$
and $\uo$ is a vector of roots of $\fand$ outside $\Dset$.
Then for $m\leq l\leq\Cdel$ we have
\[
\zpp{l-m}=\frac{\sym{\nd-l}{\uo}}{\sym{\nd-m}{\uo}}\;.
\]
Therefore 
\begin{equation}\label{eq:tailavg}
\QFZ{\zppv} = 
\frac{1}{\hLd}
\sum_{l=\Ld}^{\Ld+\hLd}
\abs*{\frac{\sym{\nd-l}{\uo}}{\sym{\nd-m}{\uo}}}^2 \fbr[\big]{l!}^{\alpha}\;.
\end{equation}
Here we use condition~\ref{c:Cd:1} defining $\Cdel$.
Similarly, for $0\leq i\leq m$ and $\rigidity\leq j\leq m$ we have
\begin{equation}\label{eq:qij}
\QFI{\zppv}=
\abs*{\sum_{k=\nd-\Cdel}^{\nd} 
\frac{\overline{\sym{k-i}{\uo}}}{\overline{\sym{\nd-m}{\uo}}}
\frac{\sym{k-j}{\uo}}{\sym{\nd-m}{\uo}}\fbr[\big]{(\nd-k)!}^{\alpha}}
\;.
\end{equation}    
\end{remark}

\subsection{The details of the three step procedure}

\subsubsection{Step 1}

First, we define the event $\OMG{\thetagap}{\infty}{1:1}{M}{\delta}$.

\begin{definition}[The event $\OMG{\thetagap}{\infty}{1:1}{M}{\delta}$]

For $M>3$, $\thetagap\in(0,1)$, and $\delta\in(0,1)$, let $\OMG{\thetagap}{\infty}{1:1}{M}{\delta}$ 
be the event in which 
all of the following conditions are satisfied:
    \begin{enumerate}[(i),font=\normalfont\bfseries,topsep=0pt]
        \item\label{c:i:1:1:1}
            $ \OGm{\infty}{\thetagap} $ occurs\;;
        \item\label{c:i:1:1:2}
            $ \max_{1\leq s<\moment} \absinv{\falphagaf}{s}{0}{\kdel} \le  M $\;;
        \item\label{c:i:1:1:3} 
            $ \invabs{\falphagaf}{\moment}{0}{\kdel} \le  M $\;;
        \item\label{c:i:1:1:4}
            $ \max_{0\leq i\leq m}\; 
            \max_{\rigidity\leq j\leq m}\; 
            \QFI{\zzv} \leq  M $\;;
        \item\label{c:i:1:1:5} 
            $  M^{-1} \le \QFZ{\zzv} \le  M $\;.
    \end{enumerate}
    \label{d:i:1:1}
\end{definition}

Next, we define the event $\OMG{\thetagap}{\nd}{1:1}{M}{\delta}$.

\begin{definition}[The event $\OMG{\thetagap}{\nd}{1:1}{M}{\delta}$]

For $M>3$, $\thetagap\in(0,1)$, and $\delta\in(0,1)$, let $\OMG{\thetagap}{\nd}{1:1}{M}{\delta}$ 
be the event in which 
all of the following conditions are satisfied: 
    \begin{enumerate}[(i), 
font = \normalfont\bfseries,topsep=0pt]
        \item\label{c:n:1:1:1}
            $ \OGm{\nd}{\theta} $ occurs\;;
        \item\label{c:n:1:1:2}
            $ \max_{1\leq s<\moment} 
            \absinv{\fand}{s}{0}{\kdel} \le  M+1 $\;; 
        \item\label{c:n:1:1:3} 
            $ \invabs{\fand}{\moment}{0}{\kdel} \le  M+1 $\;;
        \item\label{c:n:1:1:4}
            $ \max_{0\leq i\leq m}\;
            \max_{\rigidity\leq j\leq m}\; \QFI{\zpv} \le  M+1 $\;;
        \item\label{c:n:1:1:5}
            $ ( M+1)^{-1} \le \QFZ{\zpv} \le  M+1 $\;.
    \end{enumerate}
    \label{d:n:1:1}
\end{definition}

Next, we show that 
\textit{the event $\OMG{\thetagap}{\infty}{1:1}{M}{\delta}$
is $\delta^3$-included in 
the event $\OMG{\thetagap}{\nd}{1:1}{M}{\delta}$.}

\begin{proposition}

For $M>3$, $\thetagap\in(0,1)$, and $\delta\in(0,1)$ 
sufficiently small depending on $M$,
there exists an event $\eneg{1:1}$ such that
$\pneg{1:1}\leq\cneg{1:1}\delta^3$ 
for some constant $\cneg{1:1}>0$ and
\[
\OMG{\thetagap}{\infty}{1:1}{M}{\delta}
\,\setminus\,
\eneg{1:1}
\;\subset\;
\OMG{\thetagap}{\nd}{1:1}{M}{\delta}\;,
 \quad\text{i.e.},\;
\begin{tikzcd}
\OMG{\thetagap}{\infty}{1:1}{M}{\delta}
\arrow{r}{\eneg{1:1}}
&\OMG{\thetagap}{\nd}{1:1}{M}{\delta}\;.
\end{tikzcd}
\]
\label{p:1:1}
\end{proposition}

We prove Proposition~\ref{p:1:1} in Section~\ref{ss:pf:p:1:1}. 
Next, we show that 
$\oj\subset\liminf_{k\to\infty}\onkj$, 
which is the end goal of Step 1.

\begin{theorem}
For each $j\geq 1$, $\oj\subset\liminf_{k\to\infty}\onkj$ a.s.
\label{t:s:1}
\end{theorem} 

We prove Theorem~\ref{t:s:1} 
in Section~\ref{ss:pf:t:s:1}.
This concludes Step 1.

\subsubsection{Step 2}

First, we define the event $\OMG{\thetagap}{\nd}{2:1}{M}{\delta}$.

\begin{definition}[The event $\OMG{\thetagap}{\nd}{2:1}{M}{\delta}$]

For $M>3$, $\thetagap\in(0,1)$, and $\delta\in(0,1)$, let 
$\OMG{\thetagap}{\nd}{2:1}{M}{\delta}$ 
be the event in which 
all of the following conditions are satisfied: 
\begin{enumerate}[(i),font=\normalfont\bfseries,topsep=0pt]
    \item\label{c:n:2:1:1} 
        $ \OGm{\nd}{\thetagap} $ occurs\;;
    \item\label{c:n:2:1:2} 
        $ \max_{1\leq s<\moment} 
        \absinv{\fand}{s}{0}{\infty} \le  M+2 $\;;
    \item\label{c:n:2:1:3}
        $ \invabs{\fand}{\moment}{0}{\infty} \le  M+2 $\;;
    \item\label{c:n:2:1:4} 
        $ \max_{0\le i\le m}\; 
          \max_{\rigidity \le j\le m}\; 
          \QFI{\zppv} \le  M+2 $\;;
    \item\label{c:n:2:1:5} 
        $ (M+2)^{-1} \le \QFZ{\zppv} \le  M+2 $\;.
\end{enumerate}
\label{d:n:2:1}
\end{definition}

Next, we show that \textit{the event
$\OMG{\thetagap}{\nd}{1:1}{M}{\delta}$ 
is $\delta^3$-included in
$\OMG{\thetagap}{\nd}{2:1}{M}{\delta}$.}

\begin{proposition}

For all $M>3$, $\thetagap\in(0,1)$, 
and $\delta>0$ sufficiently small depending on $M$, 
there exists an event $\eneg{2:1}$ such that 
$\pneg{2:1}<\cneg{2:1}\delta^3$ 
for some constant $\cneg{2:1}>0$ and
\[
\OMG{\thetagap}{\nd}{1:1}{M}{\delta}
\,\setminus\,
\eneg{2:1}
\;\subset\;
\OMG{\thetagap}{\nd}{2:1}{M}{\delta}\;,
 \quad\text{i.e.},\;
\begin{tikzcd}
\OMG{\thetagap}{\nd}{1:1}{M}{\delta}
\arrow[r,"\eneg{2:1}"] & 
\OMG{\thetagap}{\nd}{2:1}{M}{\delta}\;.
\end{tikzcd}
\]
\label{p:2:1}\end{proposition}

We prove this Section~\ref{ss:pf:p:2:1}. Next, we define the event 
$\OMG{\thetagap}{\nd}{2:2}{M}{\delta}$ 
by subtracting a $\delta^3$-negligible event from
$\OMG{\thetagap}{\nd}{2:1}{M}{\delta}$.
Thus, \textit{$\OMG{\thetagap}{\nd}{2:1}{M}{\delta}$
is $\delta^3$-included in 
$\OMG{\thetagap}{\nd}{2:2}{M}{\delta}$}
by construction.

\begin{proposition}

For all $M>3$, $\thetagap\in(0,1)$, and $\delta\in(0,1)$,
there exists an event $\eneg{2:2}$ 
such that $\pneg{2:2}<\cneg{2:2}\delta^3$ for some constant $\cneg{2:2}>0$
and on $\omnd\setminus\eneg{2:2}$ we have the following:
Let $\uz$ be a vector consisting of roots of 
$\fand$ inside $\Dset$.
Let $\uo$ be a vector consisting of roots of 
$\fand$ outside $\Dset$.
Then 
\[
\frac{1}{2} 
\abs*{\sym{k}{\uz\cc\uo}} 
\le \abs*{\sym{k}{\uo}} 
\le \frac{3}{2} 
\abs*{\sym{k}{\uz\cc\uo}}
\]
for all $\nd-\Ld-\hLd\le k\le\nd-\Ld$. 
Let 
\[
\OMG{\thetagap}{\nd}{2:2}{ M}{\delta}
\;\coloneqq\;
\OMG{\thetagap}{\nd}{2:1}{ M}{\delta}
\,\setminus\,\eneg{2:2}\;,
 \quad\text{i.e.},\;
\begin{tikzcd}
\OMG{\thetagap}{\nd}{2:1}{M}{\delta} 
\arrow[ r , " \eneg{2:2} ", " \supset"' ] & \OMG{\thetagap}{\nd}{2:2}{M}{\delta}\;.
\end{tikzcd}
\]
\label{p:2:2}\end{proposition}

We prove this in Section~\ref{ss:pf:p:2:2}.
Next, we define $\OMG{\thetagap}{\nd}{2:3}{M}{\delta}$
by subtracting a $\delta^3$-negligible event from
$\OMG{\thetagap}{\nd}{2:2}{M}{\delta}$. Thus
\textit{$\OMG{\thetagap}{\nd}{2:2}{M}{\delta}$
is $\delta^3$-included in 
$\OMG{\thetagap}{\nd}{2:3}{M}{\delta}$}
by construction.

\begin{proposition}

For all $M>3$, $\thetagap\in(0,1)$, and $\delta\in(0,1)$,
there exists an event $\eneg{2:3}$ such that 
$\pneg{2:3}<\cneg{2:3}\delta^3$ 
for some constant $\cneg{2:3}>0$ and 
on $\Omega_{\nd}^m\setminus\eneg{2:3}$ we have
\begin{equation}\label{c:p:2:3:1}
\frac{ 8 }{ 27 }\cdot\QFZ{ \zppv }
\,\le\, 
\frac{ \InProdSqNd }{ \abs*{\xi_0}^2 }
\,\le\, 
8\cdot\QFZ{ \zppv }\;. 
\end{equation} 
Moreover, on the event 
$\OMG{\thetagap}{\nd}{2:2}{ M}{\delta}\setminus\eneg{2:3}$
we have 
\begin{equation}\label{c:p:2:3:2}
\frac{1}{\nd}
    \sum_{ k=0 }^{ \nd }
    \abs*{ \frac{ \sym{k}{\uz\cc\uo} }{ \sym{\nd-m}{\uo} } }^2
    \fbr[\big]{ (\nd-k)! }^{\alpha} 
    \;\geq\; 
    \frac{4}{27} \cdot \frac{1}{M+2}\;, 
\end{equation}
where $\uz$ is a vector consisting of the roots of $\fand$ inside 
$\Dset$, and $\uo$ is a vector consisting of the roots of $\fand$ outside
$\Dset$. Let 
\[
\OMG{\thetagap}{\nd}{2:3}{M}{\delta} 
\;\coloneqq\; 
\OMG{\thetagap}{\nd}{2:2}{M}{\delta} 
\,\setminus\,
\eneg{2:3}\;,
 \quad\text{i.e.},\;
\begin{tikzcd}
    \OMG{\thetagap}{\nd}{2:2}{M}{\delta} 
    \arrow[ r, "\eneg{2:3}", " \supset "'] & \OMG{\thetagap}{\nd}{2:3}{M}{\delta}\;.
\end{tikzcd}
\]
\label{p:2:3}
\end{proposition}

We prove this in Section~\ref{ss:pf:p:2:3}. Before 
presenting the next result let us introduce a notation.

\begin{notation} 
For $\underline{v}=(v_0,\dots,v_N)\in\CC^{N+1}$
and $S\subset\nat$ we denote by 
$\underline{v}\odot\mathbbm{1}_S$ the vector 
$\uw=(\mathrm{w}_0,\dots,\mathrm{w}_N)\in\CC^{N+1}$
such that $\mathrm{w}_i=0$ if $i\not\in S$, $\mathrm{w}_i=v_i$ 
if $i\in S$. For $k\in\NN$ let $[0:k)$ denote the
set $\Set*{0,1,\dots,k-1}$. For $k,n\in\nat$ with 
$n\geq k$ let $[k:n]$ denote the set 
$\Set*{ k, k+1, \dots, n}$.

Therefore, if $\uo$ is a vector consisting of the roots of $\fand$ outside $\Dset$, and if the event $\omnd$ occurs, then for all 
$0\leq i\leq m$ and $\rigidity\leq j\leq m$ we have
\begin{gather*}
\sum_{k=0}^{\nd-\Cdel-1}
\frac{\overline{\sym{k-i}{\uo}}}{\ndckat}
\frac{\sym{k-j}{\uo}}{\ndckat} 
\;=\; \frontsum
\;;\\
\sum_{k=\nd-\Cdel}^{\nd}
\frac{\overline{\sym{k-i}{\uo}}}{\ndckat}
\frac{\sym{k-j}{\uo}}{\ndckat} 
\;=\;
\tailsum\;.
\end{gather*}
\label{n:hadamard}
\end{notation}

Now we define the event $\OMG{\thetagap}{\nd}{2:4}{M}{\delta}$
by subtracting a $\delta^3$-negligible event from
$\OMG{\thetagap}{\nd}{2:3}{M}{\delta}$. Thus, 
by construction \textit{$\OMG{\thetagap}{\nd}{2:3}{M}{\delta}$
is $\delta^3$-included in $\OMG{\thetagap}{\nd}{2:4}{M}{\delta}$.}

\begin{proposition}

For all $M>3$, $\thetagap\in(0,1)$, and sufficiently small $\delta\in(0,1)$, 
there exists an event $\eneg{2:4}$ such that: 
$\pneg{2:4}<\cneg{2:4}\delta^3$ for some constant $\cneg{2:4}>0$, 
and on the event $\omnd\setminus\eneg{2:4}$ we have
\[
\max_{0\leq i\leq m}\;
\max_{\rigidity\leq j\leq m}\;
\frac{\displaystyle\abs*{\frontsum}}{\displaystyle\totalsumdenom}
\;\le\;
\frac{2\delta}{\nd}\;, 
\]
where $\uz$ is a vector consisting of the roots of $\fand$ inside $\Dset$,
and $\uo$ is a vector consisting of the roots of $\fand$ outside $\Dset$.
Let
\[
\OMG{\thetagap}{\nd}{2:4}{M}{\delta}
\;\coloneqq\;
\OMG{\thetagap}{\nd}{2:3}{M}{\delta}
\,\setminus\,
\eneg{2:4}\;,
\quad\text{i.e.},\;
\begin{tikzcd}
\OMG{\thetagap}{\nd}{2:3}{M}{\delta} 
\arrow[r,"\eneg{2:4}", "\supset"'] & \OMG{\thetagap}{\nd}{2:4}{M}{\delta}\;.
\end{tikzcd}
\]
\label{p:2:4}
\end{proposition}

We present the proof of Proposition~\ref{p:2:4} in Section~\ref{ss:pf:p:2:4}. Before proceeding to the next result let us introduce a notation.

\begin{notation}\label{n:con:int}
For a vector of points $\uz=(\z_1,\dots,\z_m)\in\Dset^m$ let $\constraintinternal(\uz)$ be the vector $\us=(s_1,\dots,s_{\rigidity-1})\in\CC^{\rigidity-1}$ such that
$\sum_{i=1}^m \z_i^j = s_j$ for all $1\leq j\leq \rigidity-1$. 
For a configuration of points $\Upsilon\in\cS{\Dset}$ with $|\Upsilon|=m$, let $\constraintinternalpc(\Upsilon)$ be $\constraintinternal(\uz)$ where $\aspc{\uz}=\Upsilon$.
\end{notation}

Now we show that on the event $\OMG{\thetagap}{\nd}{2:4}{M}{\delta}$
the conditional density of the roots of $\fand$ inside $\Dset$ given
the roots of $\fand$ outside $\Dset$ is well-behaved.

\begin{proposition}[Uniform bound on the ratio of the conditional densities on $\OMG{\thetagap}{\nd}{2:4}{M}{\delta}$]

Suppose $\OMG{\thetagap}{\nd}{2:4}{M}{\delta}$ occurs 
for some $M>3$, $\thetagap\in(0,1)$, and $\delta\in(0,1)$. 
Let $\uz\in\Dset^m$ be a vector consisting of the roots of $\fand$ 
inside $\Dset$. Let $\uo$ be a vector consisting of the roots of $\fand$ outside $\Dset$. Let $\us\coloneqq\constraintinternal(\uz)$. Then for a.e.\ $\uzp$ with respect to the Lebesgue measure $\el_{\usm}$ on $\Sigma_{\usm}$ we have 
\begin{equation}\label{eq:Bound-target}
\exp\fbr[\Big]{-f(M,\thetagap)}
\abs*{\frac{\van{\uzp}}{\van{\uz}}}^2
\leq
\frac{\drho{\uo\,,\,\us}{\nd}{\uzp}}{\drho{\uo\,,\,\us}{\nd}{\uz}}
\leq 
\exp\fbr[\Big]{f(M,\thetagap)}
\abs*{\frac{\van{\uzp}}{\van{\uz}}}^2\;,
\end{equation}
where $f(M,\thetagap)=\usekdm{571}\fbr*{ M^2+M\thetagap^{-1}}$ for some constant $\usekdm{571}>0$.
Equivalently, 
\begin{equation}\label{eq:Bound-target-alternative}  
\exp\fbr[\Big]{-f(M,\thetagap)}
\int_{A}\drho{\uo}{\nd}{\uzp}\deri\el_{\usm}(\uzp)
\leq 
\frac{\displaystyle\int_A \abs{\van{\uzp}}^2 \deri\el_{\usm}(\uzp) }{\displaystyle\int_{\Sigma_{\usm}}\abs{\van{\uzp}}^2 \deri\el_{\usm}(\uzp)}
\leq
\exp\fbr[\Big]{f(M,\thetagap)}
\int_{A}\drho{\uo}{\nd}{\uzp}\deri\el_{\usm}(\uzp)
\end{equation}
for all $A\subset\Borel{\Sigma_{\usm}}$.
\label{p:2:5}
\end{proposition}

This is proved in Section~\ref{ss:pf:p:2:5}. 
Now we show that \eqref{eq:main} in condition~\ref{c:ag:2}-\ref{c:ag:2:5} is satisfied
with respect to the probability kernel $\nu_{\phi,\Psi,\Dset}$ and some $\seq*{\correctionterm{j}{k}}_{k\geq 1,j\geq 1}$ which satisfies condition~\ref{c:ag:2}-\ref{c:ag:2:1}.
This is the end goal of Step 2.

\begin{theorem}\label{t:s:2}
There exists $\seq*{\correctionterm{j}{k}}_{k\geq 1, j\geq 1}$ such that for 
$A\in\FUBasisInsidempts$, 
$B\in\FUBasisOutside$, 
$j\geq 1$, 
$k\geq 1$, we have:
\begingroup
\addtolength{\jot}{0.25em}
\begin{align*}
&\Prob[\Big]{
\,\fbr[\big]{\,\xinnk\in A\,}\,
\cap
\,\fbr[\big]{\,\xoutnk \in B\,}\,
\cap
\,\onkj\,}\\
\overset{j}{\scaleto{\asymp}{8pt}}\;
& \fbr[\Big]{\int_{\xoutnk^{-1}(B)\;\cap\;\onkj}
\nu_{\Phi,\Psi,\Dset}\kernel{A}{\xout}\deri\PROB
}
+ \correctionterm{j}{k}\;,
\end{align*}
\endgroup
and for each $j\geq 1$, $\lim_{k\to\infty}\correctionterm{j}{k}=0$.
\end{theorem}

This is proved in Section~\ref{ss:pf:t:s:2}.
This concludes Step 2.

\subsubsection{Step 3}

First, we define the event $\OMG{\thetagap}{\nd}{3:1}{M}{\delta}$.

\begin{definition}[The event $\OMG{\thetagap}{\nd}{3:1}{M}{\delta}$]

For $M>3$, $\thetagap\in(0,1)$, and $\delta\in(0,1)$,    
let $\OMG{\thetagap}{\nd}{3:1}{M}{\delta}$ 
be the event in which all of the following 
conditions are satisfied: 
\begin{enumerate}[(i),font=\normalfont\bfseries,topsep=0pt]
    \item\label{c:n:3:1:1}
        $ \OGm{\nd}{\theta} $ occurs\;;
    \item\label{c:n:3:1:2}
        $ \max_{1\leq s<\moment} 
        \absinv{\fand}{s}{0}{\kdel} \le  M-1 $\;; 
    \item\label{c:n:3:1:3} 
        $ \invabs{\fand}{\moment}{0}{\kdel} \le  M-1 $\;;
    \item\label{c:n:3:1:4}
        $ \max_{0\leq i\leq m}\; 
        \max_{\rigidity\leq j\leq m} \QFI{\zpv} \le  M-1 $\;;
    \item\label{c:n:3:1:5}
        $ ( M-1)^{-1} \le \QFZ{\zpv} \le  M-1 $\;.
\end{enumerate}
\label{d:n:3:1}
\end{definition}

Now we show \textit{$\OMG{\thetagap}{\nd}{3:1}{M}{\delta}$
is $\delta^3$-included in 
$\OMG{\thetagap}{\infty}{1:1}{M}{\delta}$.}

\begin{proposition}

For $M>3$, $\thetagap\in(0,1)$, and $\delta\in(0,1)$ 
sufficiently small depending on $M$, there exists an event 
$\eneg{3:1}$ such that $\pneg{3:1}<\cneg{3:1}\delta^3$ for some
constant $\cneg{3:1}>0$ and
\[
\OMG{\thetagap}{\nd}{3:1}{M}{\delta}
\,\setminus\,
\eneg{3:1}
\;\subset\;
\OMG{\thetagap}{\infty}{1:1}{M}{\delta}\;,
\quad\text{i.e.},\;
\begin{tikzcd}
\OMG{\thetagap}{\nd}{3:1}{M}{\delta}
\arrow[r,"\eneg{3:1}"]
& \OMG{\thetagap}{\infty}{1:1}{M}{\delta}\;.
\end{tikzcd}
\]
\label{p:3:1}
\end{proposition}

We prove Proposition~\ref{p:3:1} in Section~\ref{ss:pf:p:3:1}.
Now we define the event 
$\OMG{\thetagap}{\infty}{3:1}{M}{\delta}$.

\begin{definition}[The event $\OMG{\thetagap}{\infty}{3:1}{M}{\delta}$]
    
For $M>3$, $\thetagap\in(0,1)$, $\delta\in(0,1)$    
let $ \OMG{\thetagap}{\infty}{3:1}{M}{\delta} $ 
be the event in which all of the following
conditions are satisfied:
\begin{enumerate}[(i),font=\normalfont\bfseries,topsep=0pt]
    \item\label{c:i:3:1:1} 
        $ \OGm{\infty}{\thetagap} $ occurs;
    \item\label{c:i:3:1:2}
        $ \max_{1\leq s<\moment} 
        \absinv{\falphagaf}{s}{0}{\kdel} \le  M - \frac{5}{2} \;;$
    \item\label{c:i:3:1:3} 
            $ \invabs{\falphagaf}{\moment}{0}{\kdel} \le  M-\frac{5}{2} \;;$
    \item\label{c:i:3:1:4}
        \[
        \frac{8}{M-2}+\delta
        \le \frac{\InProdSq}{\abs*{\xi_0}^2}
        \le \fbr[\big]{ M-2 }^{1/2}-\delta\;,
        \]
    \item\label{c:i:3:1:5}
        For all $ 0 \leq i \leq m $, 
            $ \rigidity \leq j \leq m $, 
            $ 0 \leq r_1 \leq \lev $, 
            $ 0 \leq r_2 \leq \lev $: 
\begingroup
\addtolength{\jot}{0.25em}
\begin{align*}
\abs*{\sum_{l=\mij}^{\Cdel} 
      \frac{\overline{\xi_{l+i+r_1}}\cdot
      \xi_{l+j+r_2}}{\ffa{l}{i+r_1}\cdot\ffa{l}{j+r_2}}
      }
      & \leq\frac{\fbr[\big]{ M-2 }^{1/2}}{4\lev^2\fbr[\big]{\usekdm{471}}^{2\lev}}\;,\\
\abs*{\sum_{ l = \mij }^{ \Cdel } 
    \frac{ \overline{ \tail{ l + i }{ \nd } } \cdot 
    \xi_{ l + j + r_2 } }{ \ffa{ l }{ i } \cdot \ffa{ l }{ j + r_2 } } 
} 
& \leq 
\frac{ ( M-2)^{ 1 / 4 } }{ 4 \lev \fbr[\big]{\usekdm{471}}^{ \lev } }\;,\\
\abs*{ \sum_{ l = \mij }^{ \Cdel } 
    \frac{ \xi_{ l + i + r_1 } \cdot \overline{ \tail{ l + j }{ \nd } } }{ \ffa{ l }{ j } \cdot \ffa{ l }{ i + r_1 } } 
}
& \leq \frac{ \fbr[\big]{M-2}^{ 1 / 4 } }{ 4 \lev \fbr[\big]{\usekdm{471}}^{ \lev } }\;,
\\
\abs*{ \sum_{ l=\mij }^{ \Cdel } 
    \frac{ \overline{ \tail{ l + i }{ \nd } } \cdot \tail{ l + j }{ \nd } }{ \ffa{ l }{ i } \cdot \ffa{ l }{ j } } 
} 
& \leq \frac{ \fbr[\big]{M-2}^{ 1 / 4 } }{ 4 }\;.
\end{align*}
\endgroup
The constant $\usekdm{471}$ is defined in Proposition~\ref{p:rec}.
\end{enumerate}
\label{d:i:3:1}
\end{definition}

Now we show
\textit{$\OMG{\thetagap}{\infty}{3:1}{M}{\delta}$
is $\delta^3$-included in 
$\OMG{\thetagap}{\nd}{3:1}{M}{\delta}$.}

\begin{proposition}

For $M>3$, $\thetagap\in(0,1)$, and $\delta\in(0,1)$ 
sufficiently small depending on $M$, 
there exists an event $\eneg{3:2}$ such that $\pneg{3:2}<\cneg{3:2}\delta^3$ for 
some constant $\cneg{3:2}>0$ and 
\[
\OMG{\thetagap}{\infty}{3:1}{M}{\delta}
\,\setminus\,
\eneg{3:2}
\;\subset\;
\OMG{\thetagap}{\nd}{3:1}{M}{\delta}\;,
\quad\text{i.e.},\;
\begin{tikzcd}
\OMG{\thetagap}{\infty}{3:1}{M}{\delta} 
\arrow[r,"\eneg{3:2}"] & 
\OMG{\thetagap}{\nd}{3:1}{M}{\delta}\;.
\end{tikzcd}
\]
\label{p:3:2}
\end{proposition}

We prove Proposition~\ref{p:3:2} in Section~\ref{ss:pf:p:3:2}.
Next, we define the event $\OMG{\thetagap}{\infty}{3:2}{M}{\bullet}$.
The $\bullet$ symbol indicates that this event does not involve the 
parameter $\delta$.

\begin{definition}[The event $\OMG{\thetagap}{\infty}{3:2}{M}{\bullet}$]

For $M>3$, $\thetagap\in(0,1)$, 
let $ \OMG{\thetagap}{\infty}{3:2}{M}{\bullet} $ 
be the event in which all of the following conditions 
are satisfied:
\begin{enumerate}[(i),font=\normalfont\bfseries,topsep=0pt]
\item\label{c:i:3:2:1} 
$ \OGm{\infty}{\theta} $ occurs\;;
\item\label{c:i:3:2:2} 
$\max_{1\leq s<\moment} 
\absinv{\falphagaf}{s}{0}{\infty} \le M-3 $\;; 
\item\label{c:i:3:2:3} 
$ \invabs{\falphagaf}{\moment}{0}{\infty} \le M-3 $\;;
\item\label{c:i:3:2:4}
\[
\frac{8}{M-3}
\le \frac{\InProdSq}{\abs*{\xi_0}^2}
\le \fbr[\big]{ M-2 }^{1/2}-1\;;
\]
\item\label{c:i:3:2:5}
    For $0\leq i\leq m$, 
    $\rigidity\leq j\leq m$, 
    $0\leq r_1\leq \lev$, 
    $0\leq r_2\leq \lev$: 
\begingroup
\addtolength{\jot}{0.25em}
\begin{align*}
\abs*{
\sum_{l=\mij}^{\infty} 
\frac{\overline{\xi_{l+i+r_1}} \cdot \xi_{l+j+r_2}}
{\ffa{l}{i+r_1} \cdot \ffa{l}{j+r_2}}
} 
& \leq \frac{\fbr[\big]{ M-2 }^{1/2}}{4\lev^2\fbr[\big]{\usekdm{471}}^{2\lev}}-1\;,\\
\abs*{
\sum_{l=\mij}^{\infty} 
\frac{\overline{\tail{l+i}{\nd}} \cdot \xi_{l+j+r_2}}
{\ffa{l}{i} \cdot \ffa{l}{j+r_2}}
}
& \leq\frac{\fbr[\big]{ M-2 }^{1/4}}{4\lev\fbr[\big]{\usekdm{471}}^{\lev}}-1\;,\\
\abs*{\sum_{l=\mij}^{\infty} 
\frac{\xi_{l+i+r_1} \cdot \overline{\tail{l+j}{\nd}}}
{\ffa{l}{j} \cdot \ffa{l}{i+r_1}}}
& \leq\frac{\fbr[\big]{ M-2 }^{1/4}}{4\lev\fbr[\big]{\usekdm{471}}^{\lev}}-1\;,\\
\abs*{
\sum_{l=\mij}^{\infty} 
\frac{\overline{\tail{l+i}{\nd}} \cdot \tail{l+j}{\nd}}
{\ffa{l}{i} \cdot \ffa{l}{j}}
}
& \leq \frac{\fbr[\big]{ M-2 }^{1/4}}{4}-1\;.
\end{align*}
\endgroup
The constant $\usekdm{471}$ is defined in Proposition~\ref{p:rec}.
\end{enumerate}
\label{d:i:3:2}
\end{definition}

Now we show
\textit{$\OMG{\thetagap}{\infty}{3:2}{M}{\bullet}$
is $\delta^3$-included in 
$\OMG{\thetagap}{\infty}{3:1}{M}{\delta}$.}

\begin{proposition}

For $M>3$, $\thetagap\in(0,1)$, and 
$\delta\in(0,1)$ 
sufficiently small depending on $M$, 
there exists an event 
$\eneg{3:3}$ 
such that $\pneg{3:3}\leq\cneg{3:3}\delta^3$ for some constant $\cneg{3:3}>0$ and 
\[
\OMG{\thetagap}{\infty}{3:2}{M}{\bullet}
\,\setminus\,
\eneg{3:3}
\;\subset\;
\OMG{\thetagap}{\infty}{3:1}{M}{\delta}\;,
\quad\text{i.e.},\;
\begin{tikzcd}
\OMG{\thetagap}{\infty}{3:2}{M}{\bullet}
\arrow[r,"\eneg{3:3}"] & 
\OMG{\thetagap}{\infty}{3:1}{M}{\delta}\;.
\end{tikzcd}
\]
\label{p:3:3}
\end{proposition}

We prove Proposition~\ref{p:3:3} in Section~\ref{ss:pf:p:3:3}. 
Now, we show that $\PROB\fbr[\big]{\om\setminus\oj}\to 0$
which is the end goal of Step 3.

\begin{theorem}
$\lim_{j\to\infty}
\PROB\fbr[\big]{\,\om\setminus\oj\,}\to 0\;.$ 
\label{t:s:3}
\end{theorem}

We prove Theorem~\ref{t:s:3} in Section~\ref{ss:pf:t:s:3}. 
This concludes Step 3.
This also concludes 
the description of the three step procedure
of verifying the conditions of Theorem~\ref{thm:abscont}
in the context of the $\alphagaf$ ensemble.
The rest of the article is devoted to 
proving the results we have stated thus far.

\section{Proofs of the results in Section~\ref{sec:finitegaf}}
\label{sec:proof1}

\subsection{Proof of Proposition~\ref{p:inv:1}}
\label{ss:pf:p:inv:1}

\paragraph{Proof of \eqref{eq:invroot1}:}

We start with the identity (c.f. Section~2.4.1 in \cite{HKPV})
\[ 
\int\frac{1}{z^l}\Phi\fbr*{\frac{z}{R}}\deri[\alphagafn](z) 
= \int\frac{1}{z^l}\Phi\fbr*{\frac{z}{R}}\laplacian\log|\falphagafn(z)|\deri\el(z)\;,
\]
where $\laplacian$ is the operator
\[
\laplacian=\frac{1}{2\pi}
\fbr*{\frac{\partial^2}{\partial x^2}+
\frac{\partial^2}{\partial y^2}}\;.
\]
Since $\log\fbr*{\Exp\abs*{\falphagafn(z)}^2}^{1/2}$ is a radial 
function and Laplacian of a radial function is also radial, we have  
\[
\int\frac{1}{z^l}\Phi\fbr*{\frac{z}{R}}
\laplacian\log\fbr[\Big]{\Exp\abs*{\falphagafn(z)}^2}^{1/2}\deri\el(z)
=0\;.
\]
Let 
\[
\widehat{\mathcal{F}}_{\alpha,n}(z)\coloneqq
\frac{\falphagafn(z)}{\fbr*{\Exp\abs*{\falphagafn(z)}^2}^{1/2}}\;.
\]
Then the above argument implies
\[
\int\frac{1}{z^l}\Phi\fbr*{\frac{z}{R}}\deri[\alphagafn](z) 
=\int\frac{1}{z^l}\Phi\fbr*{\frac{z}{R}} 
\laplacian
\log\abs*{\widehat{\mathcal{F}}_{\alpha,n}(z)}\deri\el(z)\;.
\]
Integrating by parts the right hand side we have
\[
\abs*{\int\frac{1}{z^l}\Phi\fbr*{\frac{z}{R}}\deri[\alphagafn](z)} 
\leq\int
\abs*{\laplacian\fbr*{\frac{1}{z^l}\Phi\fbr*{\frac{z}{R}}}}
\abs*{\log|\widehat{\mathcal{F}}_{\alpha,n}(z)|} 
\deri\el(z)\;.
\]
Therefore
\[
\Exp\tbr*{\abs*{\int\frac{1}{z^l}
\Phi\fbr*{\frac{z}{R}}\deri[\alphagafn](z)}} 
\leq
\int
\abs*{\laplacian\fbr*{\frac{1}{z^l}\Phi\fbr*{\frac{z}{R}}}}
\fbr*{\Exp\tbr*{\abs*{\log|\widehat{\mathcal{F}}_{\alpha,n}(z)|}}} 
\deri\el(z)\;.
\]
Since $\Phi$ is supported in the annulus 
$\Set*{ z\in\CC \given r_0\leq \abs*{z} \leq x_3 r_0}$, 
we have
\[
\abs*{\laplacian\fbr*{\frac{1}{z^l}\Phi\fbr*{\frac{z}{R}}}}
\le 
\usec{pf:p:inv:1:1}(\Phi) \cdot l^2
\cdot \frac{1}{r_0^{l+2}}
\cdot \frac{1}{R^{l+2}}
\]
for some constant $\usec{pf:p:inv:1:1}(\Phi)>0$.
Further, 
$\Exp\abs*{\log|\widehat{\mathcal{F}}_{\alpha,n}(z)|}$
is a constant because $\widehat{\mathcal{F}}_{\alpha,n}$ is $N_{\CC}(0,1)$.
Therefore 
\[
\int
\abs*{
\laplacian\fbr*{\frac{1}{z^l}\Phi\fbr*{\frac{z}{R}}}
}
\fbr*{\Exp\tbr*{\abs*{
\log|\widehat{\mathcal{F}}_{\alpha,n}(z)|
}}}
\deri\el(z) 
\leq 
\usec{pf:p:inv:1:2}(\Phi)\cdot 
l^2\cdot 
\frac{1}{r_0^l}\cdot
\frac{1}{R^l}\;,
\]
for some constant $\usec{pf:p:inv:1:1}(\Phi)>0$.
This proves \eqref{eq:invroot1}.

\paragraph{Proof of \eqref{eq:invroot2}:}

Let $\bbK_n$ be the covariance kernel of $\falphagafn$, that is $\bbK_n(z,w)=\sum_{k=0}^n\frac{(z\overline{w})^k}{(k!)^\alpha}$.
Since $\Phi$ is a radial function on $\CC$,
there exists a function $\widetilde{\Phi}$ on 
$\RR_{\ge 0}$ such that $\Phi(z)=\widetilde{\Phi}(|z|)$. 
Therefore 
\[
\Exp\tbr*{\int\frac{1}{|z|^l}\Phi\fbr*{\frac{z}{R}} \deri[\alphagafn](z)}
=\usec{pf:p:inv:1:3}
\int\frac{1}{r^l}\widetilde{\Phi}\fbr*{\frac{r}{R}}
\laplacian\log\fbr*{\bbK_n(r,r)}^{1/2} r \deri r\;,
\]
for some constant $\usec{pf:p:inv:1:3}>0$.
Integrating by parts we get
\[
\abs*{\int\frac{1}{r^l}\widetilde{\Phi}\fbr*{\frac{z}{R}}
\laplacian\log\fbr*{\bbK_n(r,r)}^{1/2} 
r \deri r } 
\leq 
\int
\abs*{\laplacian\fbr*{\frac{1}{r^l}
\widetilde{\Phi}\fbr*{\frac{z}{R}}}}
\log\fbr*{\bbK_n(r,r)}^{1/2} r \deri r\;. 
\]
For $r/R$ in the support of $\widetilde{\phi}$ we have 
\begingroup
\addtolength{\jot}{0.26em}
\begin{align*}
\log \fbr*{\bbK_n(r,r)}^{1/2}
\leq{} & \log\fbr*{\bbK(r,r)}^{1/2}\\
\leq{} & \usec{pf:p:inv:1:4}(\alpha)\fbr*{ r^{2/\alpha}+
\abs*{\log r}}
\leq{} \usec{pf:p:inv:1:5}(\alpha)\fbr*{ r_0^{2/\alpha} +
\abs*{\log r_0} + \usec{pf:p:inv:1:6}} 
R^{2/\alpha}
\end{align*}
\endgroup
and 
\[
\abs*{\laplacian\fbr*{\frac{1}{r^l}\Phi\fbr*{\frac{r}{R}}}}\le \usec{pf:p:inv:1:7}(\Phi)\cdot l^2 \cdot \frac{1}{R^{l+2}} \cdot \frac{1}{r_0^{l+2}}\;.
\]
Therefore  
\[
\int
\abs*{\laplacian\fbr*{\frac{1}{r^l}\widetilde{\Phi}\fbr*{\frac{r}{R}}}}\log\fbr*{\bbK_n(z,z)} r \deri r 
\leq \usec{pf:p:inv:1:8}(\Phi)
\cdot l^2
\cdot \frac{1}{R^{l-\frac{2}{\alpha}}}
\cdot \frac{1}{r_0^{l-\frac{2}{\alpha}}}\;.
\]
This proves \eqref{eq:invroot2} and concludes the proof of Proposition~\ref{p:inv:1}. 

\subsection{Proof of Proposition~\ref{p:inv:2}}
\label{ss:pf:p:inv:2}

We have 
\[
\Exp\tbr*{\abs*{\;
\int\frac{\widetilde{\varphi}(z)}{z^l}
\deri[\alphagafn](z)}
\;}
\leq
\Exp\tbr*{\;
\int\frac{\widetilde{\varphi}(z)}{|z|^l}
\deri[\alphagafn](z)
\;}
=\usec{pf:p:inv:2:1}
\int\frac{\widetilde{\varphi}(z)}{|z|^l}
\laplacian\log\bbK_n(z,z)\deri\el(z)
\]
for some constant $\usec{pf:p:inv:2:1}>0$. 
Recall $0\leq\widetilde{\varphi}\leq 1$.
Using the uniform convergence of 
the continuous functions 
$\laplacian\log\bbK_n(z,z)\to\laplacian\log\bbK(z,z)<\infty$
on the (compact) support of 
$\widetilde{\varphi}$, we deduce that
\[
\int\frac{\widetilde{\varphi}(z)}{|z|^l}
\laplacian\log\bbK_n(z,z)\deri\el(z)\leq 
\usec{pf:p:inv:2:2}\cdot l^2 \cdot \frac{1}{r_0^l}
\]
for some constant $\usec{pf:p:inv:2:2}>0$. 
Letting $n\to\infty$ we obtain the same bound for $\alphagaf$.

\subsection{Proof of Proposition~\ref{p:inv:3}}
\label{ss:pf:p:inv:3}

Consider $n\in\NN$ and $l\geq 1$. 
Using Propositions~\ref{p:inv:1} and 
\ref{p:inv:2} we get:
\[
\Exp\tbr*{\absinv{\falphagafn}{l}{j}{j+1}}
=\Exp\tbr*{\abs*{\int\frac{\phi_j(z)}{z^l}
\deri[\alphagafn]}}
< \usec{pf:p:inv:3:1}\cdot 
l^2\cdot 
\frac{1}{r_0^l}\cdot 
\exp\fbr*{ - j l }  
\]
for all $j\geq 0$. For $j=0$, the upper bound is given by Proposition~\ref{p:inv:2}. For $j>0$, the upper bound is given by (\ref{eq:invroot2}) with $R=e^{j}$ and $\Phi=\phi_j$. Hence for all $j\geq 0$ we have 
\[
\sum_{j^\prime=j}^{\infty}
\Exp\tbr*{\absinv{\falphagafn}{l}{j^\prime}{j^\prime+1}}
< \usec{pf:p:inv:3:2} \cdot 
l^2 \cdot 
\frac{1}{r_0^l} \cdot 
\exp\fbr*{ - j l }\;.   
\]
Hence we get that the infinite sum in the statement 
is well-defined and the bound on the expectation 
also follows. The bound on the probability follows
simply from Markov's Inequality.

\subsection{Proof of Proposition~\ref{p:van}}\label{ss:pf:p:van}

Suppose the event $\OGm{n}{\thetagap}$ occurs.
Let $\uo$ be a vector consisting of points of $\alphagafnout$.
For $\uz,\uzp\in\Dset^m$ we have
\[
\frac{\van{\uzp\cc\uo}}{\van{\uz\cc\uo}}
\;=\;
\frac{\van{\uzp}}{\van{\uz}}\cdot
\frac{\vancross{\uzp}{\uo}}{\vancross{\uz}{\uo}}\;. 
\]
To bound $\abs*{\vancross{\uzp}{\uo}/\vancross{\uz}{\uo}}$ 
from above and below uniformly in $\uz,\uzp\in\Dset^m$, 
it is sufficient to bound 
$\abs*{\vancross{\uzp}{\uo}/\vancross{\underline{0}}{\uo}}$ 
from above and below uniformly in $\uz\in\Dset^m$. Here 
$\underline{0}\in\Dset^m$ is the vector of all zeros.
Observe that
\[
\abs*{\frac{\vancross{\uzp}{\uo}}{\vancross{\underline{0}}{\uo}}}=
\prod_{i=1}^m
\prod_{j=1}^{n-m}
\abs*{\frac{\zeta_i-\omega_j}{\omega_j}}\;.
\]
Therefore, to bound 
$\abs*{\vancross{\uzp}{\uo}/\vancross{\underline{0}}{\uo}}$
it suffices to bound 
$\prod_{j=1}^{n-m}\abs*{\frac{\zeta_0-\omega_j}{\omega_j}}$ 
uniformly for $\zeta_0\in\Dset$. 
Therefore, let us fix a $\zeta_0\in\Dset$.
Due to the $\thetagap$-separation between $\partial\Dset$ and $\uo$,
the ratio $\theta_j\coloneqq\frac{\zeta_0}{\omega_j}$ satisfies 
\[
|\theta_j|\leq\frac{r_0}{r_0+\thetagap}<1\;.
\]
Thus 
\[
0 < \frac{\thetagap}{r_0+\thetagap} \leq 1-|\theta_j|.
\]
Let $\log$ be the branch of complex logarithm
given by the power series 
\[
\log\fbr*{ 1-z} = - \sum_{k=1}^\infty\frac{z^k}{k}
\]
for $|z|<1$. Then we have 
\[
\log|1-\theta_j|
= \mathfrak{R}\log(1-\theta_j)
= - \mathfrak{R}\fbr*{\sum_{k=1}^{\moment-1}\frac{\theta_j^k}{k}}
+f(\theta_j)\;,
\]
where 
\[
f(\theta_j) = -\mathfrak{R}\fbr*{\sum_{k=\moment}^{\infty}
\frac{\theta_j^k}{k}}\;.
\]
Then
\[
\abs*{f(\theta_j)}
\leq \frac{|\theta_j|^{\moment}}{1-|\theta_j|}
\leq \frac{r_0+\thetagap}{\thetagap}|\theta_j|^{\moment}\;.
\]
Hence for some constant $\usekd{pf:p:van}>0$ we have:
\begin{align*}
\abs*{\log\fbr*{\prod_{j=1}^{n-m}\abs*{\frac{\zeta_0-\omega_j}{\omega_j}}}}
= & \abs*{\sum_{j=1}^{n-m}\log\abs*{\frac{\zeta_0-\omega_j}{\omega_j}}}\\
= & \abs*{\mathfrak{R}
\fbr*{\sum_{k=1}^{\moment-1}\frac{1}{k}\sum_{j=1}^{n-m}\theta_j^k}
+\fbr*{\sum_{j=1}^{n-m}f(\theta_j)}}\\
\leq & \sum_{k=1}^{\moment-1}\frac{1}{k}\abs*{\sum_{j=1}^{n-m}\theta_j^k}
+\frac{r_0+\thetagap}{\thetagap}\sum_{j=1}^{n-m}|\theta_j|^{\moment}\\
< & \sum_{k=1}^{\moment-1}
\frac{r_0^k}{k}
\abs*{\sum_{j=1}^{n-m}\frac{1}{\omega_j^k}}
+\frac{(r_0+1)r_0^{\moment}}{\thetagap}\sum_{j=1}^{n-m}\frac{1}{|\omega_j|^{\moment}}\\
\leq{} & \usekd{pf:p:van}\thetagap^{-1}\X_n\;.
\end{align*}
Here we have used the fact $\theta<1$ and the definition (c.f. Notation \ref{n:gafsum} for GAF-s)
\[\X_n\coloneqq{} 
\sum_{k=1}^{\moment-1} 
\abs*{\sum_{\omega_j\in\alphagafnout} 
\frac{1}{\omega_j^k}}
+ \sum_{\omega_j\in\alphagafnout} 
\frac{1}{|\omega_j|^{\moment}}.\]
Thus, we get a bound not involving $\zeta_0$. 
This completes the proof of Proposition~\ref{p:van}.

\subsection{Proof of Proposition~\ref{p:sym}}
\label{ss:pf:p:sym}

Consider 
$n$, $m$, $\us$, $\uo$, $\uz$, $\uzp$ 
as in the statement
of this proposition.
We start with some simple observations:
\begin{gather}
\mbox{for all $0\leq k\leq n\,,$}
\quad
\sym{k}{\uz\cc\uo} = 
\sum_{i=0}^{m}\sym{i}{\uz}\sym{k-i}{\uo},
\;\mbox{ and }\;
\sym{k}{\uzp\cc\uo} = 
\sum_{i=0}^{m}\sym{i}{\uzp}\sym{k-i}{\uo}\;;\label{eq:p:sym:p:1}\\
\mbox{for all $1\le k\le\rigidity-1\,,$}
\quad
\sym{k}{\uz}=\sym{k}{\uzp}=s_k\,;
\qquad
\sym{0}{\uz}=\sym{0}{\uzp}=1\;;\label{eq:p:sym:p:2}\\
\mbox{for all $1\le i\le m\,,$}
\quad
\abs*{\sym{i}{\uz}}<{\binom{m}{i}}r_0^i\,,
\;\mbox{ and }
\abs*{\sym{i}{\uzp}}<{\binom{m}{i}}r_0^i\;.
\label{eq:p:sym:p:3}
\end{gather}
Using \eqref{eq:p:sym:p:1} and
\eqref{eq:p:sym:p:2}
we get for all $0\leq k\leq n$
\[
\sym{k}{\uzp\cc\uo}
=\sym{k}{\uz\cc\uo}
+\sum_{j=\rigidity}^{m}\fbr*{\sym{j}{\uzp}-\sym{j}{\uz}}\cdot
\sym{k-j}{\uo}\;.
\]
Using the identity
\[
\abs*{\sum_p a_p}^2 
= \sum_p\abs*{a_p}^2 
+ 2\sum_{p<q} \Re\fbr*{ a_p \overline{a_q}}\;, 
\]
we get
\begingroup
\addtolength{\jot}{0.25em}
\begin{align*}  
\abs*{\sym{k}{\uzp\cc\uo}}^2 ={} 
& \abs*{\sym{k}{\uz\cc\uo}}^2 \\ 
& + \sum_{j=\rigidity}^m 
\abs*{\sym{j}{\uzp}-\sym{j}{\uz}}^2
\cdot\abs*{\sym{k-j}{\uo}}^2\\ 
& + 
\sum_{j=\rigidity}^m 
2\Re\fbr*{\fbr[\big]{\sym{j}{\uzp}-\sym{j}{\uz}}\cdot
\overline{\sym{k}{\uz\cc\uo}}\cdot 
\sym{k-j}{\uo}}\\ 
& + 
\sum_{\rigidity\leq i<j\leq m}
2\Re\fbr*{\overline{\fbr[\big]{\sym{i}{\uzp}-\sym{i}{\uz}}}\cdot
\fbr[\big]{\sym{j}{\uzp}-\sym{j}{\uz}}\cdot 
\overline{\sym{k-i}{\uo}}
\cdot\sym{k-j}{\uo}}\;. 
\end{align*}
Dividing throughout by $\ncka$ and then summing the above over 
$0\leq k\leq n$ we get
\begingroup
\addtolength{\jot}{0.25em}
\begin{align*}
\vecsymn{n}{0}{(\uzp\cc\uo)} 
={} & \vecsymn{n}{0}{(\uz\cc\uo)}\\
& + \sum_{j=\rigidity}^{m}
\abs*{\sym{j}{\uzp}-\sym{j}{\uz}}^2\cdot
\Big\langle\vecsymb{n}{j}{\uo},
\vecsym{n}{j}{\uo}\Big\rangle
\\
& +  
\sum_{j=\rigidity}^{m} 
2\Re\fbr*{\sym{j}{\uzp}-\sym{j}{\uz}} 
\cdot \Big\langle\vecsymb{n}{0}{(\uz\cc\uo)}
,\vecsym{n}{j}{\uo}\Big\rangle
\\
 & + 
\sum_{\rigidity\leq i<j\leq m}
2\Re\fbr*{\overline{\fbr[\big]{\sym{i}{\uzp}-\sym{i}{\uz}}}\cdot 
\fbr[\big]{\sym{j}{\uzp}-\sym{j}{\uz}}\cdot
\Big\langle\vecsymb{n}{i}{\uo}
\cdot\vecsym{n}{j}{\uo}\Big\rangle}\;.
\end{align*}
\endgroup
Using triangle inequality we get  
\begin{equation}\label{eq:p:sym:p:4}
\fD{\uz}{\uo} 
- \FA{\uz}{\uzp}{\uo}  
\le \fD{\uzp}{\uo} 
\le \fD{\uz}{\uo} 
+ \FA{\uz}{\uzp}{\uo}\;,
\end{equation}
where 
\begingroup
\addtolength{\jot}{0.25em}
\begin{align*} 
\FA{\uz}{\uzp}{\uo}\coloneqq{} 
& \sum_{j=\rigidity}^m 
\abs*{\sym{j}{\uzp}-\sym{j}{\uz}}^2 
\cdot\abs*{\Big\langle\vecsymb{n}{j}{\uo},
\vecsym{n}{j}{\uo}\Big\rangle}
\\
& + 
2 \sum_{j=\rigidity}^{m} 
\abs*{\sym{j}{\uzp}-\sym{j}{\uz}}\cdot 
\abs*{\Big\langle\vecsymb{n}{0}{(\uz\cc\uo)},
\vecsym{n}{j}{\uo}\Big\rangle}\\
& + 
2 \sum_{\rigidity\leq i<j\leq m} 
\abs*{\sym{i}{\uzp}-\sym{i}{\uz}}\cdot
\abs*{\sym{j}{\uzp}-\sym{j}{\uz}}\cdot 
\abs*{\Big\langle\vecsymb{n}{i}{\uo},
\vecsym{n}{j}{\uo}\Big\rangle}\;.
\end{align*}
\endgroup
Dividing throughout by $\fD{\uz}{\uo}$ in \eqref{eq:p:sym:p:4}
we get
\begin{equation}\label{eq:p:sym:p:5}
1-\FB{\uz}{\uzp}{\uo}
\leq
\frac{\fD{\uzp}{\uo}}{\fD{\uz}{\uo}}
\leq
1+\FB{\uz}{\uzp}{\uo}\;,
\end{equation}
where
\begingroup
\addtolength{\jot}{0.25em}
\begin{align*} 
\FB{\uz}{\uzp}{\uo}\coloneqq{} 
& \sum_{j=\rigidity}^m 
\abs*{\sym{j}{\uzp}-\sym{j}{\uz}}^2\cdot 
\frac{\abs*{\Big\langle\vecsymb{n}{0}{\uo},
\vecsym{n}{j}{\uo}\Big\rangle}}
{\vecsymn{n}{0}{(\uz\cc\uo)}}
\\
& + 
2 \sum_{j=\rigidity}^{m} 
\abs*{\sym{j}{\uzp}-\sym{j}{\uz}}\cdot 
\frac{\abs*{\Big\langle\vecsymb{n}{0}{(\uz\cc\uo)},
\vecsym{n}{j}{\uo}\Big\rangle}}
{\vecsymn{n}{0}{(\uz\cc\uo)}}\\
& + 
2 \sum_{\rigidity\leq i<j\leq m} 
\abs*{\sym{i}{\uzp}-\sym{i}{\uz}}\cdot
\abs*{\sym{j}{\uzp}-\sym{j}{\uz}}\cdot 
\frac{\abs*{\Big\langle\vecsymb{n}{i}{\uo},
\vecsym{n}{j}{\uo}\Big\rangle}}
{\vecsymn{n}{0}{(\uz\cc\uo)}}\;.
\numberthis\label{eq:p:sym:p:6}\end{align*}
\endgroup
Using \eqref{eq:p:sym:p:1} we get 
    \begin{align*}
        \abs*{\Big\langle\vecsymb{n}{0}{(\uz\cc\uo)},
        \vecsym{n}{j}{\uo}\Big\rangle} & = \abs*{\sum_{k=0}^n\frac{\overline{\sigma_k(\uz\cc\uo)}\sigma_{k-j}(\uo)}{\fbr*{\binom{n}{k}k!}^{\alpha}}}\\
        & = \abs*{\sum_{k=0}^n\sum_{i=0}^m\frac{\overline{\sigma_i(\uz)}\overline{\sigma_{k-i}(\uo)}\sigma_{k-j}(\uo)}{\fbr*{\binom{n}{k}k!}^{\alpha}}}\\
        & = \abs*{\sum_{i=0}^m\overline{\sigma_i(\uz)}\sum_{k=0}^n\frac{\overline{\sigma_{k-i}(\uo)}\sigma_{k-j}(\uo)}{\fbr*{\binom{n}{k}k!}^{\alpha}}}\\
        & \leq \sum_{i=0}^m \abs*{\sigma_i(\uz)}\cdot\abs*{\Big\langle\vecsymb{n}{i}{(\uo)},
        \vecsym{n}{j}{\uo}\Big\rangle}, 
    \end{align*}
    and thus
\begingroup
\addtolength{\jot}{0.25em}
\begin{align*}
& \sum_{j=\rigidity}^{m} 
\abs*{\sym{j}{\uzp}-\sym{j}{\uz}}\cdot
\frac{\abs*{\Big\langle\vecsymb{n}{0}{(\uz\cc\uo)},
\vecsym{n}{j}{\uo}\Big\rangle}}
{\vecsymn{n}{0}{(\uz\cc\uo)}}\\ 
\leq{} &
\sum_{i=0}^{m}
\sum_{j=\rigidity}^{m}
\abs*{\sym{j}{\uzp}-\sym{j}{\uz}}\cdot
\abs*{\sym{i}{\uz}}\cdot
\frac{\abs*{\Big\langle\vecsymb{n}{i}{\uo},
\vecsym{n}{j}{\uo}\Big\rangle}}
{\vecsymn{n}{0}{(\uz\cc\uo)}}\;.
\numberthis\label{eq:p:sym:p:7}
\end{align*}
\endgroup
Combining \eqref{eq:p:sym:p:6}, \eqref{eq:p:sym:p:7} and using \eqref{eq:Dhatdef} and \eqref{eq:p:sym:p:3} we get
\[
\FB{\uz}{\uzp}{\uo}
\leq \usekdm{pf:p:sym} \cdot \fDhat{\uz}{\uo}
\]
for some $\usekdm{pf:p:sym}$.
Therefore, from \eqref{eq:p:sym:p:5} we get
\[
1 - \usekdm{pf:p:sym}\cdot\fDhat{\uz}{\uo}
\leq \frac{\fD{\uzp}{\uo}}{\fD{\uz}{\uo}}
\leq 1 + \usekdm{pf:p:sym}\cdot\fDhat{\uz}{\uo}\;.
\]
This concludes the proof of Proposition~\ref{p:sym}.

\subsection{Proof of Proposition~\ref{p:rec}}
\label{ss:pf:p:rec}

Suppose the event $\omn$ happens.
Let $\uz$ be a vector consisting 
of the roots of $\falphagafn$ 
inside $\Dset$.
Fix $1\leq k\leq n-m$. 
Then we have
\[
\sym{k}{\uz\cc\uo}=
\sum_{r=0}^{\min\{k,m\}}
\sym{r}{\uz}
\sym{k-r}{\uo}\;.
\]
It follows that
\[
\sym{k}{\uo}
=\sym{k}{\uz\cc\uo}-
\sum_{r=1}^{\min\{k,m\}}
\sym{r}{\uz}\sym{k-r}{\uo}\;. 
\]
We can similarly expand each of the lower order term
$\sym{k-r}{\uo}$ in terms of 
$\sym{j}{\uz\cc\uo}$ and obtain an expansion of $\sym{k}{\uo}$
in terms of $\sym{j}{\uz\cc\uo}$, $j=1,\dots,k$.
In this way we get
\begin{equation}\label{eq:pf:p:rec:1}
\sym{k}{\uo}=\sym{k}{\uz\cc\uo}
+\sum_{r=1}^k\g_r\;\sym{k-r}{\uz\cc\uo}\;.
\end{equation}
The coefficient of $\sym{k}{\uz\cc\uo}$ is $1$. 
The rest of the coefficients are polynomials in $\sym{j}{\uz}$, 
$j=1,\dots,m$.
They satisfy the recurrence relation
\begin{equation}\label{eq:pf:p:rec:2}
\begin{pmatrix}
\g_i\\
\g_{i-1}\\
\vdots\\
\g_{i-m+1}
\end{pmatrix}
= A 
\begin{pmatrix}
\g_{i-1}\\
\g_{i-2}\\
\vdots \\
\g_{i-m}
\end{pmatrix}
\end{equation}
where 
\[
A\coloneqq\begin{pmatrix}
-\sym{1}{\uz} & -\sym{2}{\uz} & \cdots & -\sym{m}{\uz}\\
1 & 0 & \cdots & 0\\
\vdots & \vdots & \ddots & \vdots\\
0 & 0 & 1 & 0
\end{pmatrix}
\]
with the boundary conditions $\g_i=-\sym{i}{\uz}$ 
for $i=0,\dots,m-1$ and $\g_i=0$ for $i<0$. The 
eigenvalues of $A$ are precisely the negatives of 
the inside zeros $\zeta_1,\dots,\zeta_m$. Due to 
the recursive structure in \eqref{eq:pf:p:rec:2}, $\g_k$ is 
an element of $A^k$ applied to the vector 
$\fbr*{-\sym{m-1}{\uz},\dots,-\sym{0}{\uz}}$. This 
implies $\g_k$ is a linear combination of the 
entries of $A^k$ (with the coefficients in the 
linear combination being independent of $k$, but 
depending on $\Dset$ and $m$). If the eigenvalues 
of a matrix $A$ have modulus $<\rho^\prime$, then 
the entries of the matrix $A^k$ are $o(\rho^{\prime k})$. 
Now the eigenvalues of $A$ are $\zeta_i$'s and for 
each $i$ we have $|\zeta_i|\leq r_0$. 
Therefore, for $1\leq r\leq n$ we have 
$\abs*{\g_r} \leq \fbr*{\usekdm{pfrec1}}^r$
for some constant $\usekdm{pfrec1}>0$.

Switching variable to $l=n-k$ in \eqref{eq:pf:p:rec:1} 
we get  
\[
\frac{\sym{n-l}{\uo}}
{\fbr*{\binom{n}{n-l}(n-l)!}^{\alpha/2}} = 
\sum_{r=0}^{n-l}(-1)^{n-l-r}\cdot  
\g_r\cdot 
\frac{\xi_{l+r}}{\xi_n}\cdot
\frac{1}{\ffa{l}{r}}
\]
with $\g_0\coloneqq1$. 

Using $\ffa{l}{r} \geq l^{r\alpha/2}$ 
and the fact that $\Exp\abs*{\xi}^2=1$
if $\xi$ is a complex Gaussian random variable, 
we get 
\[
\fbr*{\Exp\fbr*{\abs*{\tail{l}{n}}^2
\indicator{\omn}}}^{1/2}
\leq 
\sum_{r=\lev}^{n-l}
\frac{\fbr*{\usekdm{pfrec1}}^r}{l^{r\alpha/2}}
\leq 
\frac{\fbr[\big]{\usekdm{pfrec2}}^{\lev}}{l^{\lev\alpha/2}}
\]
for $l\geq\usekdm{pfrec3}$. 
This concludes the proof of 
Proposition~\ref{p:rec}.

\section{Proofs of the results in Step 1 of Section~\ref{sec:limitinggaf}}

\subsection{Some auxiliary lemmas}

\begin{lemma}\label{l:kd/i:1}
Let $\bbF$ be either $\falphagaf$ or $\falphagafn$ for some $n\in\NN$. 
Then, except on a $\delta^3$-negligible event, we have:
\begingroup
\addtolength{\jot}{0.26em}
\begin{align*}
\max_{1\leq l\leq \Cdel}
\abs*{\inv{\bbF}{l}{0}{\infty}-\inv{\bbF}{l}{0}{\kdel}}
\leq{} & \delta\wedge\edel\;;\\
\abs*{\invabs{\bbF}{\moment}{0}{\infty}-
\invabs{\bbF}{\moment}{0}{\kdel}}
\leq{} & \delta\wedge\edel\;.
\end{align*}
\endgroup
\end{lemma}

\begin{proof}
From Propositions~\ref{p:inv:3} 
and condition~\ref{c:kd:1}
defining $\kdel$ we have 
\begingroup
\addtolength{\jot}{0.26em}
\begin{align*}
& \sum_{1\leq l\leq \Cdel}
    \PROB \fbr*{ 
    \absinv{ \bbF }{ l }{ \kdel }{ \infty }
    > \usec{p43}\cdot l^2\cdot \frac{1}{r_0^l}\cdot 
    \exp \fbr*{ - \kdel \frac{l}{2} }
    } \\
\leq{} & \sum_{1\leq l\leq \Cdel} \exp \fbr*{ - \kdel \frac{l}{2} } 
\;\leq\; \delta^3\;.
\end{align*}
\endgroup
Therefore, using condition~\ref{c:kd:2} defining $\kdel$,
we get that except on a $\delta^3$-negligible event
\begingroup
\addtolength{\jot}{0.26em}
\begin{align*}
& \max_{1\leq l\leq \Cdel}
\abs*{\inv{\bbF}{l}{0}{\infty}-\inv{\bbF}{l}{0}{\kdel}}\\
\leq{} & \max_{1\leq l\leq \Cdel} \usec{p43}\cdot l^2\cdot \frac{1}{r_0^l}\cdot \exp\fbr*{-\kdel\frac{l}{2}}
\;\leq\; \delta\wedge\edel\;.
\end{align*}
\endgroup
Using Propositions~\ref{p:inv:4} and 
condition~\ref{c:kd:1} defining $\kdel$,
we have
\begingroup
\addtolength{\jot}{0.26em}
\begin{align*}
& \PROB \fbr*{ 
\invabs{ \bbF }{ \moment }{ \kdel }{ \infty }
> \usec{p44}\cdot \sqmoment\cdot \frac{1}{r_0^{\moment}}
\cdot \exp\fbr*{-\kdel\frac{\moment}{2}} 
} \\
\leq{} & \exp\fbr*{-\kdel\frac{\moment}{2}} 
\;\leq\; \delta^3\;.
\end{align*}
\endgroup
Therefore, using condition~\ref{c:kd:2} defining $\kdel$,
we get that except on a $\delta^3$-negligible event 
\begingroup
\addtolength{\jot}{0.26em}
\begin{align*}
& \abs*{\invabs{\bbF}{\moment}{0}{\infty}-
\invabs{\bbF}{\moment}{0}{\kdel}}\\
\leq{} &  
\usec{p44} \cdot \sqmoment \cdot \frac{1}{r_0^{\moment}}
\cdot \exp\fbr*{-\kdel\frac{\moment}{2}} 
\;\leq\; \delta\wedge\edel\;.
\end{align*}
\endgroup
This concludes the proof of Lemma~\ref{l:kd/i:1}.
\end{proof}

\begin{lemma}\label{l:kd/i:2} 
Consider $\thetagap\in(0,1)$ and $\delta\in(0,1)$.
Except on a $\delta^3$-negligible event we have:
\begingroup
\addtolength{\jot}{0.25em}
\begin{align*}
\max_{0\leq i\leq m}\;
\max_{\rigidity\leq j\leq m} 
\abs*{\QFI{\zpv}-\QFI{\zppv}}&<\delta\;;\\ 
\abs*{\QFZ{\zpv}-\QFZ{\zppv}}&<\delta\;.  
\end{align*}
\endgroup
\end{lemma}

\begin{proof}
Condition~\ref{c:gd:1} defining $\gdel$ implies that
except on a $\delta^3$-negligible event we have 
\[
\max_{1 \leq l \leq \Cdel}
\absinv{\fand}{l}{0}{\kdel}
\leq \gdel-1\;. 
\]
By Lemma~\ref{l:kd/i:1} we have that except on a $\delta^3$-negligible event
\[
\max_{1\leq l\leq\Cdel}
\bigg|\inv{\fand}{l}{0}{\infty}-
\inv{\fand}{l}{0}{\kdel}\bigg|
<\edel\;.
\]
Therefore, using condition~\ref{c:ed:1} defining $\edel$ 
we get that except on a $\delta^3$-negligible event 
\[
\max_{0\leq i\leq m}\;
\max_{\rigidity\leq j\leq m}
\abs*{\QFI{\zpv}-\QFI{\zppv}}<\delta\;.
\]
Similarly, from condition~\ref{c:ed:2} defining $\edel$ 
we get that except on a $\delta^3$-negligible event  
\[
\abs*{\QFZ{\zpv}-\QFZ{\zppv}}<\delta\;.
\]
This concludes the proof of Lemma~\ref{l:kd/i:2}.
\end{proof}

\begin{lemma}\label{l:nd/i:1}
Consider $\thetagap\in(0,1)$ and $\delta\in(0,1)$.
Except on a $\delta^3$-negligible we have:
\begingroup
\addtolength{\jot}{0.25em}
\begin{align*}
\max_{1\leq l\leq \Cdel}
\bigg|\inv{\falphagaf}{l}{0}{\kdel}-
\inv{\fand}{l}{0}{\kdel}\bigg|
\leq\delta\wedge\edel\;;\numberthis\label{eq:nd/i:1:1}\\
\bigg|\invabs{\falphagaf}{\moment}{0}{\kdel}-
\invabs{\fand}{\moment}{0}{\kdel}\bigg|
\leq\delta\wedge\edel\;.\numberthis\label{eq:nd/i:1:2}
\end{align*}
\endgroup
\end{lemma}

\begin{proof}
From condition~\ref{c:nd:1} defining $\nd$ we have except on a $\delta^3$-negligible event 
\[
\max_{1\leq l\leq \Cdel}
\max_{0\leq k\leq \kdel}
\bigg|\inv{\falphagaf}{l}{k}{k+1}-
\inv{\fand}{l}{k}{k+1}\bigg|
\leq\frac{\delta\wedge\edel}{\kdel+1}\;.
\]
Using triangle inequality we get \eqref{eq:nd/i:1:1}.
From condition~\ref{c:nd:1} defining $\nd$ we have
except on a $\delta^3$-negligible event
\[
\max_{0\leq k\leq \kdel}
\bigg|\invabs{\falphagaf}{\moment}{k}{k+1}-
\invabs{\fand}{\moment}{k}{k+1}\bigg|
\leq\frac{\delta\wedge\edel}{\kdel+1}\;.
\]
Using triangle inequality we get \eqref{eq:nd/i:1:2}.
\end{proof}

\begin{lemma}\label{l:nd/i:2} 
Consider $\thetagap\in(0,1)$ and $\delta\in(0,1)$.
Except on a $\delta^3$-negligible event we have:
\begingroup
\addtolength{\jot}{0.25em}
\begin{align*}
\max_{0\leq i\leq m}\;
\max_{\rigidity\leq j\leq m} 
\abs*{\QFI{\zzv}-\QFI{\zpv}}<{}&\delta\;;\\ 
\abs*{\QFZ{\zzv}-\QFZ{\zpv}}<{}&\delta\;.  
\end{align*}
\endgroup
\end{lemma}

\begin{proof}
Condition~\ref{c:gd:1} defining $\gdel$ implies that
except on a $\delta^3$-negligible event we have 
\begingroup 
\addtolength{\jot}{0.25em}
\begin{align*}
\max_{1\leq l\leq \Cdel}
\absinv{\fand}{l}{0}{\kdel}
& \leq \gdel-1\;,\\ 
\max_{1\leq l\leq \Cdel}
\absinv{\falphagaf}{l}{0}{\kdel}
& \leq \gdel-1\;. 
\end{align*}
\endgroup
From Lemma~\ref{l:nd/i:1} we have that 
except on a $\delta^3$-negligible event
\[
\max_{1\leq l\leq \Cdel}
\abs*{\inv{\falphagaf}{l}{0}{\kdel}-
\inv{\fand}{l}{0}{\kdel}}
<\edel\;.
\]
Therefore, from condition~\ref{c:ed:1} defining $\edel$ we get that  
except on a $\delta^3$-negligible event 
\[
\max_{0\leq i\leq m}\;
\max_{\rigidity\leq j\leq m}
\abs*{\QFI{\zzv}-\QFI{\zpv}} < \delta\;.
\]
Similarly, from condition~\ref{c:ed:2} defining $\edel$ we get that 
except on a $\delta^3$-negligible event 
\[
\abs*{\QFZ{\zzv}-\QFZ{\zpv}} < \delta\;.
\]
This concludes the proof of Lemma~\ref{l:nd/i:2}.
\end{proof}

\subsection{Proof of Proposition~\ref{p:1:1}}
\label{ss:pf:p:1:1}

Consider $M>3$, $\thetagap\in(0,1)$. 
Our objective is to verify that 
for $\delta\in(0,1)$ sufficiently small
depending on $M$, the conditions defining 
$ \OMG{\thetagap}{\nd}{1:1}{M}{\delta} $
hold on the event 
$ \OMG{\thetagap}{\infty}{1:1}{M}{\delta} $ 
except on a $\delta^3$-negligible event. 
\begin{enumerate}[(i),font=\normalfont\bfseries,topsep=0pt]
\item From condition~\ref{c:i:1:1:1} in the definition of  
the event $\OMG{\thetagap}{\infty}{1:1}{M}{\delta}$ 
we have that 
the event $\OGm{\infty}{\thetagap}$ 
occurs. 
From condition~\ref{c:nd:4} in the definition of $\nd$ we have 
\[
\PROB\fbr[\Big]{\,\OGm{\infty}{\thetagap}\,\symmdiff\,\OGm{\nd}{\thetagap}\,}<\delta^3\;.
\]
Therefore, 
the event $\OGm{\nd}{\thetagap}$ 
occurs
on the event $\OMG{\thetagap}{\infty}{1:1}{M}{\delta}$
except on a $\delta^3$-negligible event. 
Therefore, condition~\ref{c:n:1:1:1} 
in the definition of 
the event $\OMG{\thetagap}{\nd}{1:1}{M}{\delta}$
holds on 
the event $\OMG{\thetagap}{\infty}{1:1}{M}{\delta}$
except on a $\delta^3$-negligible event. 

\item From condition~\ref{c:i:1:1:2} defining $ \OMG{\thetagap}{\infty}{1:1}{M}{\delta} $ we have
\[
\max_{1\leq s<\moment} \absinv{\falphagaf}{s}{0}{\kdel} \leq M\;.
\]
Using Lemma~\ref{l:nd/i:1}, sacrificing a $\delta^3$-negligible event, we get
\[
\max_{1\leq s<\moment}
\absinv{\fand}{s}{0}{\kdel}\leq M+\delta\leq M+1\;,
\]
which is condition~\ref{c:n:1:1:2} defining $\OMG{\thetagap}{\nd}{1:1}{M}{\delta}$. 

\item From condition~\ref{c:i:1:1:3} defining $\OMG{\thetagap}{\infty}{1:1}{M}{\delta}$ we have  
\[
\invabs{\falphagaf}{\moment}{0}{\kdel}\leq M\;.
\]
Using Lemma~\ref{l:nd/i:1}, sacrificing a $\delta^3$-negligible event, we get
\[
\invabs{\fand}{\moment}{0}{\kdel}\leq M+\delta\leq M+1\;,
\]
which is condition~\ref{c:n:1:1:3} defining 
$\OMG{\thetagap}{\nd}{1:1}{M}{\delta}$. 

\item From condition~\ref{c:i:1:1:4} 
    defining $\OMG{\thetagap}{\infty}{1:1}{M}{\delta}$ 
    we have
\[
\adjustlimits \max_{0\leq i\leq m} \max_{\rigidity\leq j\leq m} \QFI{\zzv} \leq M\;.
\]
By Lemma~\ref{l:nd/i:2}, assuming $\delta$ is small enough and 
sacrificing a $\delta^3$-negligible event, we get 
\[
\adjustlimits\max_{0\leq i\leq m} \max_{\rigidity\leq j\leq m} 
\QFI{\zpv}\leq M+\delta \leq M+1\;,
\]
which is condition~\ref{c:n:1:1:4} defining $\OMG{\thetagap}{\nd}{1:1}{M}{\delta}$.

\item From condition~\ref{c:i:1:1:5} 
    defining $\OMG{\thetagap}{\infty}{1:1}{M}{\delta}$
    we have
    \[
    M^{-1}<\QFZ{\zzv}<M\;.
    \]
By Lemma~\ref{l:nd/i:2}, 
    assuming $\delta<M^{-1}-(M+1)^{-1}$
    and sacrificing a $\delta^3$-negligible event, 
    we get
    \[
    (M+1)^{-1}<\QFZ{\zpv}<M+1\;,
    \]
    which is condition~\ref{c:n:1:1:5}
    defining $\OMG{\thetagap}{\nd}{1:1}{M}{\delta}$.
\end{enumerate} 
This concludes the proof of Proposition~\ref{p:1:1}.

\subsection{Proof of Theorem~\ref{t:s:1}}
\label{ss:pf:t:s:1}

Recall from \eqref{eq:omegajdef} and \eqref{eq:omegajdef2} that 
\[
\oj =
\liminf_{k \to \infty}\,
\OMG{\thetagap_j}{\infty}{1:1}{M_j}{\delta_k}
\quad\text{and}\quad 
\onkj = 
\OMG{\thetagap_j}{\ndk}{1:1}{M_j}{\delta_k}\;.
\]
From Proposition~\ref{p:1:1} we get   
\[
\OMG{\thetagap}{\infty}{1:1}{M_j}{\delta_k}
\;\setminus\;\mathcal{E}^{1:1}_{\delta_k} 
\;\subset\;\OMG{\thetagap}{\ndk}{1:1}{M_j}{\delta_k}\;,
\]
and 
\[
\PROB\fbr*{\mathcal{E}^{1:1}_{\delta_k}} < \cneg{1:1} \delta_k^3\;.
\]
Using \eqref{eq:deltakchoice} we get
\[
\sum_k\PROB\fbr*{\mathcal{E}^{1:1}_{\delta_k}}<\infty\;.
\]
Therefore, using the first 
Borel-Cantelli lemma we get 
\[
\oj\;\subset\;\liminf_{k \to \infty}\onkj\quad\mbox{a.s.}
\]
This concludes the proof of Theorem~\ref{t:s:1}.

\section{Proofs of the results in Step 2 of Section~\ref{sec:limitinggaf}}

\subsection{Some auxiliary lemmas}

\begin{lemma} 
Suppose the event $\omnd$ occurs. 
Let $\uz$ be a vector consisting of the roots of $\fand$ inside $\Dset$. 
Let $\uo$ be a vector consisting of the roots of $\fand$ outside $\Dset$. 
Then we have:
\begingroup
\addtolength{\jot}{0.25em}
\begin{align}  
\abs*{\frac{\sym{k}{\uz\cc\uo}}{\sym{\nd-m}{\uo}}}^2
\fbr[\big]{(\nd-k)!}^{\alpha}
={} & \abs*{\xi_{\nd-k}}^2
\frac{\InProdSqNd}{\abs*{\xi_0}^2}\mbox{ for }0\leq k\leq\nd\;;
\label{eq:al:1:1}\\
\totalsumdenom 
={} &
\frac{\abs*{\sym{\nd-m}{\uo}}^2}{(\nd!)^{\alpha}}
\sum_{k=0}^{\nd}
\abs*{\frac{\sym{k}{\uz\cc\uo}}{\sym{\nd-m}{\uo}}}^2
\fbr[\big]{(\nd-k)!}^{\alpha}\;;
\label{eq:al:1:2}\\
\totalsumdenom 
={} &
\frac{1}{\abs*{\xi_{\nd}}^2}
\sum_{l=0}^{\nd}\abs*{\xi_l}^2\;;\label{eq:al:1:3}\\
\sum_{k=0}^{\nd}\abs*{\frac{\sym{k}{\uz\cc\uo}}{\sym{\nd-m}{\uo}}
}^2
\fbr[\big]{(\nd-k)!}^{\alpha} 
={} & \frac{\InProdSqNd}{\abs*{\xi_0}^2}
\sum_{k=0}^{\nd}\abs*{\xi_k}^2\;.
\label{eq:al:1:4}
\end{align}
\endgroup
\label{l:al:1}
\end{lemma}

\begin{proof}
From the relation between the roots and the coefficients of the 
polynomial $\fand$ we have  
\begin{equation}\label{eq:al:1:5}
\frac{\sym{k}{\uz\cc\uo}}{\ndckat}
=(-1)^{k}\frac{\xi_{\nd-k}}{\xi_{\nd}}
\end{equation}
for all $0\leq k\leq \nd$. Therefore for all $0\leq k\leq \nd$
\[
\frac{1}{\fbr*{{\binom{\nd}{k}}k!}^{\alpha}}
\abs*{\frac{\sym{k}{\uz\cc\uo}}{\sym{\nd}{\uz\cc\uo}}}^2
=\frac{1}{(\nd!)^\alpha}\frac{\abs*{\xi_{\nd-k}}^2}{\abs*{\xi_0}^2}\;.
\]
Multiplying both sides by $\InProdSqNd$ and using 
\[
 \frac{\InProdNd}{\sym{\nd}{\uz\cc\uo}}
=\frac{1}{\sym{\nd-m}{\uo}}
\]
we get \eqref{eq:al:1:1}. The equality in 
\eqref{eq:al:1:2} follows simply by expanding the norm:
\[
  \totalsumdenom 
= \sum_{k=0}^{\nd}\abs*{\frac{\sym{k}{\uz\cc\uo}}{\ndckat}}^2
= \frac{\abs*{\sym{\nd-m}{\uo}}^2}{(\nd!)^{\alpha}}
  \sum_{k=0}^{\nd}
     \abs*{\frac{\sym{k}{\uz\cc\uo}}{\sym{\nd-m}{\uo}}}^2
     \fbr[\big]{(\nd-k)!}^{\alpha}\;.
\]
Similarly, expanding the norm and and using \eqref{eq:al:1:5} we get
\[
\totalsumdenom =
\sum_{k=0}^{\nd} 
        \abs*{\frac{\sym{k}{\uz\cc\uo}}{\ndckat}}^2 
    = \frac{1}{\abs*{\xi_{\nd}}^2}
    \sum_{l=0}^{\nd}\abs*{\xi_l}^2\;. 
\]
Equation~\eqref{eq:al:1:4} follows by taking sum over $k$ in \eqref{eq:al:1:1}:
\[
\sum_{k=0}^{\nd}
\abs*{\frac{\sym{k}{\uz\cc\uo}}{\sym{\nd-m}{\uo}}}^2
\fbr*{(\nd-k)!}^{\alpha}
=\frac{\InProdSqNd}{\abs*{\xi_0}^2}
\sum_{k=0}^{\nd}\abs*{\xi_k}^2\;.
\]
This concludes the proof of Lemma~\ref{l:al:1}.
\end{proof}

\begin{lemma}\label{l:tailsum}
On the event $\omnd$ we have
\[
\abs*{\tailsum} = 
\frac{\abs*{\sym{\nd-m}{\uo}}^2}{(\nd!)^{\alpha}}\QFI{\zppv}
\]
for all $0\leq i\leq m$ and $\rigidity\leq j\leq m$, 
where $\uo$ is a vector consisting of the roots of $\fand$ outside $\Dset$.
(See Notation~\ref{n:hadamard} for explanation of the term in the l.h.s.\ 
of the above equation.) 
\end{lemma}

\begin{proof}
Observe that for $0\leq i\leq m$ and $i\leq k\leq \nd-m$
\[
\frac{\sym{k-i}{\uo}}{\sym{\nd-m}{\uo}}=
P_q\fbr*{\inv{\fand}{1}{0}{\infty},\cdots,\inv{\fand}{q}{0}{\infty}}=\zpp{q}\;,
\]
with $q=(\nd-m)-(k-i)$. Therefore, for $0\leq i\leq m$ and 
$\rigidity\leq j\leq m$ we have
\begingroup
\addtolength{\jot}{0.25em}
\begin{align*}
 & \abs*{\tailsum} \\ 
={} & \abs*{\sum_{k=\nd-\Cdel}^{\nd}
    \frac{\overline{\sym{k-i}{\uo}}}{\ndckat}
    \frac{\sym{k-j}{\uo}}{\ndckat}} \\
={} & \frac{\abs*{\sym{\nd-m}{\uo}}^2}{(\nd!)^{\alpha}}
    \abs*{\sum_{k=\nd-\Cdel}^{\nd} 
    \frac{\overline{\sym{k-i}{\uo}}}{\overline{\sym{\nd-m}{\uo}}}
    \frac{\sym{k-j}{\uo}}{\sym{\nd-m}{\uo}}
    \fbr[\big]{(\nd-k)!}^{\alpha}}\\
={} & \frac{\abs*{\sym{\nd-m}{\uo}}^2}
{(\nd!)^{\alpha}}\QFI{\zppv},&&\text{(using \eqref{eq:qij})}\;.    
\end{align*}
\endgroup
This concludes the proof of Lemma~\ref{l:tailsum}.
\end{proof}

\subsection{Proof of Proposition~\ref{p:2:1}}
\label{ss:pf:p:2:1}

Consider $M>3$ and $\thetagap\in(0,1)$. 
Our objective is to verify that 
for $\delta\in(0,1)$ sufficiently small 
depending on $M$, 
the conditions defining 
$\OMG{\thetagap}{\nd}{2:1}{M}{\delta}$
hold on the event 
$\OMG{\thetagap}{\nd}{1:1}{M}{\delta}$
except on a $\delta^3$-negligible event.
\begin{enumerate}[(i),font=\normalfont\bfseries,topsep=0pt]
\item Condition~\ref{c:n:1:1:1} defining 
$ \OMG{\thetagap}{\nd}{1:1}{M}{\delta} $ is 
same as condition~\ref{c:n:2:1:1} defining 
$ \OMG{\thetagap}{\nd}{2:1}{M}{\delta} $.

\item From condition~\ref{c:n:1:1:2} defining $ \OMG{\thetagap}{\nd}{1:1}{M}{\delta} $ we have
\[
\max_{1\leq s<\moment} \absinv{\fand}{s}{0}{\kdel} \le M+1\;.
\]
Using Lemma~\ref{l:kd/i:1}, sacrificing a $\delta^3$-negligible event, we get  
\[
\max_{1\leq s<\moment} \absinv{\fand}{s}{0}{\infty} \le M+1+\delta \le M+2\;,
\]
which is condition~\ref{c:n:2:1:2} defining $\OMG{\thetagap}{\nd}{2:1}{M}{\delta}$.

\item By condition~\ref{c:n:1:1:3} defining $\OMG{\thetagap}{\nd}{1:1}{M+1}{\delta}$ we have
\[
\invabs{\fand}{\moment}{0}{\kdel}\le M+1\;.
\]
Using Lemma~\ref{l:kd/i:1}, sacrificing a $\delta^3$-negligible event, we get  
\[
\invabs{\fand}{\moment}{0}{\infty}\le M+1+\delta \le M+2\;,
\]
which is condition~\ref{c:n:2:1:3} defining $\OMG{\thetagap}{\nd}{2:1}{M}{\delta}$.

\item By condition~\ref{c:n:1:1:4} defining $\OMG{\thetagap}{\nd}{1:1}{M+1}{\delta}$ we have
\[
\max_{0\leq i\leq m}\;
\max_{\rigidity\leq j\leq m}
\QFI{\zpv}\leq M+1\;.
\]
Using Lemma~\ref{l:kd/i:2}, sacrificing a $\delta^3$-negligible event, we get
\[
\max_{0\leq i\leq m}\;
\max_{\rigidity\leq j\leq m}
\QFI{\zppv}\leq M+1+\delta \le M+2\;,
\]
which is condition~\ref{c:n:2:1:4} defining $\OMG{\thetagap}{\nd}{2:1}{M}{\delta}$.

\item By condition~\ref{c:n:1:1:5} defining $\OMG{\thetagap}{\nd}{1:1}{M+1}{\delta}$ we have
\[
(M+1)^{-1}\leq \QFZ{\zpv} \leq M+1\;.
\]
Using Lemma~\ref{l:kd/i:2}, taking $\delta < (M+1)^{-1}-(M+2)^{-1}$ and sacrificing a $\delta^3$-negligible event, we get 
\[
(M+2)^{-1}\leq\QFZ{\zppv}\leq M+2\;,
\]
which is condition~\ref{c:n:2:1:5} defining $\OMG{\thetagap}{\nd}{2:1}{M}{\delta}$.
\end{enumerate}

This concludes the proof of Proposition~\ref{p:2:1}.

\subsection{Proof of Proposition~\ref{p:2:2}}
\label{ss:pf:p:2:2}

Consider $M>3$, $\thetagap\in(0,1)$, and $\delta\in(0,1)$.
Suppose the event $\omnd$ occurs.
Let $\uz$ be a vector consisting of the roots 
of $\fand$ inside $\Dset$. 
Let $\uo$ be a vector consisting of the roots
of $\fand$ outside $\Dset$. 
From the relationship between the roots and 
the coefficients of the polynomial $\fand$ we 
get
\[
\frac{\sym{k}{\uz\cc\uo}}{\ndckat}
=(-1)^{k}\frac{\xi_{\nd-k}}{\xi_{\nd}}
\]
for all $0\leq k\leq \nd$. 
By Proposition~\ref{p:rec} we have
for $0\leq k\leq\nd-m$
\[
\frac{\sym{k}{\uo}}{\ndckat}
= \sum_{r=0}^{k} 
(-1)^{k-r} \cdot \g_r \cdot \frac{\xi_{l+r}}{\ffa{l}{r}}\;.
\]
Therefore, for $0\leq k\leq \nd-m$ we have  
\[
\frac{\sym{k}{\uo}}{\sym{k}{\uz\cc\uo}} 
= 1 + \sum_{r=1}^{k} (-1)^{k-r}\cdot \g_r \cdot  
\frac{\xi_{\nd-k+r}}{\xi_{\nd-k}}\cdot
\frac{1}{\ffa{\nd-k}{r}}\;.
\]
Thus, our objective is to show that 
except on a $\delta^3$-negligible event we have
\[
\abs*{\sum_{r=1}^{\nd-l}(-1)^{\nd-l-r}\cdot\g_r\cdot
\frac{\xi_{l+r}}{\xi_l}\cdot\frac{1}{\ffa{l}{r}}}
\leq\frac{1}{2}
\]
for all $l$ in the range $\Ld\leq l\leq \Ld+\hLd$. 
By the inequality $1+x\geq\sqrt{x}$ we have 
\[
\frac{1}{\ffa{l}{r}}
\leq\frac{1}{l^{r\alpha/4}(r!)^{\alpha/4}}\;. 
\]
By Proposition~\ref{p:rec} we have 
$|\g_r|\le \fbr*{\usekdm{471}}^r$.
Therefore, for $\Ld\leq l\leq \Ld+\hLd$
\[
\abs*{\sum_{r=1}^{\nd-l} 
(-1)^{\nd-l-r}\cdot\g_r\cdot
\frac{\xi_{l+r}}{\xi_l}
\cdot\frac{1}{\ffa{l}{r}}} 
\leq 
\sum_{r=1}^{\nd-l}
\frac{\fbr*{\usekdm{471}}^r}{l^{r\alpha/4}(r!)^{\alpha/4}}
\abs*{\frac{\xi_{l+r}}{\xi_l}}\;.
\]
Therefore, 
assuming $\delta$ is small enough so that 
$\Ld^{\alpha/8}>\usekdm{471}$,
it is enough to show that  
\begin{equation}\label{eq:pf:p:2:2:1}
\PROB\fbr*{\;\sum_{r=1}^{\nd-l}
\frac{1}{l^{r\alpha/8}(r!)^{\alpha/4}}
\abs*{\frac{\xi_{l+r}}{\xi_l}}\leq\frac{1}{2}
\mbox{ for all }\Ld \leq l \leq \Ld + \hLd\;} 
\leq\delta^3\;.
\end{equation}
Using the fact that if $\xi$ is 
complex Gaussian then $\abs*{\xi}^2$ is an exponential random 
variable with mean $1$, we can deduce that 
\[
\PROB\fbr*{\abs*{\frac{\xi_{l+r}}{\xi_l}}>x} 
\le\frac{1}{x^2}\;,
\]
for $1\leq r\leq \nd-l$ and $\Ld\leq l\leq \Ld+\hLd$. 
Therefore  
\begin{equation}\label{eq:pf:p:2:2:2}
\PROB\fbr*{\abs*{\frac{\xi_{l+r}}{\xi_l}}
>l^{r\alpha/16}(r!)^{\alpha/8}}
\le\frac{1}{l^{r\alpha/8}(r!)^{\alpha/4}}\;.
\end{equation}
For $\Ld\leq l\leq\Ld+\hLd$ 
let $\mathcal{E}(l)$ be the 
event that for all $r\geq 1$
\[
\abs*{\frac{\xi_{l+r}}{\xi_l}}
\le l^{r\alpha/16}(r!)^{\alpha/8}\;.
\]
On this event, 
assuming $\delta$-is small enough so that 
$\Ld^{-\alpha/16}<1/3$, 
we have
\[
\sum_{r=1}^{\nd-l}
\frac{1}{l^{r\alpha/8}(r!)^{\alpha/4}}
\abs*{\frac{\xi_{l+r}}{\xi_l}}
\leq\sum_{r=1}^{\infty}
\frac{1}{l^{r\alpha/16}(r!)^{\alpha/8}}
\leq\sum_{r=1}^{\infty}\frac{1}{l^{r\alpha/16}}
\leq \frac{1}{2}\;.
\]
For each $l$ in the range 
$\Ld\leq l\leq\Ld+\hLd$ using \eqref{eq:pf:p:2:2:2} we get
\[
\PROB\fbr*{\mathcal{E}(l)^c} 
= \PROB\fbr*{\abs*{\frac{\xi_{l+r}}{\xi_l}}
>l^{r\alpha/16}(r!)^{\alpha/8}\quad\mbox{for some }r}
\leq
\sum_{r=1}^\infty
\frac{1}{l^{r\alpha/8}(r!)^{\alpha/4}}
\leq
\frac{C_0}{l^{\alpha/8}}\;, 
\]
where $C_0$ is defined in \eqref{d:czero}. 
By a union over $l$ and using condition~\ref{c:Ld:1} defining $\Ld$ we get   
\begingroup 
\addtolength{\jot}{0.25em}
\begin{align*}
& \PROB\fbr*{\;
\sum_{r=1}^{\nd-l}
\frac{1}{l^{r\alpha/8}(r!)^{\alpha/4}}
\abs*{\frac{\xi_{l+r}}{\xi_l}}
>\frac{1}{2}\mbox{ for some } \Ld\leq l\leq \Ld+\hLd\;}\\ 
\leq{} & \sum_{l=\Ld}^{\Ld+\hLd} \PROB\fbr*{\mathcal{E}(l)^c}
\;\leq\; \sum_{l=\Ld}^{\Ld+\hLd}
\frac{C_0}{l^{\alpha/8}}
\;\leq\; 
\delta^3\;.
\end{align*}
\endgroup
Thus we get \eqref{eq:pf:p:2:2:1}. This concludes the proof of Proposition~\ref{p:2:2}.

\subsection{Proof of Proposition~\ref{p:2:3}}
\label{ss:pf:p:2:3}

Consider $M>3$, $\thetagap\in(0,1)$, and $\delta\in(0,1)$.
Suppose the event $\omnd$ occurs.
Let $\uz$ be a vector consisting of the roots of $\fand$ inside $\Dset$.
Let $\uo$ be a vector consisting of the roots of $\fand$ outside $\Dset$.
By Proposition~\ref{p:2:2}, we have, except on a $\delta^3$ negligible event, 
\[
\frac{2}{3} \abs*{\sym{k}{\uo}}
\le \abs*{\sym{k}{\uz\cc\uo}} 
\le 2 \abs*{\sym{k}{\uo}}\;,
\]
for $\nd-\Ld-\hLd\le k\le \nd-\Ld$. 
Changing the variable from $k$ to $l=\nd-k$ we get
\begingroup
\addtolength{\jot}{0.25em}
\begin{align*}
\frac{4}{9}\cdot\frac{1}{\hLd}
\sum_{l=\Ld}^{\Ld+\hLd}
\abs*{\frac{\sym{\nd-l}{\uo}}{\sym{\nd-m}{\uo}}}^2 
\fbr*{ l!}^{\alpha} 
\le{} &  
\frac{1}{\hLd}
\sum_{l=\Ld}^{\Ld+\hLd}
\abs*{\frac{\sym{\nd-l}{\uz\cc\uo}}{\sym{\nd-m}{\uo}}}^2
\fbr*{ l!}^{\alpha}
\\
\le{} & 
4\cdot\frac{1}{\hLd}
\sum_{l=\Ld}^{\Ld+\hLd}
\abs*{\frac{\sym{\nd-l}{\uo}}{\sym{\nd-m}{\uo}}}^2 
\fbr*{ l!}^{\alpha}\;.
\end{align*}
\endgroup
We can rewrite the term in the middle of the inequality above,
using \eqref{eq:al:1:1} from Lemma~\ref{l:al:1}, as
\[
\frac{1}{\hLd}
\sum_{l=\Ld}^{\Ld+\hLd}
\abs*{\frac{\sym{\nd-l}{\uz\cc\uo}}{\sym{\nd-m}{\uo}}}^2
\fbr*{ l!}^{\alpha}
= 
\frac{\InProdSqNd}{\abs*{\xi_0}^2}
\frac{1}{\hLd}
\sum_{l=\Ld}^{\Ld+\hLd}\abs*{\xi_l}^2\;.
\]
By condition \ref{c:Ld:2} defining $\Ld$, 
we have, except on a $\delta^3$-negligible event,
\[
 \frac{1}{2} 
<\frac{1}{\hLd}\sum_{l=\Ld}^{\Ld+\hLd}\abs*{\xi_l}^2 
<\frac{3}{2}\;.
\]
Hence
\begingroup
\addtolength{\jot}{0.25em}
\begin{align*}
\frac{8}{27}\cdot\frac{1}{\hLd}
\sum_{l=\Ld}^{\Ld+\hLd}
\abs*{\frac{\sym{\nd-l}{\uo}}{\sym{\nd-m}{\uo}}}^2
\fbr*{ l!}^{\alpha}
\leq{} & \frac{\InProdSqNd}{\abs*{\xi_0}^2}\\
\leq{} & 8\cdot\frac{1}{\hLd}
\sum_{l=\Ld}^{\Ld+\hLd}
\abs*{\frac{\sym{\nd-l}{\uo}}{\sym{\nd-m}{\uo}}}^2 
\fbr*{ l!}^{\alpha}\;. 
\end{align*}
\endgroup
From here, using \eqref{eq:tailavg}
we get \eqref{c:p:2:3:1}.   

Now suppose the event 
$\OMG{\thetagap}{\nd}{2:2}{M}{\delta}$
occurs. 
We want to establish \eqref{c:p:2:3:2}.
From Proposition~\ref{p:2:2} we have 
$\OMG{\thetagap}{\nd}{2:2}{M}{\delta}$
is a subset of 
$\OMG{\thetagap}{\nd}{2:1}{M}{\delta}$,
and from condition~\ref{c:n:2:1:5} defining
$\OMG{\thetagap}{\nd}{2:1}{M}{\delta}$ we have
\[
(M+2)^{-1} \leq \QFZ{\zppv} \leq M+2\;.
\]
So, \eqref{c:p:2:3:1} implies 
\[
\frac{\InProdSqNd}{\abs*{\xi_0}^2}\ge
\frac{8}{27}\cdot\frac{1}{M+2}\;.
\]
By condition~\ref{c:nd:2} defining $\nd$ we have 
\[
\sum_{k=0}^{\nd}\abs*{\xi_k}^2>\frac{1}{2}\nd
\]
except on a $\delta^3$-negligible event. 
Therefore, using \eqref{eq:al:1:4} from Lemma~\ref{l:al:1}, we get 
\[
\frac{1}{\nd}
\sum_{k=0}^{\nd}
\abs*{\frac{\sym{k}{\uz\cc\uo}}{\sym{\nd-m}{\uo}}}^2
\fbr[\big]{(\nd-k)!}^{\alpha}
\geq\frac{4}{27}\cdot\frac{1}{M+2}\;.
\]
Hence we get \eqref{c:p:2:3:2}. 
This concludes the proof of Proposition~\ref{p:2:3}.

\subsection{Proof of Proposition~\ref{p:2:4}}
\label{ss:pf:p:2:4}

Consider $M>3$, $\thetagap\in(0,1)$, and $\delta\in(0,1)$.
Suppose that the event $\omnd$ occurs.
Let $\uz$ be a vector consisting of the roots $\fand$ inside $\Dset$.
Let $\uo$ be a vector consisting of the roots $\fand$ outside $\Dset$.
Consider $i$ and $j$ satisfying $0\leq i\leq m$ and $\rigidity\leq j\leq m$.
From Proposition~\ref{p:rec} it follows that for all $0\leq l\leq\nd-m$
\[
\frac{\sym{\nd-l}{\uo}}
{\fbr*{\binom{\nd}{\nd-l}(\nd-l)!}^{\alpha/2}}
=\frac{(-1)^{\nd-l}}{\xi_{\nd}}
\tbr*{
\sum_{r=0}^{(\nd-l)\wedge(\lev-1)}
(-1)^{r}
\frac{\g_r\cdot\xi_{l+r}}{\ffa{l}{r}}
+\tail{l}{\nd}}\;.
\]
Therefore, for $j\leq k\leq\nd-m+j$ and $l=\nd-k$ (so that $k-j=\nd-(l+j)$) we have
\begingroup
\addtolength{\jot}{0.25em}
\begin{align*}
\frac{\sym{k-j}{\uo}}
{\fbr*{\binom{\nd}{k}k!}^{\alpha/2}}
={} & \frac{1}{\ffa{\nd-k}{j}} \frac{\sym{k-j}{\uo}}
{\fbr*{\binom{\nd}{k-j}(k-j)!}^{\alpha/2}}\\
={} & \frac{1}{\ffa{l}{j}}
\frac{(-1)^{k-j}}{\xi_{\nd}}
\tbr*{
\sum_{r=0}^{(\nd-l-j)\wedge(\lev-1)}
(-1)^{r}
\frac{\g_r\cdot\xi_{l+j+r}}{\ffa{l+j}{r}}
+\tail{l+j}{\nd}}\\
={} & 
\frac{(-1)^{k-j}}{\xi_{\nd}}
\tbr*{
\sum_{r=0}^{(\nd-l-j)\wedge(\lev-1)}
(-1)^{r}
\frac{\g_r\cdot\xi_{l+j+r}}{\ffa{l}{j+r}}
+\frac{\tail{l+j}{\nd}}{\ffa{l}{j}}}\;.
\end{align*}
\endgroup
Using this and a similar expansion for $\overline{\sym{k-i}{\uo}}$ we get
\begingroup
\addtolength{\jot}{0.25em}
\begin{align*}
  & \frontsum  \\
={} & \sum_{k=0}^{\nd-\Cdel-1}
\frac{\overline{\sym{k-i}{\uo}}}{\ndckat}
\frac{\sym{k-j}{\uo}}{\ndckat} \\
={} & \sum_{k=i\vee j}^{\nd-\Cdel-1}
\frac{\overline{\sym{k-i}{\uo}}}{\ndckat}
\frac{\sym{k-j}{\uo}}{\ndckat} \\
={} & \frac{(-1)^{i+j}}{\abs*{\xi_{\nd}}^2}
\sum_{l=\Cdel+1}^{\nd-(i\vee j)}
\tbr*{
\sum_{r_1=0}^{(\nd-l-i)\wedge(\lev-1)}
(-1)^{r_1}
\frac{\overline{\g_{r_1}}\cdot\overline{\xi_{l+i+r_1}}}{\ffa{l}{i+r_1}}
+\frac{\overline{\tail{l+i}{\nd}}}{\ffa{l}{i}}}\\
& \qquad \cdot
\tbr*{
\sum_{r_2=0}^{(\nd-l-j)\wedge(\lev-1)}
(-1)^{r_2}
\frac{\g_{r_2}\cdot\xi_{l+j+r_2}}{\ffa{l}{j+r_2}}
+\frac{\tail{l+j}{\nd}}{\ffa{l}{j}}}\\
={} & \frac{(-1)^{i+j}}{\abs*{\xi_{\nd}}^2}
\fbr[\Big]{\bT{1}+\bT{2}+\bT{3}+\bT{4}}\;,    
\numberthis\label{eq:p:2:4:1}
\end{align*}
\endgroup
where
\begingroup
\addtolength{\jot}{0.25em}
\begin{align*}
\bT{1} \coloneqq{} &
\sum_{r_1=0}^{\lev-1}
\sum_{r_2=0}^{\lev-1}
(-1)^{r_1+r_2}\cdot
\overline{\g_{r_1}}\cdot
\g_{r_2}\cdot
\fbr*{\sum_{l=\Cdel+1}^{\nd-(i+r_1)\vee(j+r_2)}
\frac{\overline{\xi_{l+i+r_1}}\cdot
\xi_{l+j+r_2}}
{\ffa{l}{i+r_1}\cdot
\ffa{l}{j+r_2}}}\;,\\
\bT{2} \coloneqq{} &
\sum_{r_1=0}^{\lev-1}
(-1)^{r_1}\cdot
\overline{\g_{r_1}}\cdot
\fbr*{\sum_{l=\Cdel+1}^{\nd-(i+r_1)\vee(j+r_2)}
\frac{\overline{\xi_{l+i+r_1}}\cdot
\tail{l+j}{\nd}}
{\ffa{l}{i+r_1}\cdot
\ffa{l}{j}}}\;,\\
\bT{3} \coloneqq{} &
\sum_{r_2=0}^{\lev-1}
(-1)^{r_2}\cdot
\g_{r_2}\cdot
\fbr*{\sum_{l=\Cdel+1}^{\nd-(i+r_1)\vee(j+r_2)}
\frac{\overline{\tail{l+i}{\nd}}
\cdot
\xi_{l+j+r_2}}
{\ffa{l}{i}\cdot
\ffa{l}{j+r_2}}}\;,\\
\bT{4} \coloneqq{} &
\sum_{l=\Cdel+1}^{\nd-(i+r_1)\vee(j+r_2)}
\frac{\overline{\tail{l+i}{\nd}}\cdot\tail{l+j}{\nd}}{\ffa{l}{i}\cdot\ffa{l}{j}}\;.
\end{align*}
Also from Proposition~\ref{p:rec} we have 
\begin{equation}\label{eq:p:2:4:4}
\abs*{\g_r}\leq\fbr[\big]{\usekdm{471}}^{r}
\end{equation} 
for $0\leq r \leq\nd$, and for $ l \geq \usekdm{472} $ 
\begin{equation}\label{eq:p:2:4:5}
\fbr*{\Exp\fbr*{
\abs*{\tail{l}{\nd}}^2
\indicator{\omnd}}}^{1/2}
\leq
\frac{\fbr[\big]{\usekdm{473}}^{\lev}}{l^{\lev\alpha/2}}\;.
\end{equation}
Fix $0\leq r_1\leq \rigidity$ and $0\leq r_2\leq \rigidity$.
Assume $\delta$ is small enough so that $\Cdel\geq\usekdm{472}$.
By condition~\ref{c:Cd:2} defining $\Cdel$,
and the Chebyshev inequality, we get 
\begin{equation}\label{eq:p:2:4:6}
\abs*{\sum_{l=\Cdel+1}^{\nd-(i+r_1)\vee(j+r_2)}
\frac{\overline{\xi_{l+i+r_1}}\cdot\xi_{l+j+r_2}}
{\ffa{l}{i+r_1}\cdot\ffa{l}{j+r_2}}}<
\frac{\delta}{4\fbr[\big]{\usekdm{471}}^{2\rigidity}\rigidity^2}
\end{equation}
in the complement of a $\delta^3$-negligible event.
By condition~\ref{c:Cd:2} defining $\Cdel$, 
equation~\eqref{eq:p:2:4:5},
and using a Markov bound, 
we get 
\begingroup
\addtolength{\jot}{0.25em}
\begin{align*}
\abs*{\sum_{l=\Cdel+1}^{\nd-(i+r_1)\vee(j+r_2)}
\frac{\overline{\xi_{l+i+r_1}}\cdot\tail{l+j}{\nd}}
{\ffa{l}{i+r_1}\cdot\ffa{l}{j}}}
&<\frac{\delta}{4\fbr[\big]{\usekdm{471}}^{\lev}\rigidity}\;,
\numberthis\label{eq:p:2:4:7}\\
\abs*{\sum_{l=\Cdel+1}^{\nd-(i+r_1)\vee(j+r_2)}
\frac{\overline{\tail{l+i}{\nd}}\cdot\xi_{l+j+r_2}}
{\ffa{l}{i}\cdot\ffa{l}{j+r_2}}}
&<\frac{\delta}{4\fbr[\big]{\usekdm{471}}^{\lev}\rigidity}\;,
\numberthis\label{eq:p:2:4:8}\\
\abs*{\sum_{l=\Cdel+1}^{\nd-(i+r_1)\vee(j+r_2)}
\frac{\overline{\tail{l+i}{\nd}}\cdot\tail{l+j}{\nd}}
{\ffa{l}{i}\cdot\ffa{l}{j}}}
&<\frac{\delta}{4}\;,
\numberthis\label{eq:p:2:4:9}
\end{align*}
\endgroup
in the complement of a $\delta^3$-negligible event.
{%
Using \eqref{eq:p:2:4:4}, \eqref{eq:p:2:4:6}-\eqref{eq:p:2:4:9}, 
and a union bound over choices of $r_1$ and $r_2$, we get
\[
\big|\bT{1}+\bT{2}+\bT{3}+\bT{4}\big|<\delta
\]
in the complement of a $\delta^3$-negligible event.}
{%
Therefore, from \eqref{eq:p:2:4:1} we get 
\begin{equation}\label{eq:p:2:4:10}
\Big|\frontsum\Big|\leq\frac{\delta}{\abs*{\xi_{\nd}}^2}\;.
\end{equation}
}
{%
Using condition~\ref{c:nd:2} in the definition of $\nd$ 
and \eqref{eq:al:1:3} we get
\begin{equation}\label{eq:p:2:4:3}
\totalsumdenom>\frac{\nd}{2}\cdot\frac{1}{\abs*{\xi_{\nd}}^2}
\end{equation}
except on a $\delta^3$-negligible event.} 
{%
Combining \eqref{eq:p:2:4:10} and \eqref{eq:p:2:4:3} we get
\[
\frac{\Big|\frontsum\Big|}{\totalsumdenom}\leq\frac{2\delta}{\nd}\;. 
\]
}
This concludes the proof of Proposition~\ref{p:2:4}.

\subsection{Proof of Proposition~\ref{p:2:5}}\label{ss:pf:p:2:5}

Consider $M>3$, $\thetagap\in(0,1)$, and $\delta\in(0,1)$.
Suppose that the event $\OMG{\thetagap}{\nd}{2:4}{M}{\delta}$ occurs.
Let $\uz$ be a vector consisting of the roots of $\fand$ inside $\Dset$.
Let $\uo$ be a vector consisting of the roots of $\fand$ outside $\Dset$.
Let $\us=\constraintinternal(\uz)$. 
Let $\uzp\in\Sigma_{\usm}$.
Recall from \eqref{eq:413}
\begin{equation}\label{eq:413recall}
\frac{\drho{\uo\,,\,\us}{\nd}{\uzp}}
{\drho{\uo\,,\,\us}{\nd}{\uz}} 
=\abs*{\frac{\van{\uzp\,,\,\uo}}{\van{\uz\,,\,\uo}}}^2
\fbr*{\frac{\fD{\uzp}{\uo}}{\fD{\uz}{\uo}}}^{-(\nd+1)}.
\end{equation}
Therefore, 
to bound the ratio of conditional densities, 
it is sufficient to bound 
the ratio of the Vandermonde terms 
and the ratio of the symmetric functions. 
To bound the ratio of the Vandermonde terms 
we use Proposition~\ref{p:van}. 
To bound the ratio of the symmetric functions 
we use Proposition~\ref{p:sym}.
First we proceed to bound the ratio of the Vandermonde terms.

Since the event $\OMG{\thetagap}{\nd}{2:4}{M}{\delta}$ 
is a subset of the event $\OMG{\thetagap}{\nd}{2:1}{M}{\delta}$, 
from condition~\ref{c:n:2:1:1} defining $\OMG{\thetagap}{\nd}{2:1}{M}{\delta}$ 
we have that the event $\OGm{\nd}{\thetagap}$ occurs. 
And, from conditions~\ref{c:n:2:1:2} and \ref{c:n:2:1:3} 
defining $\OMG{\thetagap}{\nd}{2:1}{M}{\delta}$
we have
\[
\max_{1\leq s<\moment} \absinv{\fand}{s}{0}{\infty} \le M,
\quad\mbox{and}\quad 
\invabs{\fand}{\moment}{0}{\infty} \le M\;.
\]
Therefore, by Proposition~\ref{p:van} we get: 
\begin{equation}\label{eq:vanbound}
\exp\fbr[\Big]{-6 m \usekd{p45}\moment\thetagap^{-1}M}
\abs*{\frac{\van{\uzp}}{\van{\uz}}}^2
\leq
\abs*{\frac{\van{\uzp\,,\,\uo}}{\van{\uz\,,\,\uo}}}^2
\leq 
\abs*{\frac{\van{\uzp}}{\van{\uz}}}^2
\exp\fbr[\Big]{6 m \usekd{p45}\thetagap^{-1}\moment M}.
\end{equation}
 
Now we proceed to bound the ratio $\fD{\uzp}{\uo}/\fD{\uz}{\uo}$.
We will use Proposition~\ref{p:sym}. 
We need to bound $\fDij{\uz}{\uo}$ for $0\leq i\leq m$ and $\rigidity\leq j\leq m$.
To bound $\fDij{\uz}{\uo}$ we divide it as
\begin{align*}
\fDij{\uz}{\uo} ={} &
\frac{\abs*{\totalsumnum}}{\totalsumdenom} \\
\leq{} & \frac{\Big|\frontsum\Big|}{\totalsumdenom}\\
& + \frac{\abs*{\tailsum}}{\totalsumdenom}\;.\numberthis\label{eq:split}
\end{align*}
By Proposition~\ref{p:2:4} we have
\begin{equation}\label{eq:split1}
\frac{\Big|\frontsum\Big|}{\totalsumdenom}\leq\frac{2\delta}{\nd}\;.
\end{equation}
Using Lemma~\ref{l:tailsum} and condition~\ref{c:n:2:1:4} defining 
$\OMG{\thetagap}{\nd}{2:1}{M}{\delta}$, which holds since  
$\OMG{\thetagap}{\nd}{2:4}{M}{\delta}$ is a subset of $\OMG{\thetagap}{\nd}{2:1}{M}{\delta}$,
we get
\begin{equation}\label{eq:split2-1}
\abs*{\tailsum} 
\leq \frac{\abs*{\sym{\nd-m}{\uo}}^2}{\fbr*{\nd!}^{\alpha}} (M+2)\;.
\end{equation}
Using \eqref{eq:al:1:2} from Lemma~\ref{l:al:1} and \eqref{c:p:2:3:2} from Proposition~\ref{p:2:3} we get
\begin{equation}\label{eq:split2-2}
\totalsumdenom 
\geq 
\frac{\abs*{\sym{\nd-m}{\uo}}^2}{(\nd!)^{\alpha}}
\frac{4}{27}\frac{\nd}{M+2}\;.
\end{equation}
Combining \eqref{eq:split2-1} and \eqref{eq:split2-2} we get
\begin{equation}\label{eq:split2}
\frac{\abs*{\tailsum}}{\totalsumdenom}
\leq\frac{1}{\nd}\frac{27}{4} (M+2)^2\;.
\end{equation}
Combining \eqref{eq:split}, \eqref{eq:split1}, \eqref{eq:split2}
we get
\[
\fDij{\uz}{\uo}\leq
\usekd{pf:p:2:5:1}\frac{M^2}{\nd}\;,
\]
for some constant $\usekd{pf:p:2:5:1}>0$
Therefore, using Proposition~\ref{p:sym} we get 
\[
1-\usekdm{pf:p:2:5:2}\frac{M^2}{\nd}
\leq\frac{\fD{\uzp}{\uo}}{\fD{\uz}{\uo}}
\leq 1+\usekdm{pf:p:2:5:2}\frac{M^2}{\nd}\;, 
\]
for some constant $\usekdm{pf:p:2:5:2}>0$.
Therefore
\begin{equation}\label{eq:symbound}
\exp \fbr[\Big]{ - \usekdm{pf:p:2:5:3} M^2 } 
\leq \fbr*{\frac{\fD{\uzp}{\uo}}{\fD{\uz}{\uo}}}^{\nd+1}
\leq \exp\fbr[\Big]{ \usekdm{pf:p:2:5:3} M^2}\;, 
\end{equation}
for some constant $\usekdm{pf:p:2:5:3}>0$.
Combining \eqref{eq:413recall},
\eqref{eq:vanbound} and \eqref{eq:symbound}
we get
\begin{equation}\label{eq:Bound-target-repeat}  
\exp\fbr[\Big]{-f(M,\thetagap)} 
\abs*{\frac{\van{\uzp}}{\van{\uz}}}^2
\le\frac{\drho{\uo}{\nd}{\uzp}}{\drho{\uo}{\nd}{\uz}} 
\le\exp\fbr[\Big]{ f(M,\thetagap)}
\abs*{\frac{\van{\uzp}}{\van{\uz}}}^2\;. 
\end{equation}
where $f(M,\theta)=\usekdm{pf:p:2:5:4}\fbr*{ M^2+ M \thetagap^{-1}}$, 
for some constant $\usekdm{pf:p:2:5:4}>0$. 
Thus we get \eqref{eq:Bound-target}. 
This concludes the proof of Proposition~\ref{p:2:5}.

\begin{remark}
In the definition of the event  
$\OMG{\thetagap}{\infty}{1:1}{M}{\delta}$
(Definition~\ref{d:i:1:1})
we have used the same bound $M$ in all the compactness conditions.
Suppose we use: 
$a_M$ in conditions \ref{c:i:1:1:2} and \ref{c:i:1:1:3} in place of $M$;
$b_M$ in condition \ref{c:i:1:1:4} in place of $M$;
$c_M$ in condition \ref{c:i:1:1:5} in place of $M$;
where $a_M$, $b_M$, and $c_M$ are some functions of 
$M$ diverging to $\infty$ as $M\to\infty$. 
Suppose we also modify the definitions of all the other events 
in Step 2 accordingly. 
Then we will get \eqref{eq:Bound-target-repeat} 
with $f(M,\thetagap)=\usekdm{pf:p:2:5:5}\fbr*{ b_M C_M + a_m\thetagap^{-1}}$, for some constant $\usekdm{pf:p:2:5:5}>0$.
\end{remark}

\subsection{Proof of Theorem~\ref{t:s:2}}
\label{ss:pf:t:s:2}

Our objective is to show that there exists $\seq*{\correctionterm{j}{k}}_{j\geq 1,k\geq 1}$ such that for
$A\in\FUBasisInsidempts$, $B\in\FUBasisOutside$, $j\geq 1$, $k\geq 1$
\begin{align*}
& \Prob[\Big]{\,
\fbr[\big]{\,\zinnk\in A\,}
\,\cap\,
\fbr[\big]{\,\zoutnk\in B\,}
\,\cap\,
\onkj\,}\\
\overset{j}{\scaleto{\asymp}{8pt}}\;& 
\int_{\zoutnk^{-1}(B)\;\cap\;\onkj}
\nu_{\Phi,\Psi,\Dset}\kernel{A}{\zout}\deri\PROB + \vartheta(j,k)\;,
\numberthis\label{eq:abscondrepeat}
\end{align*}                      
and $\lim_{k\to\infty}\correctionterm{j}{k}=0$ for each $j\geq 1$.

\begin{notation}\text{}
\begin{enumerate}[(i),font=\normalfont\bfseries,topsep=0pt]
\item Let $\pss$ be the range of the function $\constraintinternal:\Dset^m\to\CC^{\rigidity-1}$ (defined in Notation~\ref{n:con:int}).
Thus, $\pss$ is a bounded open set. 

\item For $\us\in\pss$ and $A\in\BorelInsidempts$ let
\[
\h{A}{\us}\coloneqq\frac{\displaystyle\int_{\asvector{A}}
\abs*{\van{\uz}}^2 \deri\leb(\uz)}
{\displaystyle\int_{\Sigma_{\usm}}\abs*{\van{\uz}}^2\deri\leb(\uz)}\;.
\]

\item Let 
\[
\usnd\coloneqq\constraintinternalpc\fbr*{\zinnd}\;,  
\]
i.e., for $1\leq k\leq\rigidity-1$, the $k$-th coordinate of $\usnd$
is the $k$-th power sum of the roots of $\fand$ inside $\Dset$.

\end{enumerate}
\end{notation}
Recall from \eqref{eq:candidate}
\[
\nu_{\Phi,\Psi,\Dset}\kernel{A}{\zout}=\h{A}{\constraintexternal\fbr*{\zout}}\;.
\]
Since $A\in\FUBasisInsidempts$, we get that $\h{A}{\us}$ is continuous in $\us$. 
Observe that 
\[
\lim_{k\to\infty}
\usnk
=
\constraintinternalpc\fbr*{\zin}=
\constraintexternal\fbr*{\zout}\quad\mbox{a.s.}
\]
Therefore
\begin{equation}\label{n1}
\lim_{k\to\infty} 
\h{A}{\usnk}
=
\h{A}{\constraintexternal\fbr*{\zout}}
=
\nu_{\Phi,\Psi,\Dset}\kernel{A}{\zout}\quad\mbox{a.s.}
\end{equation}
Let us now record an observation as a proposition and finish the proof of Theorem~\ref{t:s:2} using this proposition.
After that we prove this proposition.

\begin{proposition}
For $A\in\FUBasisInsidempts$, 
$B\in\FUBasisOutside$, 
$M>3$, $\theta\in(0,1)$, $\delta\in(0,1)$
\begingroup
\addtolength{\jot}{0.25em}
\begin{align*} 
& \PROB\fbr[\Big]{\,
    \fbr[\big]{\,\zinnd\in A\,}
        \,\cap\,
    \fbr[\big]{\,\zoutnd\in B\,}
        \,\cap\,
    \OMG{\thetagap}{\nd}{1:1}{M}{\delta}
    \,}\\
\overset{M,\thetagap}{\scaleto{\asymp}{8pt}}\; & 
\int_{(\zoutnd)^{-1}(B)\;\cap\;\OMG{\thetagap}{\nd}{1:1}{M}{\delta}}
\h{A}{\usnd} 
\deri\PROB+O(\delta)\;.
\numberthis\label{f1} 
\end{align*}
\endgroup
\label{p:9.1}
\end{proposition}
We set $M=M_j$, $\thetagap=\thetagap_j$, and $\delta=\delta_k$ 
in \eqref{f1} and get
\begingroup
\addtolength{\jot}{0.25em}
\begin{align*}
& \PROB\fbr[\Big]{\;
\fbr[\big]{\,\zinnk\in A\,}
\,\cap\,
\fbr[\big]{\,\zinnk\in B\,}
\,\cap\,
\onkj\;} \\
\overset{j}{\scaleto{\asymp}{8pt}} \;
& \int_{(\zoutnk)^{-1}(B)\;\cap\;\onkj}
\h{A}{\usnk} 
\deri\PROB + \ok\;,\numberthis\label{f2}
\end{align*}
\endgroup 
where $\ok$ is a quantity which goes to $0$ as $k\to\infty$.
Using \eqref{n1} we get
\begingroup
\addtolength{\jot}{0.25em}
\begin{align*}
& \int_{(\zoutnk)^{-1}(B)\;\cap\;\onkj}
\h{A}{\usnk} 
\deri\PROB\\ 
={} & \int_{(\zoutnk)^{-1}(B)\;\cap\;\onkj}
\h{A}{\constraintexternal\fbr*{\zout}} 
\deri\PROB+\ok\\
={} & \int_{(\zoutnk)^{-1}(B)\;\cap\;\onkj}
\nu_{\Phi,\Psi,\Dset}
\kernel{A}{\zout}\deri\PROB + \ok\;,
\numberthis\label{f3}
\end{align*}
\endgroup
Combining \eqref{f2} and \eqref{f3} we get
\begingroup
\addtolength{\jot}{0.25em}
\begin{align*}
& 
\PROB\fbr[\Big]{\;
\fbr[\big]{\,\zinnk\in A\,}
\,\cap\, 
\fbr[\big]{\,\zoutnk\in B\,}
\,\cap\,
\onkj\;}\\ 
\overset{j}{\scaleto{\asymp}{8pt}} \; & 
\int_{(\zoutnk)^{-1}(B)\;\cap\;\onkj}
\nu_{\Phi,\Psi,\Dset}\kernel{A}{\zout}\deri\PROB+\ok\;.\numberthis\label{f4}
\end{align*}
\endgroup
Recall from \eqref{eq:omegajdef2}
\[
\onkj=\OMG{\thetagap_j}{\ndk}{1:1}{M_j}{\delta_k}\;.
\]
Therefore we get \eqref{eq:abscondrepeat}.
Now we focus on proving the Proposition~\ref{p:9.1}.

\begin{notation}\text{}
\begin{enumerate}[(i),font=\normalfont\bfseries,topsep=0pt]
\item For $\uond\in(\Dset^\c)^{\nd-m}$ let $\gammacnd{\cdot}{\uond}$ denote the conditional measure of the first $\rigidity-1$ power sums of the zeros of
$\fand$ inside $\Dset$ given that the zeros of $\fand$ outside $\Dset$ are the coordinates of $\uond$. In other words,
$\gammacnd{\cdot}{\uond}$ is the conditional density of $\constraintinternalpc(\zinnd)$ given $\aspc{\uond}=\zoutnd$. 

\item Let $\mucnd{\cdot}{\us}{\uond}$ be the conditional measure of the points of $\zinnd$ taken in uniform random order and considered as a vector, given that the points of $\zoutnd$ are the coordinates of $\uond$ (i.e., $\aspc{\uond}=\zoutnd$) and the power sum vector $\constraintinternalpc(\zinnd)=\us$. 
\end{enumerate}
\end{notation}

\begin{proof}[\textbf{Proof of Proposition~\ref{p:9.1}}]

From Propositions~\ref{p:2:1}-\ref{p:2:4} we have that 
$\OMG{\thetagap}{\nd}{1:1}{M}{\delta}$ is $\delta^3$-included in
$\OMG{\thetagap}{\nd}{2:4}{M}{\delta}$. 
Conditioning successively on $\zoutnd$ and then on $\constraintinternalpc(\zinnd)$ yields the following:

There exists an event $\goodomega$ such that: 
\begin{enumerate}[(1),font=\normalfont\bfseries,topsep=0pt]
\item $\goodomega$ is measurable with respect to $\zoutnd$;
\item $\goodomega\subset\OMG{\thetagap}{\nd}{1:1}{M}{\delta}$;
\item $\Prob{\OMG{\thetagap}{\nd}{1:1}{M}{\delta}\setminus\goodomega}\leq\delta$;
\item for each $\uond$ for which $\goodomega$ holds (i.e., $\aspc{\uond}\in\zoutnd\fbr{\goodomega}$) there exists a Borel measurable set $\pssg\subset\pss$ such that 
\begin{enumerate}[(i),font=\normalfont\bfseries,topsep=0pt]
\item $\gammacnd{\pss\setminus\pssg}{\uond}\leq\delta$;
\item for each $\us\in\pssg$ we have 
\begin{equation}\label{eq:1}
    \mucnd{\Sigma_{\usm}\setminus H\fbr[\big]{\us\boldsemicolon\uond}}{\us}{\uond}
\leq\delta\;,
\end{equation}
where 
\[
H\fbr[\big]{\us\boldsemicolon\uond}\coloneqq\Set*{\uz\in\Sigma_{\usm}\given 
\aspc{\uz\cc\uond}\in\znd\fbr*{\OMG{\thetagap}{\nd}{2:4}{M}{\delta}}}\;.
\]
\end{enumerate}
\end{enumerate}
For $A\in\FUBasisInsidempts$,
$B\in\FUBasisOutside$ let  
\begin{align} 
&\intcnd{A}{B}{1}\coloneqq
\int_{(\zoutnd)^{-1}(B)\;\cap\;\goodomega} 
\int_{\pssg} 
\int_{\asvector{A}}
\deri\mucnd{\uz}{\us}{\uond}
\deri\gammacnd{\us}{\uond}
\deri\PROB\;,\label{I1}\\
&\intcnd{A}{B}{2}\coloneqq
\int_{(\zoutnd)^{-1}(B)\;\cap\;\goodomega} 
\int_{\pssg}
\int_{\asvector{A}\;\cap\;H(\us\boldsemicolon\uond)}
\deri\mucnd{\uz}{\us}{\uond}
\deri\gammacnd{\us}{\uond}
\deri\PROB\;.\label{I2}
\end{align}
Using \eqref{eq:1} we get
\begin{equation}\label{I12} 
\intcnd{A}{B}{1} = \intcnd{A}{B}{2}+O(\delta)\;. 
\end{equation}
Let 
\begin{align*}
& \intcnd{A}{B}{3}\\
& \coloneqq 
\int_{(\zoutnd)^{-1}(B)\;\cap\;\goodomega}
\int_{\pssg}
\frac{\displaystyle\int_{\asvector{A}\;\cap\;H(\us\boldsemicolon\uond)}
|\van{\uz}|^2\deri\leb(\uz)}
{\displaystyle\int_{\Sigma_{\usm}}|\van{\uzp}|^2\deri\leb(\uzp)}
\deri\gammacnd{\us}{\uond}
\deri\PROB\;.\numberthis\label{I3}
\end{align*}
By Proposition~\ref{p:2:5} we have
\begin{equation}\label{I23}
\intcnd{A}{B}{2} 
\;\overset{M,\thetagap}{\scaleto{\asymp}{8pt}}\;
\intcnd{A}{B}{3}\;.
\end{equation}
Let 
\begin{equation}\label{I4}
\intcnd{A}{B}{4} \coloneqq 
\int_{(\zoutnd)^{-1}(B)\,\cap\,\goodomega}\;
\int_{\pssg}\; 
\frac{\displaystyle\int_{\asvector{A}\,\cap\,\Sigma_{\usm}}\abs{\van{\uz}}^2\deri\leb(\uz)}
{\displaystyle \int_{\Sigma_{\usm}}\abs{\van{\uzp}}^2\deri\leb(\uzp)}
\deri\gammacnd{\us}{\uond}
\deri\PROB\;.
\end{equation}
Using \eqref{eq:1} we get
\begin{equation}\label{I34}
\intcnd{A}{B}{3} = \intcnd{A}{B}{4} + O(\delta)\;. 
\end{equation}
Therefore
\[ 
\intcnd{A}{B}{1} 
\;\overset{M,\thetagap}{\scaleto{\asymp}{8pt}}\;
\int_{(\zoutnd)^{-1}(B)\,\cap\,\goodomega} 
\int_{\pssg} 
\h{A}{\us} \deri\gammacnd{\us}{\uond}
\deri\PROB[\uond] + O(\delta)\;. 
\]
Using $0\leq \h{A}{\us}\leq 1$ and  $\gammacnd{\pss\setminus\pssg}{\uond}<\delta$
for $\uond\in\goodomega$, we have 
\begin{equation}\label{eq:above}
\intcnd{A}{B}{1} 
\;\overset{M,\thetagap}{\scaleto{\asymp}{8pt}}\;
\int_{(\zoutnd)^{-1}(B)\,\cap\,\goodomega}
\int_{\pss}\h{A}{\us}
\deri\gammacnd{\us}{\uond}
\deri\PROB[\uond] 
+ O(\delta)\;. 
\end{equation}
Then the integral in the right hand side of 
\eqref{eq:above} can be written as
\[
\int_{(\zoutnd)^{-1}(B)\,\cap\,\goodomega}
\h{A}{\usnd}\deri\PROB\;. 
\]
Using $\PROB\fbr*{\OMG{\thetagap}{\nd}{1:1}{M}{\delta}
\setminus\goodomega}<\delta$ we have
\[
\intcnd{A}{B}{1} 
\;\overset{j}{\scaleto{\asymp}{8pt}}\;
\int_{(\zoutnd)^{-1}(B)\,\cap\,\OMG{\thetagap}{\nd}{1:1}{M}{\delta}}
\h{A}{\usnd}\deri\PROB+O(\delta)\;.
\]
Using $\Prob*{\OMG{\thetagap}{\nd}{1:1}{M}{\delta}\setminus\goodomega}<\delta$ 
and \eqref{I1} we also get 
\begingroup 
\addtolength{\jot}{0.25em}
\begin{align*}
  & \intcnd{A}{B}{1} \\ 
={} & \int_{(\zoutnd)^{-1}(B)\,\cap\,\OMG{\thetagap}{\nd}{1:1}{M}{\delta}}
\int_{\pss}
\int_{A\,\cap\,\Sigma_{\usm}}
\deri\mucnd{\uz}{s}{\uond}
\deri\gammacnd{s}{\uond}
\deri\PROB
+ O(\delta) \\
={} &  \Prob[\Big]{
\,\fbr[\big]{\,\zinnd\in A\,}\,
\cap
\,\fbr[\big]{\,\zoutnd\in B\,}\,
\cap
\OMG{\thetagap}{\nd}{1:1}{M}{\delta}\,} + O(\delta)\;. 
\end{align*}
\endgroup
Therefore
\begingroup 
\addtolength{\jot}{0.25em}
\begin{align*} 
& \Prob[\Big]{
\,\fbr[\big]{\,\zinnd\in A\,}\,
\cap
\,\fbr[\big]{\,\zoutnd\in B\,}\,
\cap
\OMG{\thetagap}{\nd}{1:1}{M}{\delta}}\\
\overset{M,\thetagap}{\scaleto{\asymp}{8pt}}\; & 
\int\h{A}{\usnd} 
\bigindicator{\zoutnd\in B}
\Bigindicator{\OMG{\thetagap}{\nd}{1:1}{M}{\delta}}
\deri\PROB + O(\delta)\;.
\end{align*}
\endgroup
This concludes the proof of Proposition~\ref{p:9.1}.
\end{proof}

\begin{remark}
From the proof of Proposition~\ref{p:9.1}, 
it follows that there exists $\seq*{\correctionterm{j}{k}}_{j\geq 1, k\geq 1}$
such that $\lim_{k\to\infty}\correctionterm{j}{k}=0$ for each $j\geq 1$,
and for 
$A\in\FUBasisInsidempts$, 
$B\in\FUBasisOutside$, 
$j\geq 1$, 
$k\geq 1$
\[
\exp\fbr[\Big]{-f(M_j,\thetagap_j)}
\;\leq\; 
\frac{\displaystyle\Prob[\Big]{
\,\fbr[\big]{\,\zinnk\in A\,}\,
\cap 
\,\fbr[\big]{\,\zoutnk \in B\,}\, 
\cap 
\onkj
\;}}
{\displaystyle\fbr[\Big]{
\int_{\zoutnk^{-1}(B)\;\cap\;\onkj}
\nu_{\Phi,\Psi,\Dset}\kernel{A}{\zout}\deri\PROB}
+ \correctionterm{j}{k}}
\;\leq\;
\exp\fbr[\Big]{f(M_j,\thetagap_j)}\;,
\]
where the function $f$ is the function defined in Proposition~\ref{p:2:5}.
\end{remark}

\section{Proofs of the results in Step 3 of Section~\ref{sec:limitinggaf}}
\label{sec:step3}

\subsection{Proof of Proposition~\ref{p:3:1}}
\label{ss:pf:p:3:1}

Observe that, 
    the event $\OMG{\thetagap}{\nd}{3:1}{M}{\delta}$ 
    is same as the event 
    $\OMG{\thetagap}{\nd}{1:1}{M}{\delta}$
    with $M$ replaced by $M-1$.
So Proposition~\ref{p:3:1} can be proved
    simply by reversing the arguments of 
    Proposition~\ref{p:1:1}. 
Therefore, we skip the details.

\subsection{Proof of Proposition~\ref{p:3:2}}
\label{ss:pf:p:3:2}

Consider $M>3$, $\thetagap\in(0,1)$. 
Our objective is to verify that 
for $\delta\in(0,1)$ sufficiently small
depending on $M$, the conditions defining 
the event 
$\OMG{\thetagap}{\nd}{3:1}{M}{\delta}$ 
also hold on the event 
$\OMG{\thetagap}{\infty}{3:1}{M}{\delta}$ 
except on a $\delta^3$-negligible event.
\begin{enumerate}[(i),font=\normalfont\bfseries,topsep=0pt]
\item Condition~\ref{c:i:3:1:1} defining
$\OMG{\thetagap}{\infty}{3:1}{M}{\delta}$ 
is that the event $\OGm{\infty}{\thetagap}$ 
occurs. Condition~\ref{c:nd:4} defining 
$\nd$ is 
\[
\PROB\fbr[\Big]{\,\OGm{\infty}{\thetagap}\,\symmdiff\,
\OGm{\nd}{\thetagap}\,}\,<\,\delta^3\;.
\]
Therefore, the event $\OGm{\nd}{\thetagap}$ occurs on 
$\OMG{\thetagap}{\infty}{3:1}{M}{\delta}$
except on a $\delta^3$-negligible event. 
Thus condition~\ref{c:n:3:1:1} defining 
$\OMG{\thetagap}{\nd}{3:1}{M}{\delta}$ is satisfied
on $\OMG{\thetagap}{\infty}{3:1}{M}{\delta}$ except on
a $\delta^3$-negligible event.

\item From condition~\ref{c:i:3:1:2} defining $\OMG{\thetagap}{\infty}{3:1}{M}{\delta}$ we have 
\[
\max_{1\leq s<\moment}
\absinv{\falphagaf}{s}{0}{\kdel}
\leq M-\frac{5}{2}\;.
\]
Using Lemma~\ref{l:nd/i:1} we get, sacrificing a $\delta^3$-negligible event, 
\[
\max_{1\leq s<\moment}
\absinv{\fand}{s}{0}{\kdel}
\leq M-\frac{5}{2}+\delta 
\leq M-1\;,
\]
which is condition~\ref{c:n:3:1:2} defining $\OMG{\thetagap}{\nd}{3:1}{M}{\delta}$.

\item From condition~\ref{c:i:3:1:3} defining $\OMG{\thetagap}{\infty}{3:1}{M}{\delta}$ we have \[
\invabs{\falphagaf}{\moment}{0}{\kdel}<M-\frac{5}{2}\;.
\]
Using Lemma~\ref{l:nd/i:1} we get, sacrificing a $\delta^3$-negligible event, 
\[
\invabs{\fand}{\moment}{0}{\kdel}<M-1\;,
\]
which is condition~\ref{c:n:3:1:3} defining $\OMG{\thetagap}{\nd}{3:1}{M}{\delta}$.

\item Using Lemma~\ref{l:tailsum} we have
\[
\QFI{\zppv} = \abs*{\tailsum} \frac{(\nd!)^{\alpha}}{\abs*{\sym{\nd-m}{\uo}}^2}\;.
\]
From \eqref{eq:al:1:1} in Lemma~\ref{l:al:1} we get 
\[
\frac{(\nd!)^{\alpha}}{\abs*{\sym{\nd-m}{\uo}}^2}
=\frac{\InProdSqNd}{\abs*{\xi_0}^2}\abs*{\xi_{\nd}}^2\;.
\]
Also recall the Notation~\ref{n:hadamard}, since $\underline{\omega}\in\big(\Dset^\c \big)^{n_\delta-m}$, we get
\[
\tailsum=\sum_{k=\nd-\Cdel}^{\ndmij}
\frac{\overline{\sym{k-i}{\uo}}}{\ndckat}
\frac{\sym{k-j}{\uo}}{\ndckat}\;.
\]
Here we use the fact that $\overline{\sigma_{k-i}(\underline{\omega})}\sigma_{k-j}(\underline{\omega})=0$ when $k>n_\delta-m+(i\wedge j)$.
Thus
\begin{equation}\label{eq:pf:p:3:2:1}
\QFI{\zppv} = 
\frac{\InProdSqNd}{\abs*{\xi_0}^2}
\abs*{\sum_{k=\nd-\Cdel}^{\ndmij}
\frac{\overline{\sym{k-i}{\uo}}}{\ndckat}
\frac{\sym{k-j}{\uo}}{\ndckat}}
\abs*{\xi_{\nd}}^2\;.
\end{equation}
From condition~ \ref{c:i:3:1:4} defining $\OMG{\thetagap}{\infty}{3:1}{M}{\delta}$ we have 
\[
\frac{\InProdSq}{\abs*{\xi_0}^2}\leq \fbr[\big]{ M-2 }^{1/2}-\delta\;.
\]
From condition~\ref{c:nd:3} defining $\nd$ we have 
\[
\abs*{\frac{\InProdSqNd}{\abs*{\xi_0}^2} 
-\frac{\InProdSq}{\abs*{\xi_0}^2}}<\delta
\]
except on a $\delta^3$-negligible event. Thus we get
\begin{equation}\label{eq:pf:p:3:2:2}
\frac{\InProdSqNd}{\abs*{\xi_0}^2}
\leq \fbr[\big]{ M-2 }^{1/2}\;.
\end{equation}
Using Propositions~\ref{p:rec} we get 
\[
\sum_{k=\nd-\Cdel}^{\ndmij}
\frac{\overline{\sym{k-i}{\uo}}}{\ndckat}
\frac{\sym{k-j}{\uo}}{\ndckat}
=\frac{(-1)^{i+j}}{\abs*{\xi_{\nd}}^2}
\fbr[\Big]{\bTp{1}+\bTp{2}+\bTp{3}+\bTp{4}}\;,
\]
where
{
\begingroup
\addtolength{\jot}{0.25em}
\begin{align*}
\bTp{1} \coloneqq{} &
\sum_{r_1=0}^{\lev-1}
\sum_{r_2=0}^{\lev-1}
(-1)^{r_1+r_2}\cdot\overline{\g_{r_1}}\cdot\g_{r_2}\cdot
\fbr*{\sum_{l=\mij}^{\Cdel} 
\frac{\overline{\xi_{l+i+r_1}}\cdot\xi_{l+j+r_2}}
{\ffa{l}{i+r_1}\cdot\ffa{l}{j+r_2}}}\;,\\
\bTp{2} \coloneqq{} & 
\sum_{r_2=0}^{\lev-1}
(-1)^{r_2}\cdot\g_{r_2}\cdot
\fbr*{\sum_{l=\mij}^{\Cdel} 
\frac{\overline{\tail{l+i}{\nd}}\cdot\xi_{l+j+r_2}}
{\ffa{l}{i}\cdot\ffa{l}{j+r_2}}}\;,\\
\bTp{3} \coloneqq{} &
\sum_{r_1=0}^{\lev-1}
(-1)^{r_1}\cdot\overline{\g_{r_1}}\cdot
\fbr*{\sum_{l=\mij}^{\Cdel} 
\frac{\overline{\xi_{l+i+r_1}}\cdot\tail{l+j}{\nd}}
{\ffa{l}{j}\cdot\ffa{l}{i+r_1}}}\;,\\
\bTp{4} \coloneqq{} &  
\sum_{l=\mij}^{\Cdel} 
\frac{\overline{\tail{l+i}{\nd}}\cdot\tail{l+j}{\nd}}
{\ffa{l}{i}\cdot\ffa{l}{j}}\;.
\end{align*}
\endgroup}
\noindent By condition~\ref{c:i:3:1:5} defining 
$\OMG{\thetagap}{\infty}{3:1}{M}{\delta}$ we have
$\abs*{\bTp{p}}\leq \fbr[\big]{ M-2 }^{1/2}/4$ for each 
$1\leq p\leq 4$. 
Therefore 
\begin{equation}\label{eq:pf:p:3:2:3}
\abs*{\sum_{k=\nd-\Cdel}^{\ndmij}
\frac{\overline{\sym{k-i}{\uo}}}{\ndckat}
\frac{\sym{k-j}{\uo}}{\ndckat}}
\leq \frac{1}{\abs*{\xi_{\nd}}^2} \fbr[\big]{ M-2 }^{1/2}\;.
\end{equation}
Combining \eqref{eq:pf:p:3:2:1} - \eqref{eq:pf:p:3:2:3} we get
\[
    \QFI{\zppv}\le M-2\;.
\]
Using Lemma~\ref{l:kd/i:2}, sacrificing a $\delta^3$-negligible event, we get 
\[
\QFI{\zpv}\leq M-2+\delta\leq M-1\;,
\]
which is condition~\ref{c:n:3:1:4} defining 
$\OMG{\thetagap}{\nd}{3:1}{M}{\delta}$.

\item By condition~\ref{c:nd:3} defining $\nd$ we have
\[
\abs*{\frac{\InProdSq}{\abs*{\xi_0}^2}
-\frac{\InProdSqNd}{\abs*{\xi_0}^2}}
<\delta
\]
except on a $\delta^3$-negligible event.
Therefore, using (\ref{c:p:2:3:1}) and Lemma~\ref{l:kd/i:2} we get 
that the following is true except on a 
$\delta^3$-negligible event:
\begingroup
\addtolength{\jot}{0.25em}
\begin{align*}
     \QFZ{\zpv}
\geq{} & \QFZ{\zppv} 
- \abs*{\QFZ{\zpv} - \QFZ{\zppv}}\\ 
\geq{} & \frac{1}{8}\frac{\InProdSqNd}{\abs*{\xi_0}^2} -\delta\\ 
\geq{} & \frac{1}{8}\frac{\InProdSq}{\abs*{\xi_0}^2}
- \frac{1}{8}\abs*{\frac{\InProdSqNd}{\abs*{\xi_0}^2} 
 - \frac{\InProdSq}{\abs*{\xi_0}^2}} -\delta\\ 
\geq{} & \frac{1}{8}\frac{\InProdSq}{\abs*{\xi_0}^2} -
\frac{9}{8}\delta\;.
\numberthis\label{exh1}
\end{align*}
\endgroup
We can apply (\ref{c:p:2:3:1}) here since we have shown that $\OGm{\nd}{\thetagap}\subset\Omega_{\nd}^m$ occurs except on a $\delta^3$-negligible event.  
Similarly, 
except on a $\delta^3$-negligible event, 
we have:
\begin{equation}\label{exh2} 
\QFZ{\zpv}\leq
\frac{27}{8}\frac{\InProdSq}{\abs*{\xi_0}^2}+\frac{35}{8}\delta\;. 
\end{equation}
Equations~\eqref{exh1} and \eqref{exh2},
along with condition~\ref{c:i:3:1:4} defining 
$\OMG{\thetagap}{\infty}{3:1}{M}{\delta}$ gives 
us condition~\ref{c:n:3:1:5} defining 
$\OMG{\thetagap}{\nd}{3:1}{M}{\delta}$. 
This concludes the proof of Proposition~\ref{p:3:2}.
\end{enumerate}

\subsection{Proof of Proposition~\ref{p:3:3}}
\label{ss:pf:p:3:3}

We verify that the conditions defining the event 
$\OMG{\thetagap}{\infty}{3:2}{M}{\bullet}$ 
are also satisfied on the event 
$\OMG{\thetagap}{\infty}{3:1}{M}{\delta}$ 
except on a $\delta^3$-negligible event. 
\begin{enumerate}[(i),font=\normalfont\bfseries,topsep=0pt]
\item Condition~\ref{c:i:3:1:1} defining 
$\OMG{\thetagap}{\infty}{3:1}{M}{\delta}$ 
is same as condition~\ref{c:i:3:2:1} defining 
$\OMG{\thetagap}{\infty}{3:2}{M}{\bullet}$.

\item From condition~\ref{c:i:3:2:2} defining $\OMG{\thetagap}{\infty}{3:2}{M}{\bullet}$ we have
\[
\max_{1\leq s\leq\moment}\;
\absinv{\falphagaf}{s}{0}{\infty}
\leq M-3\;.
\]
Using Lemma~\ref{l:kd/i:1} and assuming $\delta$ is small enough
we get that except on a $\delta^3$-negligible event
\[
\max_{1\leq s\leq\moment}\;
\absinv{\falphagaf}{s}{0}{\kdel}
\leq M-3+\delta 
\leq M-\frac{5}{2}\;,
\]
which is condition~\ref{c:i:3:1:2} 
defining $\OMG{\thetagap}{\infty}{3:1}{M}{\delta}$.

\item From condition~\ref{c:i:3:2:3} defining $\OMG{\thetagap}{\infty}{3:2}{M}{\bullet}$ we have \[
\invabs{\falphagaf}{\moment}{0}{\infty}
\leq M-3\;.
\]
Using Lemma~\ref{l:kd/i:1} and assuming $\delta$ is small enough
we get that except on a $\delta^3$-negligible event
\[
\invabs{\falphagaf}{\moment}{0}{\kdel}
\leq M-3+\delta
\leq M-\frac{5}{2}\;,
\]
which is condition~\ref{c:i:3:1:3} defining 
$\OMG{\thetagap}{\infty}{3:1}{M}{\delta}$.

\item Condition~\ref{c:i:3:2:4} defining 
$\OMG{\thetagap}{\infty}{3:2}{M}{\bullet}$
implies condition~\ref{c:i:3:1:4} defining 
$\OMG{\thetagap}{\infty}{3:1}{M}{\delta}$
when $\delta$ is small enough depending on $M$.

\item For $0\leq i\leq m$, 
$\rigidity\leq j\leq m$, 
$0\leq r_1\leq \lev$, 
$0\leq r_2\leq \lev$, 
using condition~\ref{c:Cd:2} defining $\Cdel$ 
and the Chebyshev inequality we get
\[
\abs*{
\sum_{l=\Cdel+1}^{\infty} 
\frac{\overline{\xi_{l+i+r_1}}\cdot
\xi_{l+j+r_2}}{\ffa{l}{i+r_1}\cdot\ffa{l}{j+r_2}}
}
\leq\delta 
\]
except on a $\delta^3$-negligible event. 
Similarly using Proposition~\ref{p:rec} and
a Markov bound we get
\begingroup
\addtolength{\jot}{0.25em}
\begin{align*}
\abs*{
\sum_{l=\Cdel+1}^{\infty} 
\frac{\overline{\tail{l+i}{\nd}}\cdot\xi_{l+j+r_2}}
{\ffa{l}{i}\cdot\ffa{l}{j+r_2}}
}
\leq{} & \delta\;,\\
\abs*{\sum_{l=\Cdel+1}^{\infty} 
\frac{\xi_{l+i+r_1}\cdot\overline{\tail{l+j}{\nd}}}
{\ffa{l}{j}\cdot\ffa{l}{i+r_1}}}
\leq{} & \delta\;,\\
\abs*{
\sum_{l=\Cdel+1}^{\infty} 
\frac{\overline{\tail{l+i}{\nd}}\cdot\tail{l+j}{\nd}}
{\ffa{l}{i}\cdot\ffa{l}{j}}
}
\leq{} & \delta\;,
\end{align*}
\endgroup
except on a $\delta^3$-negligible event. 
Therefore condition~\ref{c:i:3:1:5} defining 
$\OMG{\thetagap}{\infty}{3:1}{M}{\delta}$ is 
implied by condition~\ref{c:i:3:2:5} 
defining $\OMG{\thetagap}{\infty}{3:2}{M}{\bullet}$.
\end{enumerate}
This concludes the proof of Proposition~\ref{p:3:3}.

\subsection{Proof of Theorem~\ref{t:s:3}}\label{ss:pf:t:s:3}

Recall from \eqref{eq:omegajdef} that 
    $\oj=\liminf_{k\to\infty}\; 
    \OMG{\thetagap_j}{\infty}{1:1}{M_j}{\delta_k}$.
From Propositions~\ref{p:3:1}-\ref{p:3:3} we get the following. 
For each $j\geq 1$, and large enough $k$
    (so that $\delta_k$ is small enough)
    $\OMG{\thetagap_j}{\infty}{3:2}{M_j}{\bullet}$ 
    is $\delta_k^3$-included in 
    $\OMG{\thetagap_j}{\infty}{1:1}{M_j}{\delta_k}$ 
    i.e., there exists an event 
    $\mathcal{E}_{\delta_k}$ 
    such that 
    \[
    \OMG{\thetagap_j}{\infty}{3:2}{M_j}{\bullet}
    \;\setminus\;
    \mathcal{E}_{\delta_k}
    \;\subset\;
    \OMG{\thetagap_j}{\infty}{1:1}{M_j}{\delta_k}
    \]
    and $\PROB\fbr*{\mathcal{E}_{\delta_k}}\leq C\delta_k^3$ 
    for some constant $C>0$. 
    From Definition~\ref{d:i:1:1} we get 
    \[
    \OMG{\thetagap_j}{\infty}{1:1}{M_j}{\delta_k}
    \;\subset\;
    \OGm{\infty}{\thetagap_j}
    \;\subset\;
    \om\;.
    \]
Therefore, using the Borel-Cantelli Lemma 
    and \eqref{eq:deltakchoice} we get
    \[
    \OMG{\thetagap_j}{\infty}{3:2}{M_j}{\bullet}
    \;\subset\;
    \oj
    \;\subset\;
    \OGm{\infty}{\thetagap_j}
    \;\subset\;
    \om\;.
    \]
Thus, it is enough to show that 
    $\seq*{\OMG{\thetagap}{\infty}{3:2}{M_j}{\bullet}}_{j=1}^\infty$
    exhausts $\om$. 
To this end, 
    we observe that the random variables 
    appearing in the definition of 
    $\OMG{\thetagap}{\infty}{3:2}{M}{\bullet}$ 
    (Definition~\ref{d:i:3:2}) have no mass at $\infty$. 
Also the random variable $\InProdSq/\abs*{\xi_0}^2$, 
    appearing in condition~\ref{c:i:3:2:2}, 
    has no atom at $0$. 
From the definition of the event 
    $\OGm{\infty}{\thetagap}$ (Definition~\ref{d:gaf:gap}) 
    it is clear that 
    $\seq*{\OGm{\infty}{\thetagap_j}}_{j=1}^\infty$ exhausts $\om$. 
Therefore, $\seq*{\OMG{\thetagap}{\infty}{3:2}{M_j}{\bullet}}_{j=1}^\infty$ exhausts $\om$. 
Therefore, $\seq*{\oj}_{j=1}^\infty$ exhausts $\om$.
This concludes the proof of Theorem~\ref{t:s:3}.

\section*{Acknowledgments}
\addcontentsline{toc}{section}{Acknowledgments}
S.G. was supported in part by the Singapore Ministry of Education grants R-146- 000-250-133, R-146-000-312-114 and MOE-T2EP20121-0013. S.G. is very grateful to Fedor Nazarov for many insightful discussions on the zeros of the standard planar Gaussian analytic function. The authors thank the anonymous referee for his/her suggestions for the improvement of the paper.

\addcontentsline{toc}{section}{References}
\bibliographystyle{plain}

\begin{appendix}

\section{Reduction of \texorpdfstring{$\Dset$}{Dset} from a general domain to a disk}\label{a:disk}

Here we show that in order to prove Theorems~\ref{thm:ginibre-main} and \ref{thm:alphagaf-main}, it is enough to prove them when the domain $\Dset$ is a disk. 
First we will show this in the context of Theorem~\ref{thm:ginibre-main}.
The proof in the context of Theorem~\ref{thm:alphagaf-main} is essentially the same.
So we will only present a rough sketch and point out the main differences.

Let us suppose that Theorem~\ref{thm:ginibre-main} holds in the case $\Dset$ is a disk. 
Consider $\Dset$ to be a bounded open set in $\CC$ whose boundary has zero Lebesgue measure. 
By the translation invariance of $\GIF$, we assume that the origin is in the interior of $\Dset$. 
Let $\Dset_0$ be a disk centered at the origin containing the closure of $\Dset$ in its interior.

\begin{figure}[H]
\centering
\begin{tikzpicture}
\draw[fill=gray!30,draw=gray!30] (-3.56,2.2) rectangle (3.56,-2.2);
\draw[very thick, fill=gray!20](0,0) circle (2);
\draw[thick, fill=gray!10] plot [smooth cycle] coordinates {(-1.0,-1.4)(-0.5,-1.3)(0.6,-1.0)(0.7,-0.0)(0.6,1.3)(-0.6,1.0)(-1.5,-1.0)}; 
\draw[fill=black] (0,0) circle (2pt);
\draw node at (0.2,0.2) {$\boldsymbol{0}$};
\draw node at (-0.6,0.2) {$\Dset$};
\draw node at (1.3,0.2) {$\Dset_0\setminus\Dset$};
\draw node at (-2.5,0.2) {$\Dset_0^c$};
\end{tikzpicture}
\caption{The domain $\Dset$ is contained inside the disk $\Dset_0$. 
The point configuration $\Upsilon$ is a configuration outside $\Dset$,
i.e., it is supported on $\Dset^\c = (\Dset_0\setminus\Dset)\sqcup\Dset_0^\c$. 
The point configuration $\Upsilon_0$ is restriction of $\Upsilon$ on $\Dset_0^\c$.}
\end{figure}

Consider a point configuration $\Upsilon\in\cS{\Dset^\c}$. 
Let $\Upsilon_0\coloneqq\Upsilon\cap\Dset_0^\c\in\cS{\Dset_0^\c}$ be the restriction of $\Upsilon$ on $\Dset_0^\c$.
Consider the event $\GIFOUT=\Upsilon$.
Since the Ginibre ensemble is number-rigid (see Theorem~\ref{t:1.1} for reference),
the point configuration $\Upsilon_0$ determines the number 
$N_0$ of points of $\GIF$ inside $\Dset_0$. 
From $\Upsilon$ the number of points in 
$\Dset_0\setminus\Dset$ is also determined.
Thus we can determine the number $N$ of points inside 
$\Dset$ from $\Upsilon$.

By our assumption, Theorem~\ref{thm:ginibre-main} holds for $\Dset_0$. Thus, the conditional distribution of the vector consisting of points of $\GIF$ inside $\Dset_0$ taken in uniform random order, given $\GIF\cap\Dset_0^\c=\Upsilon_0$, is supported on $\Dset_0^{N_0}$. Moreover, this distribution has a density, say $f_0$, with respect to the Lebesgue measure. This density satisfies
\begin{equation}\label{eq:ap:1.1}
\mm_0\fbr*{\Upsilon_0}\abs*{\van{\uz}}^2 \le f_0\fbr*{\uz} \le \MM_0\fbr*{\Upsilon_0}\abs*{\van{\uz}}^2 
\end{equation}
for some measurable functions $\mm_0,\MM_0:\cS{\Dset_0^\c}\to(0,\infty)$. Let $\rho\kernel{\cdot}{\Upsilon_\out}$ be the conditional distribution of the vector consisting of the points of $\GIF$ inside $\Dset$ taken in uniform random order, given $\GIFOUT=\Upsilon$. Then for all Borel subset $A\subset\Dset^N$ we have \[ \rho\kernel{A}{\Upsilon}=\frac{\displaystyle\int_A f_0\fbr*{\uz^{[1]},\uz^{[2]}}\deri\el\fbr*{\uz^{[1]}}}{\displaystyle\int_{\Dset^N} f_0\fbr*{\uz^{[1]},\uz^{[2]}}\deri\el\fbr*{\uz^{[1]}}}\;,\] where $\uz^{[2]}$ is a vector consisting of the points of $\Upsilon$ in $\Dset_0\setminus\Dset$, $\uz^{[1]},\uz^{[2]}$ denotes the concatenated vector, and the integrals are taken with respect to $\uz^{[1]}\in\Dset^N$. The denominator is a measurable function of $\Upsilon$. To get an upper bound of $\rho\kernel{A}{\Upsilon}$, we use the upper bound of $f_0$ from \eqref{eq:ap:1.1}. The quantity $\MM_0(\Upsilon_0)$ is measurable with respect to $\Upsilon$ because $\Upsilon_0$ is the restriction of $\Upsilon$ on $\Dset_0^\c$. We can split the Vandermonde term as \[ \van{\uz}=\van{\uz^{[1]}}\cdot\van{\uz^{[2]}}\cdot\vancross{\uz^{[1]}}{\uz^{[2]}}\;. \] The term $\van{\uz^{[2]}}$ is measurable with respect to $\Upsilon$. The term $\vancross{\uz^{[1]}}{\uz^{[2]}}$ is bounded above by $(2r_0)^{N N_0}$ where $r_0$ is the radius $\Dset_0$, and it is bounded below by $\thetagap^{N N_0}$ where $\thetagap$ is the gap between points in $\Upsilon$ and the boundary of $\Dset$. This gap is almost surely positive and measurable with respect to $\Upsilon$. Hence we get that the ratio $\rho\kernel{A}{\Upsilon}/\int_A\abs{\van{\uz^{[1]}}}^2\deri\el\fbr{\uz^{[1]}}$ is bounded above and below by quantities measurable with respect to $\Upsilon$. Then we get \eqref{eq:target2} by the Radon-Nikodym Theorem.

Now let us show that in the context of Theorem~\ref{thm:alphagaf-main} it is enough to assume that $\Dset$ is a disk. As mentioned before, we only point out the key differences. 
Suppose Theorem~\ref{thm:alphagaf-main} is true when $\Dset$ is a disk.
Now consider $\Dset$ to be a bounded open set in $\CC$ whose boundary has zero Lebesgue measure.
By the translation invariance of $\alphagaf$,
we take the origin to be in the interior of 
$\Dset$. Let $\Dset_0$ be a disk centered at origin
containing $\closure{\Dset}$ in its interior. 
Consider the event $\alphagafout=\Upsilon$. 
Let $\Upsilon_0\coloneqq\Upsilon\cap\Dset_0^\c\in\cS{\Dset_0^\c}$ be the restriction of $\Upsilon_0$ on $\Dset_0^\c$.
By Theorem~\ref{thm:app:alphagaf}, the point configuration $\Upsilon_0$ determines the number $N_0$, and the power sums up to order $\rigidity-1$ $\us^{(0)}=(s_1^{(0)},\dots,s_{\rigidity-1}^{(0)})$ of points inside $\Dset_0$. 
Similarly, the point configuration $\Upsilon$ determines the number $N$, and the power sums up to order $\rigidity-1$, say $\us=(s_1,\dots,s_{\rigidity-1})$ of the points inside $\Dset$. 
Define
\[
\Sigma\coloneqq\Set*{(\zeta_1,\dots,\zeta_N)\in\Dset^{N}\given\sum_{i=1}^N\zeta_i^j=s_j\mbox{ for all }1\leq j\leq \rigidity-1}\;,
\]
and 
\[
\Sigma^{(0)}\coloneqq\Set*{(\zeta_1,\dots,\zeta_{N_0})\in\Dset_0^{N_0}\given\sum_{i=1}^N\zeta_i^j=s_j^{(0)}\mbox{ for all }1\leq j\leq \rigidity-1}\;.
\]
By the assumption that Theorem~\ref{thm:alphagaf-main} is valid for $\mD_0$ we get that the conditional distribution of the vector of points inside $\Dset_0$, given $\alphagaf\cap\Dset_0^\c=\Upsilon_0$, lives on $\Sigma^{(0)}$, and has a density with respect to the Lebesgue measure on $\Sigma^{(0)}$, say $f_0$, which satisfies
\[
\mm_0\fbr*{\Upsilon_0}\abs*{\van{\uz}}^2\le f_0(\uz)\le \MM_0\fbr*{\Upsilon_0}\abs*{\van{\uz}}^2 
\]
for some measurable functions $\mm_0,\MM_0:\cS{\Dset_0^\c}\to(0,\infty)$. The rest of the proof is same as in the context of Theorem~\ref{thm:ginibre-main}. The only difference is, instead $\Dset^N$ we now have $\Sigma$, and instead of $\Dset_0^{N_0}$ we have $\Sigma^{(0)}$.

\section{Ginibre Ensemble: Definition, Rigidity, Tolerance}
\label{a:ginibre}

Consider a $n\times n$ matrix $M^n$ whose entries are i.i.d.\ 
standard complex Gaussian random variables. 
The vector of its eigenvalues, in uniform random order, 
has a joint density with respect to the Lebesgue measure on 
$\CC^n$ given by
\[
p\fbr*{z_1,\ldots,z_n}
=\fbr*{\pi^n\prod_{k=1}^n k!}^{-1}\cdot
\exp\fbr*{-\sum_{k=1}^n\abs*{z_k}^2}\cdot
\abs*{\van{z_1,\ldots,z_n}}^2\;.
\]
The set of eigenvalues of $M^n$, 
considered as a random point configuration, 
is called the $n$-dimensional Ginibre ensemble.
We denote it by $\Ginin$.
As $n\to\infty$, $\Ginin$ 
converges in distribution to a point process 
called the (infinite) Ginibre
ensemble, denoted by $\GIF$. 
Both $\Ginin$ and $\GIF$ are determinantal point processes.
Recall that a determinantal point process on the 
Euclidean space $\RR^d$ with kernel $\bbK$ and
background measure $\mu$ is a point process on $\RR^d$ 
whose $k$-point intensity functions with respect to the 
measure $\mu^{(k)}$ are given
by
\[
p_k\fbr*{x_1,\ldots,x_k} \coloneqq
\det\tbr*{\seq[\big]{\bbK(x_i,x_j)}_{1\leq i,j\leq k}}\;. 
\]
The finite Ginibre ensemble $\Ginin$ is
a determinantal point process with kernel 
$\bbK_n(z,w)\coloneqq\sum_{k=0}^{n-1}\frac{(z\Bar{w})^k}{k!}$
with respect to the background measure 
$\deri\gamma(z)=\frac{1}{\pi}e^{-|z|^2}\deri\mathcal{L}$, 
where $\mathcal{L}$ denotes the Lebesgue measure on $\CC$.
The Ginibre ensemble $\GIF$ is a determinantal point process 
with kernel $\bbK_{\infty}(z,w)\coloneqq\sum_{k=0}^{\infty}\frac{(z\Bar{w})^k}{k!}$ with
respect to the measure $\gamma$. We can define $\seq*{\Ginin}$ and $\GIF$ on the same 
probability space such that $\Ginin\to\GIF$ a.s. We refer to \cite{GP} for a brief explanation. The following rigidity and tolerance properties were established in \cite{GP}.

\begin{theorem}[Rigidity of the Ginibre ensemble: Theorem~1.1 in \cite{GP}]
Let $\Dset\subset\CC$ be a bounded open set whose boundary has zero Lebesgue measure.
For the Ginibre ensemble $\GIF$, there is a measurable 
function $\numberofpoints:\cS{\Dset^\c}\to\nat$ such that 
the number of points in $\GIF\cap\Dset$ is $\numberofpoints\fbr[\big]{\GIF\cap\Dset^\c}$ a.s.
\label{t:1.1}
\end{theorem}

\begin{theorem}[Tolerance of the Ginibre ensemble: Theorem~1.2 in \cite{GP}]\label{thm:Gini:tol}
Let $\Dset\subset\CC$ be a bounded open set whose boundary has zero Lebesgue measure.
Let $\frho{\cdot}{\GIF\cap\Dset^\c}$ be the conditional measure of $\GIF\cap\Dset$ given $\GIF\cap\Dset^\c$ where we identify $\GIF\cap\Dset$ with a vector in $\Dset^{\numberofpoints(\GIF\cap\Dset^\c)}$ by taking the points of $\GIF\cap\Dset$ in uniform random order. Then the measure $\frho{\cdot}{\GIF\cap\Dset^\c}$ and the Lebesgue measure on $\Dset^{\numberofpoints(\GIF\cap\Dset^\c)}$ are mutually absolutely continuous a.s.
\end{theorem}

\section{Gaussian Analytic Function: Definition, Rigidity, Tolerance}\label{a:gaf}

Let $\seq*{\xi_k}_{k=0}^{\infty}$ be i.i.d.\ standard complex Gaussian random variables. The $n$-dimensional standard Gaussian analytic function and the (infinite dimensional) standard Gaussian analytic function are 
\[
\fgafn(z)\coloneqq\sum_{k=0}^n\frac{\xi_k}{(k!)^{1/2}}z^k
\quad\mbox{and}\quad
\fgaf(z)\coloneqq\sum_{k=0}^\infty\frac{\xi_k}{(k!)^{1/2}}z^k\;,
\]
respectively. Both $\fgafn$ and $\fgaf$ are Gaussian processes on $\CC$     with covariance kernels given by 
\[
\bbK_n(z,w)\coloneqq\sum_{k=0}^n \frac{(z\overline{w})^k}{k!}
\quad\mbox{and}\quad 
\bbK(z,w)\coloneqq\sum_{k=0}^\infty\frac{(z\overline{w})^k}{k!}=e^{z\overline{w}}
\]
respectively. We denote the ensemble of roots of $\fgaf$ by $\gaf$ and
the ensemble of roots of $\fgafn$ by $\gafn$. 
As $n\to\infty$, $\gafn$ converges to $\gaf$ a.s.
The ensemble $\gaf$ satisfies the following rigidity and tolerance properties.

\begin{theorem}[Rigidity of the GAF-zero ensemble: Theorem~1.3 in \cite{GP}]
For the GAF-zero ensemble $\gaf$ we have: 
\begin{enumerate}[(a),font=\normalfont\bfseries,topsep=0pt]
\item there is a measurable function $\numberofpoints:\cS{\Dset^\c}\to\nat$ such that the number of points in $\gaf\cap\Dset$ is $\numberofpoints\fbr[\big]{\gaf\cap\Dset^\c}$ a.s.
\item there is a measurable function 
    $\extsum:\cS{\Dset^\c}\to\CC$ such that
    the sum of the points in $\gaf\cap\Dset$ is $    \extsum\fbr[\big]{\gaf\cap\Dset^\c}$ a.s.
\end{enumerate}
\label{thm:app:gaf:rigidity}
\end{theorem}

\begin{theorem}[Tolerance of the GAF-zero ensemble: Theorem~1.4 in \cite{GP}] \label{thm:app:gaf:tol}
Consider the submanifold 
\[
\Sigma_{m,s}\coloneqq\Set*{(\zeta_1,\ldots,\zeta_m)\in\Dset^m\given \zeta_1+\cdots+\zeta_m=s}
\]
where $m=\numberofpoints\fbr[\big]{\gaf\cap\Dset^\c}$, and $s=\extsum\fbr[\big]{\gaf\cap\Dset^\c}$.
Let $\frho{\cdot}{\gaf\cap\Dset^\c}$ be the measure on $\Sigma_{m,s}$ which is the conditional distribution of $\gaf\cap\Dset$ given $\gaf\cap\Dset^\c$ where we identify $\gaf\cap\Dset$ with a vector in $\Dset^m$ by taking the points in uniform random order. Then, the measure $\frho{\cdot}{\gaf\cap\Dset^\c}$ and the Lebesgue measure on $\Sigma_{\usm}$ are mutually absolutely continuous a.s.
\end{theorem}

\section{\texorpdfstring{$\alpha$}{alpha}-Gaussian Analytic Function: Definition, Rigidity}\label{a:agaf}

Let $\seq*{\xi_k}_{k=0}^\infty$ be i.i.d.\ standard complex Gaussian random variables. 
For $\alpha>0$, the $n$-dimensional $\alpha$-Gaussian analytic function and 
    the (infinite dimensional) $\alpha$-Gaussian analytic function are respectively
    \[
    \falphagafn(z)\coloneqq\sum_{k=0}^n\frac{\xi_k}{(k!)^{\alpha/2}} z^k
    \quad\mbox{and}\quad
    \falphagaf(z):=\sum_{k=0}^\infty\frac{\xi_k}{(k!)^{\alpha/2}} z^k\;. 
    \]
These are Gaussian processes on $\CC$ with covariance kernels given by 
    \[
    \bbK_n(z,w)=\sum_{k=0}^n \frac{(z\overline{w})^k}{(k!)^\alpha}
    \quad\mbox{and}\quad 
    \bbK(z,w)=\sum_{k=0}^\infty\frac{(z\overline{w})^k}{(k!)^{\alpha}}
    \]
respectively. The ensemble of roots of $\falphagafn$ and $\falphagaf$ are denoted by $\alphagafn$ and $\alphagaf$ respectively.
As $n\to\infty$, $\alphagafn$ converges to $\alphagaf$ almost surely. The ensemble $\alphagaf$ satisfies the following rigidity property.

\begin{theorem}[Rigidity of the ensemble of roots of $\alpha$-GAF: Theorem~2.1 in \cite{GK}]
Let $\Dset\subset\CC$ be a bounded open set whose boundary has zero Lebesgue measure.
\begin{enumerate}[(a),font=\normalfont\bfseries,topsep=0pt]
\item There exists a measurable function $\numberofpoints:\cS{\Dset^\c}\to\nat$ such that the number of points of $\alphagaf\cap\Dset$ is $\numberofpoints\fbr[\big]{\alphagaf\cap\Dset^\c}$ a.s.
\item Let $\rigidity\coloneqq 1+\foa$. There is a measurable function 
$\constraintexternal:\cS{\Dset^\c}\to\CC^{\rigidity-1}$ such that, if
$m\coloneqq\numberofpoints\fbr[\big]{\alphagaf\cap\Dset^\c}$
and 
$(\zeta_1,\dots,\zeta_m)$ are the points of
$\alphagaf\cap\Dset$,
then
$\constraintexternal\fbr[\big]{\alphagaf\cap\Dset^\c}\coloneqq(s_1,\dots,s_{\rigidity-1})$, 
where $s_j\coloneqq\sum_{i=1}^m \zeta_i^j$ for all $1\leq j\leq\rigidity-1$.
\end{enumerate} 
\label{thm:app:alphagaf}
\end{theorem}

\section{Estimates for the finite Ginibre ensemble}\label{a:ginibre-estimates}

Consider the finite Ginibre ensemble $\Gini_n$.
Let $\Dset\subset\CC$ be a bounded open set whose boundary has zero Lebesgue measure.
Consider $1\leq m\leq n$ and $\uo\in(\Dset^\c)^{n-m}$.
The density with respect to the Lebesgue measure on $\Dset^m$ 
of the conditional distribution of a vector consisting of
the points of $\Ginin\cap\Dset$ taken in uniform random order
given that the points of $\Ginin\cap\Dset^\c$ are the coordinates of $\uo$ is given by
\[
\drho{\uo}{n}{\uz} \coloneqq C(\uo)\cdot\abs*{\van{\uz\cc\uo}}^2\cdot\exp\fbr*{-\sum_{k=1}^m\abs*{\z_k}^2}\;,
\]
where $C(\uo)$ is the appropriate normalizing constant. 
Recall the definition of $\ginisumone(n)$, $\ginisumtwo(n)$, $\ginisumthree(n)$, and $\X_n$ 
as defined in Notation~\ref{n:ginisum}. We have the following results from \cite{GP}.

\begin{proposition}[Bounding the fluctuation of conditional densities: Proposition 8.2 in \cite{GP}]\label{prop:ginicondratio}
Consider positive integers $n$, $m$ with $n\geq m$.
Consider $\uo\in(\Dset^\c)^{n-m}$ such that the points given by the coordinates of $\uo$
are at least at a distance $\thetagap>0$ from the boundary of $\Dset$.
Then for all $\uz,\uzp\in\Dset^m$ we have
    \[
    \exp\fbr*{ - m \usekd{pgcr}\thetagap^{-1} \X_n}
    \abs*{\frac{\van{\uzp}}{\van{\uz}}}^2
    \;\leq\; 
    \frac{\drho{\uo}{n}{\uzp}}{\drho{\uo}{n}{\uz}}
    \;\leq\; 
    \abs*{\frac{\van{\uzp}}{\van{\uz}}}^2
    \exp\fbr*{m \usekd{pgcr} \thetagap^{-1} \X_n }\;,
    \]
    where $\usekd{pgcr}>0$ is a constant. 
\end{proposition}

\begin{proposition}[Uniform $L_1$ bound of inverse power sums: Proposition~8.6 in \cite{GP}]
    There exist constant $\usekd{pginil1}>0$
    such that for all $n\in\NN$ we have 
    $\Exp\tbr{\abs{\ginisumone(n)}}<\usekd{pginil1}$, 
    $\Exp\tbr{\abs{\ginisumtwo(n)}}<\usekd{pginil1}$,
    $\Exp\tbr{\ginisumthree(n)}<\usekd{pginil1}$. 
\end{proposition}

\begin{proposition}[Convergence of the inverse power sums: Proposition~8.8 in \cite{GP}]
    There exists random variables $\ginisumone$, $\ginisumtwo$, $\ginisumthree$,
    measurable with respect to $\GIFOUT$, 
    such that we have the following convergences in probability 
    $\ginisumone(n)\to\ginisumone$, 
    $\ginisumtwo(n)\to\ginisumtwo$, 
    $\ginisumthree(n)\to\ginisumthree$. 
\end{proposition}

\section{The Quasi-Gibbs property: Definition}\label{a:Osada}

Here we briefly recall the definition of Quasi-Gibbs property.
For more details we refer to \cite{Os13-1}.
Let $\qgC$ be a closed subset of $\mathbb{R}^d$ such that the origin $\vm{0}$ is contained in $\qgC$ and $\overline{\qgC^\circ}=\qgC$.
Denote by $\mathcal{M}$ the set of all discrete measures on $\qgC$ with measure of every compact set finite.
Consider the vague topology on $\mathcal{M}$ making it a Polish space.
For $b>0$ let $\qgB(b)$ be the open ball in $\qgC$ of radius $b$ around $\vm{0}$ i.e.,
\[
\qgB(b)\coloneqq\Set*{\vm{x}\in\qgC\given\lVert\vm{x}\rVert<b}\;.
\]
Let $\mathcal{M}_b^m$ be the set of measures in $\mathcal{M}$ which has $m$ points in $\mathtt{B}(b)$ i.e.,
\[
\mathcal{M}_b^m\coloneqq\Set*{\nu\in\mathcal{M}\given\nu\fbr*{\qgB(b)}=m}\;.
\]
For a subset $\mathtt{A}\subset\mathtt{C}$,
let $\mathbf{\pi}_{\mathtt{A}}:\mathcal{M}\to\mathcal{M}$
be the map given by 
\[
\mathbf{\pi}_{\mathtt{A}}(\nu) = \nu(\mathtt{A}\cap\cdot)\;.
\]
So this maps a measure in $\mathcal{M}$ to its restriction in $\mathtt{A}$.
For a bounded subset $\mathtt{A}\subset\mathtt{C}$ and Borel measurable functions 
$\Phi:\mathtt{C}\to\mathbb{R}\cup\left\{\infty\right\}$
and 
$\Psi:\mathtt{C}\times\mathtt{C}\to\mathbb{R}\cup\left\{\infty\right\}$ 
with 
$\Psi(x,y)=\Psi(y,x)$,
let $\mathcal{H}_{\mathtt{A}}^{\Phi,\Psi}:\mathcal{M}\to\mathcal{M}$ be the map given by
\[
\mathcal{H}_{\mathtt{A}}^{\Phi,\Psi}(\mathtt{x}) 
\coloneqq
\sum_{\mathclap{x_i\in\mathtt{A}}}\Phi(x_i) + 
\sum_{\mathclap{\substack{x_i,x_j\in\mathtt{A}\\i<j}}}
\Psi(x_i,x_j)
\]
where
$\mathtt{x} = \sum_{i} \delta_{x_i}\in\mathcal{M}$.

A probability measure $\mu$ on $\mathcal{M}$ is called a $(\Phi,\Psi)$-quasi Gibbs measure if there exists an increasing sequence $\left(b_r\right)_{r=1}^{\infty}$ of natural numbers and a set of measures $\{\mu_{r,k}^m : r\geq 1, k\geq 1\}$ on $\mathcal{M}$ such that the following hold: 
\begin{enumerate}[(i),font=\normalfont\bfseries,topsep=0pt]
\item For each $r,m\in\mathbb{N}$, $\left(\mu_{r,k}^m\right)_{k=1}^\infty$ and 
$\mu_r^m:=\mu(\cdot\cap\mathcal{M}_{b_r}^m)$
satisfy
$\mu_{r,k}^m \leq \mu_{r,k+1}^m$ for all $k$ and
$\lim_{k\to\infty}\mu_{r,k}^m = \mu_r^m$ weakly.

\item For all $r,m,k\in\mathbb{N}$ and for $\mu^m_{r,k}$-a.e.\ $\nu\in\mathcal{M}$,
\begin{align*}
& C^{-1}\,
\exp\fbr*{-\mathcal{H}_{\qgB(b_r)}^{\Phi,\Psi}(\mathtt{x})}\, 
\indicatorone\tbr[\big]{\mathtt{x}\in\mathcal{M}_{b_r}^m}\,
\Poisson_{\qgC}(\deri\mathtt{x}\,)\\
\leq{} &  
\mu_{r,k,\nu}^m (\deri\mathtt{x}\,)\\
\leq{} &  
C\,
\exp\fbr*{-\mathcal{H}_{\qgB(b_r)}^{\Phi,\Psi}(\mathtt{x})}\, 
\indicatorone\tbr[\big]{\mathtt{x}\in\mathcal{M}_{b_r}^m}\,
\Poisson_{\qgC}(\deri\mathtt{x}\,)\;.
\end{align*}
Here $\Poisson_{\qgC}$ is the Poisson point process whose intensity measure is the Lebesgue measure on $\qgC$,  
$\mu_{r,k,\nu}^m$ is the conditional distribution of $\pi_{\qgB(b_r)}(\nu)$ given $\pi_{\qgB(b_r)^\c}(\nu)$ i.e.,
\[
\mu_{r,k,\nu}^m (\deri\mathtt{x}\,) := 
\mu_{r,k}^m\fbr*{\pi_{\mathtt{B}(b_r)}(\nu)\in\deri\mathtt{x}\,|\,\pi_{\mathtt{B}(b_r)^c}(\nu)}\;,
\]
and $C$ is a constant depending on $r$, $m$, $k$, and $\pi_{\qgB(b_r)^c}(\nu)$.
\end{enumerate}

\section{Glossary of Notations}
Here we list some commonly used symbols with hyperlinks to their definitions. 
\renewcommand*{\arraystretch}{1.1}
\begin{longtable}{ll}
\ensuremath{\prsp} & Section~\vref{sec:technique}.\\ 
\ensuremath{\cS{\cdot}} & Section~\vref{sec:technique}.\\ 
\ensuremath{\Borel{\cdot}} & Section~\vref{sec:technique}.\\ 
\ensuremath{\bfD} & Section~\vref{sec:technique}.\\ 
\ensuremath{\bbX} & Definition~\vref{def:gengibbs}.\\ 
\ensuremath{\bbX_{\inn}} & Definition~\vref{def:gengibbs}.\\ 
\ensuremath{\bbX_{\out}} & Definition~\vref{def:gengibbs}.\\ 
\ensuremath{\BasisInsidempts} & Notation~\vref{n:insidetopology}.\\
\ensuremath{\FUBasisInsidempts} & Notation~\vref{n:insidetopology}.\\
\ensuremath{\BorelInsidempts} & Notation~\vref{n:insidetopology}\\
\ensuremath{\BasisOutside} & Notation~\vref{n:outsidetopology}.\\
\ensuremath{\FUBasisOutside} & Notation~\vref{n:outsidetopology}.\\
\ensuremath{\BorelOutside} & Notation~\vref{n:outsidetopology}.\\
\ensuremath{\asymp} & Notation~\vref{n:asymp}.\\
\ensuremath{\bbX_{\infty}} & Notation~\vref{n:genpp}.\\ 
\ensuremath{\bbX_{\infty,\inn}} & Notation~\vref{n:genpp}.\\ 
\ensuremath{\bbX_{\infty,\out}} & Notation~\vref{n:genpp}.\\ 
\ensuremath{\bbX_{n}} & Notation~\vref{n:genpp}.\\ 
\ensuremath{\bbX_{n,\inn}} & Notation~\vref{n:genpp}.\\ 
\ensuremath{\bbX_{n,\out}} & Notation~\vref{n:genpp}.\\ 
\ensuremath{\om} & Definition~\vref{d:om:ommn}.\\
\ensuremath{\omn} & Definition~\vref{d:om:ommn}.\\
\ensuremath{\GIF} & Section~\vref{ss:ginisetup}.\\
\ensuremath{\GIFIN} & Section~\vref{ss:ginisetup}.\\
\ensuremath{\GIFOUT} & Section~\vref{ss:ginisetup}.\\
\ensuremath{\van{\cdot}} & Notation~\vref{n:vandermonde}.\\
\ensuremath{\vancross{\cdot}{\cdot}} & Notation~\vref{n:vandermonde}.\\
\ensuremath{\numberofpoints} &  Section~\vref{ss:ginisetup} in the context of Ginibre ensemble;\\
& \indent Section~\vref{ss:agafsetup} in the context of $\alpha$-GAF.\\
\ensuremath{\frho{\cdot}{\cdot}} &  Section~\vref{ss:ginisetup} in the context of Ginibre ensemble;\\
& \indent Section~\vref{ss:agafsetup} in the context of $\alpha$-GAF.\\
\ensuremath{\Sigma_{\usm}} & Notation~\vref{n:manifold}.\\
\ensuremath{\falphagaf} & Section~\vref{ss:agafsetup}.\\
\ensuremath{\alphagaf} & Section~\vref{ss:agafsetup}.\\
\ensuremath{\alphagafin} & Section~\vref{ss:agafsetup}.\\
\ensuremath{\alphagafout} & Section~\vref{ss:agafsetup}.\\
\ensuremath{\rigidity} & Section~\vref{ss:agafsetup}.\\
\ensuremath{\constraintexternal} & Section~\vref{ss:agafsetup}.\\
\ensuremath{\Ginin} & Section~\vref{sec:ginibre}.\\
\ensuremath{\Gininin} & Section~\vref{ss:ginilimiting}.\\
\ensuremath{\Gininout} & Section~\vref{ss:ginilimiting}.\\
\ensuremath{\aspc{\cdot}} & Defined in Notation~\vref{n:projection}.\\
\ensuremath{\Uptheta^{m}_{n}[\thetagap]} & Definition~\vref{d:gap:gini} in the context of Ginibre ensemble;\\
& \indent Definition~\vref{d:gaf:gap} in the context of the $\alpha$-GAF.\\
\ensuremath{\X_n} & Notation~\vref{n:ginisum} in the context of the Ginibre ensemble;\\
& \indent Notation~\vref{n:gafsum} in the context of $\alpha$-GAF.\\
\ensuremath{\drho{\cdot}{\cdot}{\cdot}} & Section~\vref{ss:GiniStep2} in the context of the Ginibre ensemble;\\ 
& \indent Section~\vref{ss:agaf:ratio} in the context of the $\alpha$-GAF.\\
\ensuremath{\falphagafn} & Section~\vref{ss:agaf:ratio}.\\
\ensuremath{\alphagafn} & Section~\vref{ss:agaf:ratio}.\\
\ensuremath{\vecsym{\cdot}{\cdot}{\cdot}} & Defined in Notation~\vref{n:vecsym}.\\
\ensuremath{\vecsymb{\cdot}{\cdot}{\cdot}} & Defined in Notation~\vref{n:vecsym}.\\
\ensuremath{\fD{\cdot}{\cdot}} & Defined in equation~\vref{eq:412}.\\
\ensuremath{\inv{\cdot}{\cdot}{\cdot}{\cdot}} & Section~\vref{ss:EIPZ}.\\
\ensuremath{\invabs{\cdot}{\cdot}{\cdot}{\cdot}} & Section~\vref{ss:EIPZ}.\\
\ensuremath{\moment} & Defined in Proposition~\vref{p:inv:4}.\\
\ensuremath{\fDij{\cdot}{\cdot}} & Defined in equation~\vref{eq:dijdef}.\\
\ensuremath{\fDhat{\cdot}{\cdot}} & Defined in equation~\vref{eq:Dhatdef}.\\
\ensuremath{\ffa{\cdot}{\cdot}} & Defined in Notation~\vref{n:factorial}.\\
\ensuremath{\QFIop} & Equation~\vref{eq:fijdef}.\\
\ensuremath{\QFZop} & Equation~\vref{eq:fzdef}.\\
\ensuremath{\InProd} & Notation~\vref{n:product}.\\
\ensuremath{\InProdNd} & Notation~\vref{n:product}.\\
\ensuremath{\OMG{\thetagap}{\infty}{1:1}{M}{\delta}} & Definition~\vref{d:i:1:1}.\\
\ensuremath{\OMG{\thetagap}{\nd}{1:1}{M}{\delta}} & Definition~\vref{d:n:1:1}.\\
\ensuremath{\OMG{\thetagap}{\nd}{2:1}{M}{\delta}} & Definition~\vref{d:n:2:1}.\\
\ensuremath{\OMG{\thetagap}{\nd}{2:2}{M}{\delta}} & Proposition~\vref{p:2:2}.\\
\ensuremath{\OMG{\thetagap}{\nd}{2:3}{M}{\delta}} & Proposition~\vref{p:2:3}.\\ 
\ensuremath{\OMG{\thetagap}{\nd}{2:4}{M}{\delta}} & Proposition~\vref{p:2:4}.\\ 
\ensuremath{\constraintinternal} & Notation~\vref{n:con:int}.\\
\ensuremath{\constraintinternalpc} & Notation~\vref{n:con:int}.\\
\ensuremath{\OMG{\thetagap}{\nd}{3:1}{M}{\delta}} & Definition~\vref{d:n:3:1}.\\
\ensuremath{\OMG{\thetagap}{\infty}{3:1}{M}{\delta}} & Definition~\vref{d:i:3:1}.\\
\ensuremath{\OMG{\thetagap}{\infty}{3:2}{M}{\delta}} & Definition~\vref{d:i:3:2}.\\
\end{longtable}
\end{appendix}
\end{document}